\documentclass[11pt]{article}
\usepackage{amscd,amssymb,amsmath,verbatim,amsthm}
\usepackage{amsfonts,latexsym}
\usepackage{eucal}
\usepackage{geometry}
\usepackage{nicefrac}
\usepackage{caption}
\usepackage[]{subcaption}
\usepackage{tikz}

\usepackage{color}
\definecolor{hanblue}{rgb}{0.27, 0.42, 0.81}
\definecolor{red}{rgb}{1.0, 0.0, 0.0}
\usepackage[colorlinks, citecolor=hanblue,linkcolor=red]{hyperref}

\newtheorem{theorem}{Theorem}
\newtheorem{corollary}[theorem]{Corollary}
\newtheorem{definition}[theorem]{Definition}

\newtheorem{lemma}[theorem]{Lemma}
\newtheorem{proposition}[theorem]{Proposition}
\newtheorem{remark}[theorem]{Remark}

\numberwithin{equation}{section}
\numberwithin{theorem}{section}

\renewcommand{\tilde}{\widetilde}
\renewcommand{\bar}{\overline}

\renewcommand{\hat}[1]{\widehat{#1}}

\newcommand\eps\varepsilon

\renewcommand\det{\operatorname{det}}

\newcommand\bff{\operatorname{bf}}
\newcommand\ff{\operatorname{ff}}

\newcommand{\lf}{\operatorname{lf}}
\newcommand{\rf}{\operatorname{rf}}

\newcommand\paperintro%
        {%
         }
\newcommand\paperbody%
        {%
         }

\newcommand\bbM{\mathbb{M}}

\newcommand\cC{\mathcal{C}}

\newcommand{\res}{\mathop{\hbox{\vrule height 7pt width .5pt depth 0pt
\vrule height .5pt width 6pt depth 0pt}}\nolimits}

\DeclareMathAlphabet{\mathpzc}{OT1}{pzc}{m}{it}

\newcommand{\NN}{\mathbb N}
\newcommand{\RR}{\mathbb R}

\newcommand{\del}{\partial}

\newcommand{\calC}{{\mathcal C}}

\newcommand{\calE}{{\mathcal E}}
\newcommand{\calF}{{\mathcal F}}

\newcommand{\calL}{{\mathcal L}}
\newcommand{\calM}{{\mathcal M}}
\newcommand{\calN}{{\mathcal N}}
\newcommand{\calO}{{\mathcal O}}

\newcommand{\calT}{{\mathcal T}}

\newcommand{\calV}{{\mathcal V}}

\newcommand{\calX}{{\mathcal X}}
\newcommand{\calY}{{\mathcal Y}}
\newcommand{\calZ}{{\mathcal Z}}

\newcommand{\ssim}{{\mathrm{ss}}}
\newcommand{\inte}{{\mathrm{int}}}
\newcommand{\ori}{{\mathrm{or}}}
\newcommand{\new}{{\mathrm{new}}}
\newcommand{\ext}{{\mathrm{ext}}}

\def\bpf{\begin{proof}}
\def\epf{\end{proof}}
\def\beq{\begin{equation}}
\def\eeq{\end{equation}}

\definecolor{airforceblue}{rgb}{0.36, 0.54, 0.66}

\date{}

\begin{document}

\title{Short-time existence for the network flow}
\author{Jorge Lira \\ Federal University of Ceará \and Rafe Mazzeo \\ Stanford University \and   
Alessandra Pluda \\ University of Pisa \and Mariel Saez \\ Pontificia Universidad Cat\'olica de Chile} 

\maketitle

\begin{abstract}
This paper contains a new proof of the short-time existence for the flow by curvature of a network of curves in the plane. 
Appearing initially in metallurgy and as a model for the evolution of grain boundaries, this flow was later treated by Brakke \cite{Br}
using varifold methods.  There is good reason to treat this problem by a direct PDE approach, but doing so
requires one to deal with the singular nature of the PDE at the vertices of the network. This was handled in cases of 
increasing generality by Bronsard-Reitich \cite{BrRe}, Mantegazza-Novaga-Tortorelli \cite{MNT} and eventually, 
in the most general case of irregular networks by Ilmanen-Neves-Schulze \cite{INS}.  Although the present paper proves 
a result similar to the one in \cite{INS}, the method here provides substantially more detailed information about
how an irregular network `resolves' into a regular one. Either approach relies on the existence of self-similar expanding 
solutions found in \cite{MS}.  As a precursor to and illustration of the main theorem, we also prove an unexpected regularity result for the 
mixed Cauchy-Dirichlet boundary problem for the linear heat equation on a manifold with boundary.
\end{abstract}

\section{Introduction}
This paper contains a new proof of short-time existence for the curve-shortening flow for networks of curves 
in the plane, or more generally in any Riemannian surface.  This problem has been treated before, and the 
history is recalled below, but our approach draws out some important features and precise information
not accessible by previous methods.  The inspiration for this method draws on the methods of geometric
microlocal analysis, and in particular the application of those methods to the study of geometric
flows on spaces with conic singularities, cf.\ \cite{MRS}. Beyond the geometric naturality of this proof, it 
should be well adapted for proving analogous results in higher dimensions, e.g.\ for the flow by mean 
curvature of `foams' of surfaces in $\RR^3$. 

The \emph{Mean Curvature Flow} (MCF) is surely one of the best studied geometric evolution problems; it requires
a submanifold to evolve with normal velocity equal to its mean curvature. It is the gradient flow of the length/area/volume 
functional, and hence, the length of a curve evolving by this flow decreases in time; this explains the name curve-shortening flow. 
The normal velocity at any $(p,t)$ is the curvature of the curve at that point of space and time: 
\begin{equation*}
v^\perp=\vec{\kappa}\,.
\end{equation*}
The complete analysis of the curve-shortening flow for embedded closed curves, due to Gage-Hamilton \cite{GH} and
Grayson \cite{Gr}, was the one of earliest successes in the theory of geometric flows.   However, some years 
before, motivated in part by the desire to understand dynamics of grain boundaries, Brakke \cite{Br} defined a weak
form of this curvature flow in the class of varifolds, which thus provides one avenue for studying the evolution of 
certain singular geometric objects.  This weak formulation yields global existence of solutions, but on the other hand, does 
not allow one to control the geometry of these evolving objects.
This has motivated much further study.   We mention the work of Kim and Tonegawa~\cite{KT}, who showed 
that one can obtain a nontrivial mean curvature flow even when the initial datum is singular, cf.\ also the recent
monograph~\cite{To}.  However, it is worth emphasizing that questions concerning uniqueness or non-uniqueness 
of solutions are mostly inaccessible in this measure theoretic context.

It is of interest to recast Brakke's results if one is interested in objects which are nearly smooth;
in one dimension, this leads to the study of networks of curves.  By definition, a network of curves $\Gamma$ 
in some domain $\Omega \subset \RR^2$ is a collection of smooth embedded curves $\{\gamma^{(j)}(x)\}$ with the following
properties. First, the interiors of these curves are disjoint. Each end of any $\gamma^{(j)}$ is either an `external' or `internal'
vertex.  The external vertices are disjoint and lie on $\del \Omega$ and are regarded as termini of the network. We impose 
some boundary conditions at these, e.g. they might be fixed or evolve remaining normal to the boundary, etc. 
At each interior vertex, however, three or more of the $\gamma^{(j)}$ meet non-tangentially.  We always
assume that there is at least one interior vertex.  For simplicity we always assume that $\Omega \subset \RR^2$
and Dirichlet (fixed) boundary conditions are imposed at the internal vertices. However, everything we say here
can be easily generalized if $\Omega$ is either a closed Riemannian surface, or a surface with boundary, or if other 
boundary conditions are imposed, or even more generally, if $\Gamma$ lies in $\RR^2$,  the network consists of only finitely
many arcs, and the `external' arcs are asymptotic to  rays at infinity (see \S~\ref{networks} below for a bit more on this). 

As indicated above, we wish to study this problem in a strong sense, using the methods of PDE rather than geometric measure theory.
Thus with respect to an arbitrary parametrization of each $\gamma^{(j)}$, consider the equations
\begin{equation*}
\langle \del_t \gamma^{(j)} , \nu^{(j)}\rangle = \kappa^{(j)}\, ;
\end{equation*}
here $\nu^{(j)}$ is the unit normal vector to $\gamma^{(j)}$ and the curvature vector is $\vec{\kappa}^{(j)}=\kappa^{(j)}\nu^{(j)}$.

Consider first the situation where the initial configuration $\Gamma_0$ is {\it regular}, i.e., each interior vertex is 
trivalent, with the three curves meeting at each such vertex in equal angles of $2\pi/3$ (this is
called the {\it Herring condition}).   Regular vertices are one feature of stable critical points of the length functional in the 
class of networks.  If the curves meeting at an interior vertex are labelled $\gamma^{(i_\ell)}$, $\ell=1,2,3$, then
for every $t > 0$, 
\begin{equation*}
\begin{aligned}
& \gamma^{(i_1)}(t,0) = \gamma^{(i_2)}(t,0) = \gamma^{(i_3)}(t,0)\,,  \quad \mbox{and}\\
& \tau^{(i_1)}(t,0) +\tau^{(i_2)}(t,0) + \tau^{(i_3)}(t,0)  = 0,
\end{aligned}
\end{equation*}
where $\tau^{(j)}$ is the unit tangent vector to $\gamma^{(j)}(t,0)$. 

Short-time existence of this flow for regular networks was obtained in the papers of Kinderlehrer and Liu~\cite{KL}, 
Bronsard and Reitich~\cite{BrRe}, and Mantegazza, Novaga and Tortorelli~\cite{MNT}.  In this last paper, 
as well as in the more recent paper by Magni, Mantegazza and Novaga~\cite{MMN},  certain long-time existence results
are also proved. We refer also to the survey by Mantegazza et al.~\cite{Man}. 

Our main interest here is the situation where $\Gamma_0$ is not regular. In other words, we are interested in the
short-time existence of this flow when the initial network has interior vertices which are either trivalent but
for which the Herring conditions are not satisfied, or else $k$-valent for some $k > 3$.

There are several motivations for studying this problem. The most basic one is the inherent interest in enlarging the 
class of admissible initial data. However, another is that the class of regular networks is not preserved by the flow. 
An initial regular network can evolve into one which at some later time is not regular. For example, two triple junctions 
might coalesce, or an enclosed region bounded by a loop of curves in the network might collapse to a point.
The best scenario is when the curvature of the constituent curves remains bounded in this singularity formation
(as happens when two triple points collide).  At the very least, we would like to know whether this limiting
network can evolve past this singular time. The short time existence result for irregular networks `restarts'
the evolution.

Short time existence for irregular networks was established in the paper of Ilmanen, Neves and Schulze~\cite{INS}.
They prove that if $\Gamma_0=\{\gamma_0^{(j)}\}$ is an irregular initial network, then there exists a time $T>0$ 
and a solution of the flow $\Gamma(t)=\{\gamma_t^{(j)}\}_{t\in [0,T)}$ such that  
$\Gamma(t)$ converges to $\Gamma_0$ in the sense of varifolds. (As we explain later, $\Gamma(t)$
may have a larger number of arcs than $\Gamma_0$ itself.)  Although this convergence is already slightly 
stronger than requiring that $\mathcal{H}^1\res \Gamma_t \to \mathcal{H}^1\res \Gamma_0$ as $t \to 0$ 
in the sense of measures, it is desirable to have an even more detailed understanding of the geometry 
of this `singularity resolution'.

The key point in all of this has relatively little to do with the regularity of
the initial curves, but is more closely related to their geometry. We must, in particular, show
precisely how a single multi-point gives birth to a cluster of triple junctions.   This already
appears in \cite{INS}, but our approach provides a more detailed understanding. 
Our main result is the following
 \begin{theorem}\label{main}
Let $\Gamma_0$ be an initial network where all component arcs are of class $\calC^2$ and
at least one interior vertex is irregular (even in the mildest sense that the curves which meet 
at the vertex do not satisfy all infinitely many of the matching conditions there). 
Then there exists a time $T>0$ and an evolving family of regular networks $\Gamma(t)$, $0 < t < T$,
with the property that $\Gamma(t) \to \Gamma_0$ in a strong sense to be described below.

The set of possible flowouts is classified by the collection of all (appropriate) expanding soliton 
solutions of the flow at each interior vertex. 
\end{theorem}

The method developed here was initially inspirred by the paper of the second author, Rubinstein and 
Sesum~\cite{MRS}, which studies the Ricci flow on surfaces with isolated conical singularities. The  main
results there are a set of sharp estimates for linear heat equations on spaces with conic singularities,
from which the short-time existence of the flow is readily proved. (That paper also considers a number
of long-time and limiting results.)   At first we had planned to set things up to be able to invoke
those estimates, but in fact we ultimately decided that it is more transparent to approach the present problem
by an explicit construction of a very accurate approximate solution to the problem,
followed by some more standard-looking analytic estimates which are used to solve away the (rapidly
decaying) error term. 

To be more precise, suppose that $p$ is an irregular vertex of our initial network $\Gamma_0$ which is
$k$-valent for some $k > 3$ (the construction near vertices of valence $3$ is slightly simpler). Assume
for simplicity $p$ is placed at the origin.  Enumerate the curves meeting at $p$ by $\gamma^{(1)}, \ldots, \gamma^{(k)}$, 
sequentially going counterclockwise.  Denote by $\tau_j$ the unit tangent vector to $\gamma_j$ at $p$, and let
$\ell_j = \RR^+ \tau_j$ be the ray in $\RR^2$ in the direction $\tau_j$.  The incoming edges at $p$ 
thus determine a fan of rays in $\RR^2$ emanating from the origin.  As proved in \cite{MS}, there exists at
least one (and in fact, multiple when $k > 3$) expanding soliton solutions for the curve-shortening flow 
which have this fan as initial condition. (At least for generic fans, but conjecturally for all fans, there are only finitely many 
such solitons.) Choose one of these solutions, and denote it by $S_p$.  Having made this choice for every irregular 
interior vertex,  we prove that there exists a unique solution $\Gamma(t)$ of the network flow, defined on some interval 
$0 < t < T$, whose combinatorics are the same as if we were to replace a small ball around each irregular vertex $p$ 
with the corresponding choice of soliton $S_p$.   This solution converges as $t \to 0$ to $\Gamma_0$, but
this  convergence happens in a slightly odd fashion since $\Gamma(t)$ may have more  curves than $\Gamma_0$,
so certain curves in $\Gamma(t)$ converge to curves in $\Gamma_0$  (as smoothly as the regularity of those curves 
in $\Gamma_0$ allow), while the remaining curves in $\Gamma(t)$ shrink in size and `disappear into' these
irregular vertices as $t \searrow 0$.  There is a sense in which the entire convergence can be regarded
as $\calC^\infty$  if the initial network $\Gamma_0$ is smooth.  
At each irregular vertex $p$, the  soliton $S_p$ is the tangent flow at $(p,0)$.  This replacement  of
a neighborhood of an irregular vertex $p$ with a soliton $S_p$ appears also in \cite{INS}, but their
analysis only describes the convergence as $t \searrow 0$ in a  much weaker sense. 

The key to this construction (as well as the one in  \cite{INS}) is the existence of expanding solitons.  For 
trivalent vertices, their  existence and uniqueness was proved by Schn\"urer and Schulze~\cite{SS}. For
more general vertices, their existence and classification was obtained by the second and fourth authors 
here~\cite{MS}. One striking feature and outcome of this is that if any interior vertex has valence greater
than $3$, then the flow is ill-posed: there exists more than one solution emanating from the given 
initial configuration.  The chosen soliton at each $p$ emerges in a very natural way, as we shall explain.
This illustrates yet again that nonuniqueness is endemic in the realm of flows on spaces with singular
initial data. 

After describing the general features of the problem in \S 2, we introduce in \S 3 our key tool: the blowups
of the domain and range spaces. These blowups (which are a central tool in geometric microlocal analysis)
provide sufficient `room' to explain the $t \searrow 0$ asymptotics.  The main work is the construction
of an approximate solution, i.e., a family of networks $\widehat{\Gamma}(t)$ which converges to $\Gamma_0$
and which satisfy the flow equations up to an error which vanishes to all orders at $t=0$.  This is
rather delicate, and as a warm-up to it, we carry out this construction for a similar but quite classical problem in \S 4,
namely the solution of the linear heat equation on an interval, where the prescribed initial and boundary conditions
do not satisfy the compatibility conditions which would allow the solution to be smooth down to $t=0$.
The construction of the approximate solution for the nonlinear problem appears in \S 5.  There is still
a need to solve away the rapidly vanishing error term. This could be done using the heat space theory as in \cite{MRS},
but we have chosen to do this using a classical approach based on a priori estimates.  Due to some
unusual features of the equation, this takes a bit of work, and is done in \S 6.  The short-time existence
is then proved quickly in \S 7, and \S 8 contains a number of remarks concerning extensions of this theorem. 

This plan for this paper, and the realization that a proof of this nature should be possible, was envisioned
over a decade ago, motivated by the then-recent work \cite{MS} but also by conversations 
between the second author and Felix Schulze during a sequence of quarter-long visits by Schulze to Stanford. 
Those conversations were perhaps precursors to \cite{INS}, but certainly also to the present paper. The details 
have been slow in coming, not for any particularly good reason, 
but in any event, gradually crystalized during a number of visits of varying lengths to Stanford by the first, third and 
fourth authors in the intervening years, and later, as the project came together, in other visits by the second author 
to Rome and Pisa. 

The first author has been supported by  CNPq grants 239573/2013-7 and 307410/2018-8. The second author has been supported by an NSF grant throughout, most recently by DMS-1608223.
The third author has been supported by 
the Deutsche Forschungsgemeinschaft (DFG) via the GRK 1692 ``Curvature, Cycles, and Cohomology'' and by INdAM-GNAMPA project 2019 ``Geometric problems for singular structures''.
The fourth author has been supported by a Fondecyt grant throughout, most recently by Fondecyt Regular 1190388. 

\section{The evolution equation for networks} 
\subsection{Networks}\label{networks}
We now develop language to state our precise results. 

\begin{definition}\label{network}
A planar \emph{network} $\Gamma$ is a finite union of embedded regular curves $\{\gamma^{(j)}\}_{j=1}^N$, 
each of class $\calC^2$, which meet in groups of at least $3$ at each interior vertex and which have some
fixed ends (or else $\Gamma$ is noncompact and each end is a curve asymptotic to a half-line at infinity).

We  also make the convention that we consider only networks where the curves meet nontangentially at
each interior vertex.
\end{definition}

Since the focus in this paper is the behavior near interior vertices, we assume that $\Gamma$ is compact
and lying in a smoothly bounded domain $\Omega \subset \RR^2$.   Were we to be studying long-term
behavior of the flow, it would be highly advantageous to assume that $\Omega$ is convex, but this
is irrelevant for our purposes. The methods here apply
equally well to noncompact networks, with noncompact edges asymptotic to straight lines, or also 
to networks lying in a Riemannian surface $(\Omega,g)$, with or without boundary. 

\smallskip

A single curve (closed or not) is a particularly trivial network without vertices.  Networks with only two curves
meeting in a vertex have been studied from the point of view of curves with cornerrs or cusps, cf.~\cite{angen3,angen1,angen2,gage,GH,Gr}.

\begin{definition}\label{reg-non-reg-networks}
A network $\Gamma$  is said to be \emph{regular} if the only vertices are \emph{triple} junctions where the unit tangent vectors 
of the curves meeting there form angles of $2\pi/3$  (this is \emph{the Herring condition}), and in addition the compatibility
conditions on higher derivatives are satisfied at each vertex, cf.\ the beginning of \S 3. (If the network is only $\calC^2$,
these compatibility conditions state only that the curvatures of the three curves are equal.)  Otherwise  the network is called irregular.
\end{definition}

The regular networks which are also trees have some particularly nice properties, as we shall mention later. 

\subsection{Evolution equation}\label{sectionevolutioneq}

Given a network composed of $m$ curves we define its \emph{length} $L$ as the sum of the length of each curve.
The curve-shortening flow is formally the $L^2$--gradient flow of the length. We aim to define the motion using 
classical PDE theory, so it is helpful to parametrize each curve of the network rather than to look at the network as a set.
To avoid confusion we denote by $\Gamma=\{\gamma^{(j)}\}_{j=1}^m$ a \emph{parametrized} network if each
$\gamma^{(j)}$ is a $\calC^2$ map from $[0,1]$ to $\Omega$ with derivative everywhere nonvanishing, and 
by $\mathcal{N}$ the corresponding set: $\mathcal{N}=\bigcup_{j=1}^m\gamma^{(j)}([0,1])$.
The notion of regular network is geometric: if a parametrized network $\Gamma$ is regular, 
than any smooth reparametrization of $\Gamma$ is still a regular network.

\medskip

Consider first a family of smooth parametrized curves $\gamma(t,x)$ where $x$ is an arbitrary monotone parametrization
(i.e., $\del_x \gamma \neq 0$). As already noted, the basic curve-shortening flow equation is 
\begin{equation}
\label{evolution}
\langle \del_t \gamma, \nu \rangle  = \kappa\,, \qquad \gamma(0, \cdot) = \gamma_0(\cdot)\,,
\end{equation}
where $\nu$ is the unit normal to $\gamma$ at any $(t,x)$, or equivalently, denoting the unit tangent vector to $\gamma$ by $\tau$,
\begin{equation}
\del_t \gamma = \kappa \nu + \zeta \tau\,,
\label{evolution2}
\end{equation}
for some function $\zeta$. Here $\kappa$ is the oriented curvature and equations~\eqref{evolution} and \eqref{evolution2} can
also be written in terms of the curvature vector as $(\del_t\gamma)^{\perp}=\vec{\kappa}$ where $(\partial_t\gamma)^{\perp}$ 
denotes the normal component of the velocity.

We next generalize  the definition of curve shortening flow to a family of time-dependent networks $\Gamma(t) = \{\gamma^{(j)}(t,x)\}$.
We ask the network to be regular for $t > 0$, at least for some short time interval: if $\gamma^{(i_\ell)}(0,{\color{green} x})$, $\ell=1,2,3$, intersect at
an interior vertex, then for every $t > 0$
\begin{equation}
\begin{aligned}
& \gamma^{(i_1)}(t,0) = \gamma^{(i_2)}(t,0) = \gamma^{(i_3)}(t,0)\,, \\
& \tau^{(i_1)}(t,0) +\tau^{(i_2)}(t,0) + \tau^{(i_3)}(t,0)  = 0\,,
\end{aligned}
\label{mbc}
\end{equation}
where $\tau^{(j)}$ is the unit tangent vector to $\gamma^{(j)}$. These matching boundary conditions~\eqref{mbc}
are called the Herring conditions. 

We now define the system of PDEs driving the motion by curvature of networks; it remains to discuss the freedom in
the choice of the tangential component $\zeta$ in~\eqref{evolution2}. At first sight, the choice of $\zeta$ in~\eqref{evolution2} 
does not seem important since it only corresponds to a reparametrization. However, observe that we cannot take $\zeta^{(j)} \equiv 0$.
Indeed, following~\cite[\S 2]{Man}, if the solution is sufficiently smooth in space and time, then the values of $\zeta^{(j)}(t,0)$
are determined entirely by the values of $\kappa^{(j)}(t,0)$ and vice versa.  Different choices of $\zeta^{(j)}$ taking these required values 
at the vertices then do correspond to reparametrizations. However, suppose $\zeta^{(j)}\equiv 0$, and let $P(t)$ denote the motion of any
interior vertex. Then for each $j$, $P'(t) = \del_t \gamma^{(j)}(t,0) = \kappa^{(j)}(t,0) \nu^{(j)}(t,0)$; since the $\nu^{(j)}$ are independent,
we deduce that $P'(t) \equiv 0$, i.e., the vertex remains fixed. However, imposing \eqref{mbc} and also requiring this vertex to be constant 
leads to an overdetermined system of PDEs.
It is also the case that, due to the `gauge invariance', i.e., the possibility of reparametrization, \eqref{evolution2} with $\zeta \equiv 0$ is not 
strictly parabolic, see~\cite{GH}. 

In any case, there are various natural choices for $\zeta$, but the one we use here is 
$\zeta = |\del_x \gamma|^{-3} \langle \del_x^2 \gamma, \del_x \gamma \rangle$
so that the flow equation becomes 
\begin{equation*}
\del_t \gamma = \frac{\del_x^2 \gamma}{|\del_x\gamma|^2}\,.
\end{equation*}

\begin{definition}[Parametrized network flow]\label{solutions}
Let $\Gamma_0=\{\gamma_0^{(j)}\}_{j=1}^n$ be a (possibly irregular) initial network consisting 
of $n$ regular curves $\gamma^{(j)}\in \calC^{2}([0,1],\overline{\Omega})$.  A one-parameter family
$\Gamma(t) = \{\gamma^{(j)}(t,\cdot) \}_{j=1}^m$, $0 < t < T$ and $m \geq n$, of regular networks is 
a solution to the \emph{network flow} with initial condition $\Gamma_0$ and with 
Dirichlet boundary conditions at $\del \Omega$ if the following is true:  for each $t \in (0,T)$, 
$\gamma^{(j)}(t,\cdot)$ is a smoothly parametrized curve in $\Omega$ with $\del_x \gamma^{(j)} \neq 0$ for
all $(t,x) \neq ((0,0)$; each network $\Gamma(t)$ is regular; each $\gamma^{(j)}$ evolves according to the equation
\begin{equation}\label{evolution3}
\del_t \gamma^{(j)}(t,x) =\frac{\del_x^2 \gamma^{(j)}(t,x)}{|\del_x \gamma^{(j)}|^{2}};
\end{equation}
at each triple junction where three curves meet, the matching conditions~\eqref{mbc} are satisfied,
while all exterior vertices, where any one of the $\gamma^{(j)}(t, 1) \in \del \Omega$, remain fixed as $t$ varies; 
finally, we require that  $\Gamma(t)$ converges uniformly in Hausdorff distance as a family of sets to $\Gamma_0$ 
as $t \searrow 0$.  
\end{definition}

At this stage the convergence of $\Gamma(t)$ to $\Gamma_0$ as $t \searrow 0$
can be taken as locally uniform
convergence of sets, but we shall give a more precise description of this convergence later. 

\begin{remark}
As checked in \S 6, the Herring boundary conditions satisfy the Lopatinski-Shapiro 
conditions for the linearization of the network flow. This, of course, was known much earlier,
cf.\ \cite{BrRe}, but  here is quite simple. This implies that each $\gamma^{(j)}$ is 
$\cC^\infty$ for $x \in [0,1]$ and $t > 0$. 
\end{remark}

\subsection{Irregular initial data} 
Now let us turn to the main question in this paper.   Suppose that $\Gamma_0$ is an irregular
network. It has a collection of exterior vertices $q_1, \ldots, q_r$ on $\del \Omega$ 
as well as some number of interior vertices $p_1, \ldots, p_s$.  Associated to each $p_i$
is a collection of indices $I_i \subset \{1, \ldots, n\}$ such that each of the curves $\gamma^{(j)}_0$,
$j \in I_i$, terminate at $p_i$ at either $x=0$ or $x=1$.  We set $k_i = |I_i|$ arcs; this is the valence of that vertex.
Similarly, each end of any $\gamma^{(j)}_0$ is either an interior or exterior vertex. We assume that
if an exterior vertex $q_i$ is one end of the curve $\gamma^{(j)}_0$, then $\gamma^{(j)}_0(1) = q_i$. 

For each internal vertex $p_i$, there is a set of rays emanating from $p_i$ determined by the tangents 
$\tau^{(j)} = \del_x \gamma^{(j)}_0$ of incoming arcs. This forms a  {\it fan} $\calF_i$.
In the next subsection we recall an old result by two of us \cite{MS} that shows that for each fan $\calF_i$ with at least $3$ 
rays there exists one (and, if $k_i > 3$, many) regular self-similar expanding solutions of the network flow  which 
converge to $\calF_i$ as $t \searrow 0$. Denote this set of expanding solitons by $\{\Gamma^{\ssim, (i)}_j\}_{j \in J}$. 
There is a unique such soliton for each combinatorial type of network. A priori there might be infinitely many, so we do not specify 
the enumeration precisely. Choosing one soliton for each $i$ we can 
define a regular network obtained by inserting a truncated version of $\Gamma^{\ssim, (i)}_j$ at each $p_i$. This determines a 
particular combinatorial type of regular network, and our eventual theorem shows that there exists an actual
solution of the network flow with this topology, at least for some short interval of time. For these fans, and also 
for the original network $\Gamma_0$,  the resolution appears instantaneously in the flow at time $0$; a number of 
new internal arcs are generated at each internal vertex.

For the truncation to a ball of large radius of each $\Gamma^{\ssim, (i)}_j$ there are $k_i$ external vertices. An elementary combinatorial 
argument proves that if this soliton network is a tree, then there are $k_i - 2$ interior vertices and $2k_i - 3$ edges, of which $k_i-3$ 
are interior edges.  The exterior edges are those which limit to the rays of the fan $\calF_i$. These enumerations show that, assuming
$\Gamma_0$ and each of the solitons are trees, then
\[
m = n + \sum_i (k_i - 3).
\]
The defect $m-n$ determines the number of new edges which appear instantaneously in the flowout.  In general, however,
if some of these networks are not trees, then $m-n$ may be much larger. 

In the following we denote by $\calV_0$ the set of all vertices; this set is the union $\calV_{0, \inte} \cup \calV_{0, \ext}$
of interior and exterior vertices.  The set of vertices of a flowout is denoted by $\calV$, and decomposes into
$\calV_{\inte, \ori} \cup \calV_{\inte, \new} \cup \calV_{\ext}$, where $\calV_{\ext}=\calV_{0, \ext}$ are the exterior vertices, 
$\calV_{\inte, \ori}$ is the set of vertices already present at $t=0$, and $\calV_{\inte, \new}$ is the set of new vertices that appear 
because of the  self-similar resolution. Similarly, the set $\calE_0$ of all edges of an initial network $\Gamma_0$
decomposes into $\calE_{0,\inte} \cup \calE_{0,\ext}$ of interior and exterior edges; for the set $\calE$ of edges
of a flowout, the decomposition becomes $\calE_{\ext} \cup \calE_{\inte, \ori} \cup \calE_{\inte, \new}$, where 
$\calE_{\inte, \ori}$ consists of the original edges which exist both at $t=0$ and $t > 0$, while $\calE_{\inte,\new}$ 
contains those new edges which arise as internal edges of the self-similar resolution of each internal vertex. 

\subsection{Expanding solitons}

It is now standard in any geometric flow that there is an intimate relationship between singularity formation in
the flow and the existence and nature of self-similar solutions to the flow.  Self-similar collapsing solutions
determine, or at least predict, how singularities form as the network degenerates, whereas self-similar
expanding solutions indicate how singular configurations (i.e., irregular networks in our setting) should 
flow to smooth (regular) ones.  Beyond this, of course, self-similar solutions also provide an 
interest class of examples of solutions.

A network $\Gamma = \{\gamma^{(j)}\}$ is an expanding self--similar solution to network flow if each evolving curve 
$\gamma^{(j)}$ has the form
\begin{equation}\label{solitonlambda}
\gamma^{(j)}(t,x) = \lambda(t) \eta^{(j)}\left(x/\lambda(t)\right); 
\end{equation}
the expansion factor $\lambda(t)$ is determined below. 

\medskip

As in the last subsection, let $\calF = \{\ell^{(1)}, \ldots, \ell^{(k)}\}$ be a fan of rays emanating from the origin in $\RR^2$.
Inspired by \cite{SS}, the paper \cite{MS} establishes the existence of at least one expanding self-similar solution
of the network flow which limits to $\calF$ as $t \searrow 0$. When $k = 3$, this solution is unique, but for $k > 3$
there are many (in fact, the number of solutions grows very rapidly as $k$ increases). 

First, a short calculation obtained by inserting \eqref{solitonlambda} into \eqref{evolution} and fixing the 
normalization $\lambda(1/2) = 1$ leads to the \emph{expander equation} 
\begin{equation}\label{expanders}
\kappa- \left\langle\eta, \nu\right\rangle=0, \qquad \lambda(t)=\sqrt{2t}.
\end{equation}
This equation is equivalent to the geodesic equation for the metric $g = e^{|x|^2} |dx|^2$ on $\RR^2$. 
The $t=1/2$ slice of a regular expanding self-similar network consists of a union of curves, each an arc 
or half-ray solution of \eqref{expanders}, which together constitute a regular network. 
Such a network is a stable critical point of the length functional in $(\RR^2,g)$ (where, as usual, it
suffices to look at the length of networks in any large ball $B_R(0)$).     The condition that 
expanding self-similar solution be a flowout of the initial fan $\calF$ means that for each external edge,
the image of $x \mapsto \sqrt{2t}\, \eta^{(j)}(x/\sqrt{2t})$ converges to $\ell^{(j)}$ as a set as $t \searrow 0$. 

The individual geodesic arcs and rays for the metric $g$ are qualitatively similar to geodesics in hyperbolic space,
as expected since the curvature of $g$ is everywhere negative.  For example, if $\ell$ and $\ell'$ are any two 
rays emanating from the origin, then there is a unique complete geodesic for this metric which is asymptotic to these rays 
along its two ends.  A convenient way to picture this is to realize the geodesic compactification of $(\RR^2, g)$ as
a closed ball $B$.  A limiting direction, i.e., the asymptotic limit of any ray $\ell$, corresponds to a point $q \in \del B$.
Thus the specification of the $k$ rays in the fan $\calF$ is the same as choosing $k$ distinct points
$q_1, \ldots, q_k \in \del B$. A solution of this problem, i.e., the configuration at $t=1/2$ of an expanding 
self-similar solution emanating from $\calF$ is thus a geodesic network for this metric. If there are no loops, then it is a
Steiner tree for $(\RR^2, g)$. 

\begin{proposition}\cite[Main Theorem]{MS}
Let $\mathcal{F} = \{\ell^{(1)}, \ldots, \ell^{(k)}\}$ be a fan of rays in $\RR^2$ emanating from the origin, and $q_1, \ldots, q_k$ the
corresponding points (listed in cyclic order) around $\del B$. Then the set of expanding self-similar solutions of the network flow 
with initial datum $\calF$ is in bijective correspondence with the set of (possibly disconnected) regular networks on $B$, each arc 
of which is a geodesic for the metric $g$, with boundary $\{q_1, \ldots, q_k\}$. For each decomposition of the set of boundary
points into subsets $J_i$ of collections of adjacent points, there is at least one geodesic Steiner tree whose components
are characterized by having boundary in the points $q_j$, $j \in J_i$. 
\label{Steiner}
\end{proposition}

We `normalize' by only considering solutions of \eqref{evolution3} rather than \eqref{evolution}. 
\begin{definition}
An \emph{expanding soliton} is a solution $\Gamma(t) = \{ \gamma^{(j)}(t,x) \}$ of the \emph{network flow} 
with initial network $\Gamma_0$, where for each $j$, 
\begin{equation*}
\gamma^{(j)}(t,x) = \sqrt{2t} \, \eta^{(j)}(x/\sqrt{2t}),  
\end{equation*}
where, if we set $s = x/\sqrt{2t}$, then each $\eta^{(j)}$ solves
\begin{equation}\label{soleq}
\frac{\del_s^2 \eta^{(j)}}{|\del_s \eta^{(j)}|^2} + ( s\del_s\eta^{(j)} - \eta^{(j)}) = 0\,.
\end{equation}
This soliton is called regular if the curves $\eta^{(j)}(s)$ are arcs of rays of a regular geodesic network in $(\RR^2, g)$,
and minimizing if  it  is  regular and a  tree.
\end{definition}
The parametrization of each $\gamma^{(j)}$ is singular at $t =0$; we explain how to think about this
in the next section. 

\medskip

To derive \eqref{soleq}, define $s = x/\sqrt{2 t}$ as above, and $\tau = \sqrt{2t}$. Thus $\gamma^{(j)} = \tau \eta^{(j)}(s)$. 
We then compute that
$$
\partial_t= \frac{1}{\sqrt{2}t}\partial_\tau-\frac{1}{2\sqrt{2}}\frac{x}{\sqrt{t^3}}\partial_s
=\frac{1}{\tau}\partial_\tau-\frac{s}{\tau^2}\partial_s\quad 
\partial_x=\frac{1}{\sqrt{2t}}\partial_s=\frac{1}{\tau}\partial_s\quad
\partial^2_x=\frac{1}{2t}\partial^2_s=\frac{1}{\tau^2}\partial^2_s\, ,
$$
and hence the flow equation reduces to
\[
\left(\frac{1}{\tau}\partial_\tau-\frac{s}{\tau^2}\partial_s\right)(\tau\eta(s))
=\frac{1}{\tau}\frac{\partial_s^2\eta(s)}{\vert \partial_s\eta(s)\vert^2}.
\]
This is the same as \eqref{soleq}.

\medskip

Later we shall require precise asymptotic information about the solutions $\eta$ to this equation. 
\begin{lemma}
Let $\eta$ be a solution to \eqref{soleq}. Then as $s \to \infty$, 
\begin{equation*}
\eta(s) = as + \calO(e^{-Cs^2})
\end{equation*}
for some $a \in \RR^2$ and $C > 0$. 
\end{lemma}
\begin{proof}
Define the quantities
\[
\mu = s\del_s \eta - \eta,\ \sigma = |\mu|^2,\ \mbox{and}\ \ \nu = |\del_s\eta|^2;
\]
using $'$ for $\del_s$ for simplicity, we  compute 
\[
\mu' = s \eta'' = - s \nu \mu,\ \ \sigma' = -2s \nu \sigma,\ \ \mbox{and}\quad \nu' = 2\langle \eta', \eta''\rangle=-2\nu\langle \eta', \mu\rangle
\geq -2\nu^{3/2} \sigma^{1/2}. 
\]

The equation for $\sigma$ gives $(\log \sigma)' = -2s \nu \leq 0$, hence $\sigma(s) \leq \sigma(0)$ for $s > 0$, while 
\begin{equation}
\nu^{-3/2} \nu' \geq -2\sigma^{1/2} \Rightarrow \nu(s) \geq  \left( C_1 +  \int_1^s \sigma^{1/2}(s')\, ds'\right)^{-2}.
\label{estnu}
\end{equation}

Our proximate goal is to show that $\nu(s) \geq C > 0$ as $s \to \infty$. Indeed, if that is the case, then
$(\log \sigma)' \leq  -Cs$, hence $\sigma \leq C_1 e^{-C_2 s^2}$.   

To show this positive lower bound on $\nu$, insert the initial estimate $\sigma \leq C$ into \eqref{estnu} to get $\nu \geq C/s^2$, hence 
$-s\nu \leq -C/s$.  Using this in the equation for $(\log \sigma)'$ yields 
$\sigma(s) \leq C s^{-\alpha}$ for some $\alpha > 0$.    We now repeat this argument, taking advantage of this improved bound.
There are three possible estimates for $\int_1^s \sigma^{1/2}(s')\, ds'$, depending on whether $\alpha < 2$, $\alpha = 2$ or
$\alpha > 2$.  In the first case, we derive that 
\[
\int_1^s \sigma^{1/2} \leq C s^{1-\alpha/2} \ \Rightarrow\ \nu \geq C s^{\alpha-2} \ \Rightarrow\ (\log \sigma)' \leq - C s^{\alpha-1}
\ \Rightarrow\ \sigma \leq C e^{-Cs^{\alpha}} \ \Rightarrow\ \nu \geq C > 0.
\] 
In the third case, $\int_1^s \sigma^{1/2} \leq C$, which gives directly the positive lower bound for $\nu$. The borderline 
case between these two produces an exponential decay rate for $\sigma$ which implies as before that $\nu \geq C > 0$. 

Returning to the original equation for $\sigma$, a positive lower bound for $\nu$ shows that $\sigma \leq C_1 e^{-C_2 s^2}$. 
With slightly more care, one can deduce that $C_2 = 1$, but this is not important for our purposes. 

Now observe that $( s^{-1} \eta)' = s^{-1} \eta' - s^{-2} \eta = s^{-2} \mu$ decays like $e^{-Cs^2}$, so we conclude finally that
$$
\eta(s) = as + \calO(e^{-Cs^2})
$$
for some $a \in \RR^2$. 
\end{proof}

\begin{lemma}\label{reparametrisation}
The solutions of~\eqref{expanders} and \eqref{soleq} are in bijective correspondence. 
\end{lemma}
\begin{proof}
Any solution of the network flow is a solution of the motion by curvature, so it suffices to consider only the normal 
component of the velocity. If $\gamma^{(j)}(t,x) = \sqrt{2t} \eta^{(j)}(x/\sqrt{2t})$ is a solution of the network flow,
then $\eta^{(j)}$ is a solution of~\eqref{expanders}. On the other hand, fix a solution of~\eqref{expanders}. 
This is the value at $t=1/2$ of a self-similar solution of the network flow, and the collection of $\eta^{(j)}$ 
determines a regular network. By~\cite[Theorem 4.14]{Man}, regular solutions of this flow are unique 
up to a reparametrization. The change of variable $s = x/\sqrt{2t}$, $\tau = \sqrt{2t}$ yields a  solution of~\eqref{soleq}. 
\end{proof}

Combining Proposition~\ref{Steiner} and Lemma~\ref{reparametrisation}, we see that for every $\calF$, there exists
at least one expanding soliton that flows out from $\calF$, and that these flowouts are in one-to-one correspondence
with geodesic networks with asymptotic boundary values determined by $\calF$.

\section{Parabolic blowups and the lifted flow}
We now come to the main conceptual step of this paper. This involves recasting the flow equations by introducing 
parabolic blowups of the interior vertices of the network at $t=0$, both in the domain and the range. This process 
exhibits the true homogeneities of the problem and leads to the precise regularity theory below. It also makes
the appearance of solitons particularly natural. 

\medskip

Consider for the moment the simplest situation, where $\Gamma_0=\{\gamma_0^{(j)}\}_{j=1}^3$ is a ``smooth triod'', i.e., 
a network with exactly three curves meeting at a regular vertex, the other three endpoints are three distinct external vertices. 
Even in this case a solution $\Gamma_t$ of the network flow is not necessarily smooth at the interior vertex at $t=0$.
In such a situation, smoothness at this `corner' only holds if the initial triod $\Gamma_0=\{\gamma_0^{(j)}\}_{j=1}^3$ satisfies 
an infinite set of compatibility conditions at this vertex. This is explained for the nonlinear flow in \cite{BrRe,MNT}. 
These compatibility conditions are obtained  
by taking successive derivatives of the equation and boundary conditions and using these to equate mixed partials of the 
various $\gamma^{(j)}$ at the corner $x=t = 0$.  Equality of mixed partials of course presupposes smoothness
at this corner. Thus, for example, the first compatibility condition comes from evaluating \eqref{mbc} at this corner.  The next 
condition arises by combining the $t$ derivative of the first equation in \eqref{mbc} with the equation of motion \eqref{evolution3}, 
again evaluated at the corner, to get
\begin{equation*}
\frac{\partial_x^2\gamma_0^{(i)}(0)}{\vert\partial_x\gamma_0^{(i)}(0)\vert^2}
=\frac{\partial_x^2\gamma_0^{(j)}(0)}{\vert\partial_x\gamma_0^{(j)}(0)\vert^2}
\end{equation*}
for every $i \neq j$.  The full sequence of compatibility conditions is obtained similarly. 

This issue is not special to this systems, of course.  Even for the most basic linear parabolic problem on a manifold with 
boundary $\Omega$, a solution of a mixed Cauchy--Dirichlet boundary problem 
\begin{equation*}
(\del_t - \Delta) u = 0, \qquad  u(0, \cdot) 
= \phi\in \calC^\infty, \quad \left. u \right|_{\del \Omega \times \RR^+} = \psi \in \calC^\infty
\end{equation*}
is not smooth at $\del \Omega \times \{0\}$ unless infinitely many compatibility conditions relating the Taylor series of $\phi$ 
at $\del \Omega$ and $\psi$ at $t=0$ are satisfied.  While certainly known as folklore for a century, this was 
investigated in some detail by Smale \cite{Smale}.  Our point is that it is not just a minor curiosity, but has some 
important consequences. For example, if the compatibility conditions are only satisfied up to some order $K$, then the mixed 
boundary problem can only be solved in $\calC^{ k+\delta, (k+\delta)/2}([0,T)\times \Omega)$ when $k \leq K$.  
In particular, if $\phi \equiv 0$, the Dirichlet heat semigroup is not strongly continuous on $\calC^{k+\delta}(\Omega)$. 

It is not a standard perspective, but also not hard to prove, that for this linear problem, $u$ becomes smooth if we pass to the space obtained 
by parabolically blowing up this corner. This means the following. Choose local coordinates $(x,y)$ on $\Omega$ with $x$ a defining function 
for the boundary. Now introduce ``parabolic polar coordinates'' in $(x,t)$, namely $x = r \cos \theta$, $t = r^2 \sin \theta$ (or equivalently,
$r = \sqrt{x^2 + t}$, $(x/r, t/r^2) = (\cos \theta, \sin \theta)$.  The parabolic blowup
\[
\Omega_h := [\Omega \times \RR^+; \del \Omega \times \{0\} ; dt] 
\]
is the space on which $r, \theta, y$ become smooth coordinates. This is called the heat single space and has corners up to 
codimension $2$; it has the two original boundary hypersurfaces, $\del \Omega \times \RR^+$ and $\Omega \times \{0\}$, as well
as the new hypersurface $r = 0$, which we call the front face and denote $\ff$.  
It is often more convenient to use projective coordinates on this space. One such coordinate system, for example, is
\[
s = x/\sqrt{t},  \tau = t, y,
\]
while another is
\[
\sigma = t / x^2,  \eta = x, y.
\]
Note that $(s,\tau,y)$ is only a nonsingular coordinate system on $\ff$ and the lift of $\del \Omega \times \RR^+$, while $(\sigma, \eta, y)$ 
is only nonsingular on $\ff$ and $\Omega  \times \{0\}$.  The regularity theorem stated above says, more precisely, that if $u$ is a solution 
to the Cauchy-Dirichlet problem above, with $\phi, \psi$ both smooth, then the lift of $u$ lies in $\calC^\infty(\Omega_h)$. The new front 
face intermediates between $\phi$ and $\psi$ and regularizes the solution.

\subsection{Blowup constructions}
Let us now return to the problem at hand. Suppose that $\Gamma_0$ is a (possibly irregular) network in a compact smoothly
bounded domain $\Omega \subset \RR^2$.  We shall introduce two separate blowups, one of the domain and the other of the range, 
which `regularize' the network flow problem. 

Recalling the notation of \S 2, consider an initial network $\Gamma_0=\{\gamma_0^{(j)}\}_{j=1}^n$ where each
$\gamma_0^{(j)}$ is a smooth map $[0,1]\to\mathbb{R}^2$. The images of these maps constitute the set of initial edges $\calE_0$.
This network has a set of external vertices $q_1, \ldots, q_r \in \del \Omega$ and interior vertices $p_1, \ldots, p_s$.
For simplicity, we require that the compatibility conditions are satisfied to all orders at every exterior vertex; these conditions state
that $\partial_x^i \gamma_0^{j}(1)=0$ for all $i \geq 2$ and indices $j$ corresponding to exterior edges.  If these conditions
are all satisfied, then classical methods are sufficient to prove estimates near these boundary points.  It is straightforward to modify
the constructions and arguments of this paper to include this analysis, and in fact to extend this to the case where these
compatibility conditions are not satisfied.  We leave this to the interested reader.

The eventual solution will be an evolving network $\Gamma(t) = \{\gamma^{(j)}\}$ with $m \geq n$ curves. Extra curves
are generated at every interior vertex where the  valence is greater than $3$.   We may  regard this entire network as
a collection of mappings from a disjoint union of regions, each in the $(t,x)$ plane. These regions are labelled $Q_j$, $j = 1, \ldots, n$,
for the evolutions of each of the initial curves $\gamma^{(j)}_0 \in \Gamma_0$ and  $P_j$, $j = n+1, \ldots, m$, for
the new families of curves generating out of the vertices.  We shall often think of  each curve  $\gamma^{(j)}$ simultaneously as a $1$-parameter 
family of maps $x \mapsto \gamma^{(j)}(t,x) \in \Omega$ but also as a map $(t,x) \mapsto (t, \gamma^{(j)}(t,x)) \in \RR^+ \times \Omega$. 
For simplicity, no separate notation is introduced for the latter; this conflation should not lead to any confusion. 

\medskip

\noindent{\bf Blowup of the domain:} 
For each $j = 1, \ldots, n$, let $Q^{(j)}=\{ (t,x) \in \mathbb{R}^2\, :  0 \leq t,\  0\leq  x\leq 1\}$ be the domain parametrizing
the evolution of $\gamma^{(j)}_0$.  If this is an interior edge, we define the parabolic blowup
\[
Q_h^{(j)}=[ Q^{(j)}, (0,0) \cup (0,1); dt];
\]
this is, by definition, the set obtained by replacing each of the two points $(0,0)$ and $(0,1)$ by their inward-pointing unit parabolic normal bundles,
and endowing the resulting space with the natural topology and smooth structure for which the lifts of smooth functions 
on $Q^{(j)}$ and parabolic polar coordinates around these two corner points are both smooth. These parabolic polar coordinates are defined near
$(0,0)$ as
\[
\rho = \sqrt{t + x^2} \geq 0,\ \ \omega = \arcsin(t/\rho^2) = \arccos(x/\rho) \in [0, \pi/2] 
\Leftrightarrow  (t,x) = (\rho \cos  \omega, \rho^2 \sin \omega)\,,
\]
and similarly near $(0,1)$.
The space $Q_h^{(j)}$ has two new boundary faces, called the front faces of $Q_h^{(j)}$, and denoted $\ff_0$ and $\ff_1$ (or just $\ff_0$
if we do not need to specify). If $\gamma^{(j)}_0$ is an exterior edge, then $Q_h^{(j)}$ is defined by only blowing up the corner at $(0,0)$ (see Figure \ref{Fig1}).  

There is another useful way to think about parabolic blowup: points of $\ff_0$, for example, are identified as equivalence classes of smooth paths 
$\zeta(\xi) = (\zeta_1(\xi), \zeta_2(\xi))$ in $Q^{(j)}$, $\xi \geq 0$, with $\zeta(0) = (0,0)$. Two paths $\zeta$, $\tilde{\zeta}$ 
are equivalent, $\zeta \sim \tilde{\zeta}$, if  there exists some $c > 0$ such that $|\tilde{\zeta}_1(s) - \zeta_1(cs)| \leq C s^2$,  
$|\tilde{\zeta}_2(s) - \zeta_2(cs)| \leq C s^3$.  Thus this construction separates out the endpoints of the paths $(\alpha s^2, s)$, 
$0 < \alpha < \infty$; each such path determines a unique point in the interior of $\ff_0$. 
Paths $\zeta(\xi)$ which do not approach the corner parabolically limit to one or the other boundary of $\ff_0$. 

Define
\[
Q_h = \sqcup_{j=1}^n Q^{(j)}_h, \qquad Q = \sqcup_{j=1}^n Q^{(j)}. 
\]
There is a natural blowdown map $\beta: Q_h \to Q$ which is the identity away from $\ff$ and which collapses each 
front face to $(0,0)$ and $(0,1)$, respectively. As above, $\calC^\infty(Q_h)$ is generated by $\beta^*(\calC^\infty(Q))$ 
and the functions $\rho$ and $\omega$. 

Each component of $Q_h$ has either three or four boundary faces: the front faces $\ff_0$ and $\ff_1$, two side faces $\lf$ 
and $\rf$ and the bottom face $\bff$.  Each $\ff_j$ is given by $\rho = 0$ in local parabolic polar coordinates; the left and right faces 
are the vertical sides `above' $\ff$ where $\omega = \pi/2$, and the bottom face is the initial face $t=0$, between
the two front faces, and at $\omega = 0$. The eventual map $\gamma^{(j)}(t,x)$ determined by the evolution will be defined 
on $Q^{(j)}_h$ rather  than $Q^{(j)}$. It has initial condition $\gamma^{(j)}(x)$ on $\bff$ and satisfies the matching (Herring) conditions
along the left and right faces. Its behavior on the front face(s) is one of the key issues addressed below. 

As in \S 2, the evolving network $\Gamma(t)$ will also include a set of new internal edges $\gamma^{(j)}$, $j = n+1, \ldots, m$.
These are defined on the set $P^{(j)}$ defined by $t \geq 0$ and $0 \leq x \leq \sqrt{t}$.  The fact that this region shrinks
to a point at $t=0$ corresponds to the fact that these curves $\gamma^{(j)}$ disappear into the interior vertex points $p_i$.
We also blow these up parabolically at $(0,0)$, obtaining regions
\[
P^{(j)}_h = [ P^{(j)}, (0,0); dt].
\]
In parabolic polar coordinates defined exactly as above, this space has coordinates $(\rho, \omega)$ with $\rho \geq 0$ 
and $0 \leq \omega \leq \arcsin (1/2) = \pi/6$.  We write
\[
P_h = \sqcup_{j=n+1}^m  P^{(j)}_h.
\]
The space $Q_h \sqcup P_h$ is the entire desingularized domain of the evolving network $\Gamma(t)$.

In computations it is usually simpler to work in coordinate systems different than the parabolic polar coordinates above.
In particular, we introduce two sets of projective coordinates near each front face.  In the following, for simplicity,
we describe these only near the front face over $(0,0)$, but these are obviously defined near the front face over $(0,1)$. 

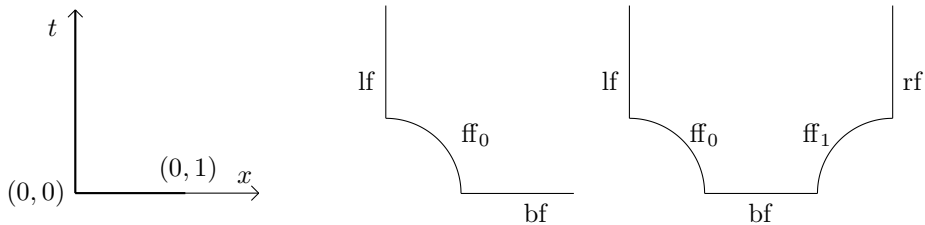
\begin{figure}[h]
\begin{center}
\begin{tikzpicture}[scale=1.5]
\draw[black, scale=0.65]
(0,0)--(0,2.5)
(0,0)--(2.5,0);
\draw[black, thick, scale=0.65]
(0,0)--(0,2.5)
(0,0)--(1.5,0);
\draw[black, scale=0.65, shift={(2.5,0)},  rotate=-90]
(0,0)to[out= -45,in=135, looseness=1] (0.1,-0.1)
(0,0)to[out= -135,in=45, looseness=1] (-0.1,-0.1);
\draw[black, scale=0.65, shift={(0,2.5)},  rotate=0]
(0,0)to[out= -45,in=135, looseness=1] (0.1,-0.1)
(0,0)to[out= -135,in=45, looseness=1] (-0.1,-0.1);
\draw[white]
(0,-0.335)--(0.2,-0.335);
\path[font=\small]
(0,0)[left] node{$(0,0)$}
(1,0)[above] node{$(0,1)$}
(1.5,0)[above] node{$x$}
(-0.2,1.3)[above] node{$t$};
\end{tikzpicture}\qquad\quad
\begin{tikzpicture}
\draw[black]
(0,1)--(0,2.5)
(1,0)--(2.5,0);
\draw[color=black,scale=1,domain=0: 1.57,
smooth,variable=\t,shift={(0,0)},rotate=0]plot({1.*sin(\t r)},
{1.*cos(\t r)}) ; 
\path[font=\small]
(2,0)[below]node{$\bff$}
(1.2, 0.5)[above] node{$\ff_0$}
(0,1.5)[left] node{$\lf$};
\end{tikzpicture}\;
\begin{tikzpicture}
\draw[black]
(3.5,1)--(3.5,2.5)
(0,1)--(0,2.5)
(1,0)--(2.5,0);
\draw[color=black,scale=1,domain=0: 1.57,
smooth,variable=\t,shift={(0,0)},rotate=0]plot({1.*sin(\t r)},
{1.*cos(\t r)}) ; 
\draw[color=black,scale=1,domain=-1.57: 0,
smooth,variable=\t,shift={(3.5,0)},rotate=0]plot({1.*sin(\t r)},
{1.*cos(\t r)}) ; 
\path[font=\small]
(1.75,0)[below]node{$\bff$}
(2.5, 0.5)[above] node{$\ff_1$}
(3.5,1.5)[right] node{$\rf$}
(1, 0.5)[above] node{$\ff_0$}
(0,1.5)[left] node{$\lf$};
\end{tikzpicture}
\end{center}
\caption{The spaces $Q$ and $Q_h$  (in the case $\gamma^{(j)}(0)$ is an interior irregular vertex and 
$\gamma^{(j)}(1)$ is an exterior vertex or in the case both $\gamma^{(j)}(0)$ and $\gamma^{(j)}(1)$ are interior irregular vertices).}
\label{Fig1}
\end{figure}

\begin{itemize}
\item {\bf Coordinates} $(\tau,s)$.\\
Define $\tau=\sqrt{2t}$ and  $s=\frac{x}{\sqrt{2t}}$. Then $(\tau, s)$ is a nondegenerate coordinate system on $Q_h^{(j)}$ near
$\ff_0$ away from $\bff$; there is a corresponding set of projective coordinates near $\ff_1$, of course. In addition,
we may also use these coordinates in $P^{(j)}_h$ near its front face where $\rho = 0$. Note that in each $Q^{(j)}_h$, 
$s \in [0,\infty)$, while in $P^{(j)}_h$, $0 \leq s \leq 1$. The variable $\tau$ is a defining function for the front face in each
of these cases in the sense that it vanishes simply on $\ff$ and nowhere else, and is commensurable with $\rho$ on 
any compact set in $Q_h^{(j)}$ which does not intersect $\rf$, and on the entirety of $P^{(j)}_h$. The variable 
$s$ is a defining function for $\lf$ or $\rf$, and identifies $\mathrm{int}\, \ff$ with $\RR^+$. 

We compute 
\begin{equation*}
\partial_t=\frac{1}{\tau}\partial_\tau-\frac{s}{\tau^2}\partial_s\,,\qquad
\partial_x=\frac{1}{\tau}\partial s\,,\qquad \partial^2_x=\frac{1}{\tau^2}\partial s^2\,,
\end{equation*}
and hence the linear heat equation becomes
\begin{equation}\label{heatcoordstau}
\del_t - \del_x^2 = \frac{1}{\tau^2} \left(\tau\partial_\tau u-s\partial_s u-\partial^2_s u\right)=0\,.
\end{equation}

The singular factor on the right is inevitable because the two sides of this equality must both have homogeneity
$-2$ with respect to the scaling $(t,x) \mapsto (\lambda^2 t, \lambda x)$, and $s$ and $\tau$ are homogeneous
of degree $0$.  In any case, it is sufficient to study the operator 
\begin{equation*}
 \tau\del_\tau - s \del_s - \del_s^2.
\end{equation*}

\item {\bf Coordinates} $(T,y)$.\\
The $(\tau,s)$ coordinates are not valid near the bottom face where $t=0$ and in particular near the intersection
of $\ff$ and $\bff$. Near this corner we introduce an alternate set of projective coordinates $y=x$ and $T=\frac{t}{x^2}$.  
These are singular along the positive $t$-axis; the variable $y$ is now the defining function for $\ff$ and $T$ is the defining 
function for $\bff$. 

The expressions for $\del_t$ and $\del_x$ in these new coordinates are
\begin{equation*}
\partial_t=\frac{1}{y^2}\partial_T\,,\qquad \partial_x=\partial_y-\frac{2T}{y}\partial_T\,,\qquad
\partial^2_x=\partial^2_y-\frac{4T}{y}\partial_y\partial_T+\frac{4T^2}{y^2}\partial^2_T
+\frac{6T}{y^2}\partial_T\, ,
\end{equation*}
and now the the linear heat operator takes the much less tractable form
\begin{equation*}
\frac{1}{y^2}\left(\partial_Tu-6T\partial_Tu-4T^2\partial_T^2u
+4Ty\partial_T\partial_yu-y^2\partial^2_y u\right)=0.
\end{equation*}
As before, it suffices to consider the slightly nicer looking operator 
\begin{equation}
\del_T-6T\del_T-4T^2\del_T^2
+4Ty\del_T\del_y-y^2\del^2_y.
\label{ho2}
\end{equation} 
\end{itemize}

We also record that
\begin{equation*}
T=\frac{1}{2s^2}, \quad y=s\tau.
\end{equation*}

This change of variables leads to sufficiently ungainly expressions that we do not use this coordinate system for
any serious calculations. 

\medskip

\noindent {\bf Blowup of the range:}  The dilation properties of expanding solitons suggest that the
homogeneity in the range should also be emphasized. In other words, the ansatz that solitons take the form 
$\gamma(t,x) = \sqrt{2t} \eta( x/\sqrt{2t})$ already indicates the relevance of the projective coordinate $s = x/\sqrt{2t}$, 
but also suggests the change of variable $z \in \RR^2$ to $w = z/\sqrt{2t}$ in the range. We formalize this as follows.

Define $Z = \RR^+_t \times \RR^2_z$.  Following the remark just before \S 3.1, in the following we regard each 
$\gamma^{(j)}$ as a map into $\RR^+ \times \Omega$ via $(t,x) \mapsto (t, \gamma^{(j)}(t,x))$.  As suggested
by the preceding paragraph, we lift this map by blowing up both the domain and the range. In other words, 
each lift should be regarded as a map
\[
\gamma^{(j)}(t,x): Q_h^{(j)} \longrightarrow  Z_h = [Z, \{(0,p_1), \ldots, (0,p_s)\} ; dt].
\]
This latter space is obtained from $Z$ by taking the parabolic blowup of $Z$ at each of the interior vertices $p_i$ at $t=0$.
This parabolic blowup is defined exactly as before, by replacing each $(0,p_i)$ with the inward-pointing spherical
parabolic normal bundle. As before, this becomes more tangible in locally defined parabolic polar coordinates.
Fixing one interior vertex $p_i$, suppose that $p_i = 0$ and define
\[
R = \sqrt{t + |z|^2}, \ \ \Theta =  (t/ R^2, z/R) = (\Theta_0, \Theta').
\]
Then $Z_h$ has a front face $\ff = \{R=0\}$ and a bottom face $\bff = \{\Theta_0 = 0\}$. There is a codimension two corner 
where these two faces intersect.  

There are useful projective coordinates here too, namely
\[
\tau = \sqrt{2t}, \qquad w = z/\sqrt{2t}.
\]
Thus $\tau = 0$ is a defining function for $\ff$, while $w$ is a projectively natural linear coordinate {\color{violet}for} $\ff$. These 
coordinates are valid away from $\bff$. Thus $(t, \gamma^{(j)}(t,x))$ lifts to $(\frac12 \tau^2,  \tau 
\eta^{(j)}(\tau,s))$. 

\medskip

\noindent {\bf The lifted equation:}  
We now consider the evolution equation in terms of these blowups, i.e., we lift the maps $\gamma^{(j)}$ to maps 
between $Q_h^{(j)}$ and $Z_h$.  This lifting is effected simply by writing 
\begin{equation}\label{motion}
\partial_t\gamma=\frac{\partial^2_x \gamma}{\vert\partial_x\gamma\vert^2}
\end{equation}
using the coordinate systems $(\tau,s)$ on $Q_h^{(j)}$ and $(\tau, w)$ on $Z_h$. 

Set $\gamma^{(j)} = \tau \eta^{(j)}$, which corresponds to the introduction of the projective coordinate on $Z_h$, and 
for simplicity, drop the superscript $(j)$ for the time being.  As noted earlier, if $\gamma$ is an arc in an expanding soliton, 
then $\eta$ depends only on $s$.  Recalling that $\del_t = \tau^{-2}( \tau\del_\tau - s\del_s)$ and $\del_x = \tau^{-1} \del_s$, hence 
\eqref{motion} becomes
\[
\frac{1}{\tau^2} (\tau \del_\tau - s\del_s) (\tau \eta) =  \frac{ \tau^{-2} \del_s^2 (\tau \eta)}{ | \tau^{-1} \del_s (\tau \eta)|^2},
\]
or finally,
\begin{equation}
(\tau \del_\tau + 1 - s\del_s) \eta = \frac{ \del_s^2 \eta}{|\del_s\eta|^2}.
\label{liftedeq1}
\end{equation}

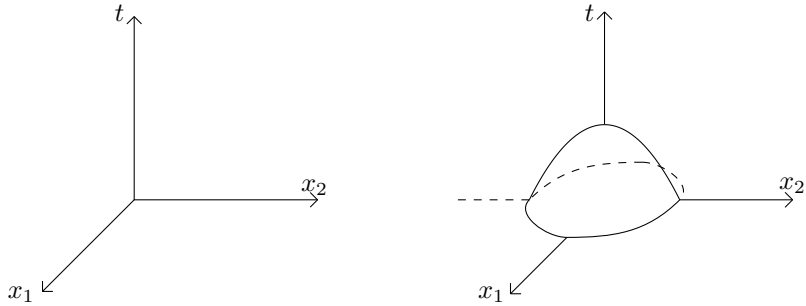
\begin{figure}[h]
\begin{center}
\begin{tikzpicture}[scale=1.5]
\draw[black, scale=0.65]
(0,0)--(-1.25,-1.25)
(0,0)--(0,2.5)
(0,0)--(2.5,0);
\draw[black, scale=0.65,shift={(-1.25,-1.25)},  rotate=135]
(0,0)to[out= -45,in=135, looseness=1] (0.1,-0.1)
(0,0)to[out= -135,in=45, looseness=1] (-0.1,-0.1);
\draw[black, scale=0.65, shift={(2.5,0)},  rotate=-90]
(0,0)to[out= -45,in=135, looseness=1] (0.1,-0.1)
(0,0)to[out= -135,in=45, looseness=1] (-0.1,-0.1);
\draw[black, scale=0.65, shift={(0,2.5)},  rotate=0]
(0,0)to[out= -45,in=135, looseness=1] (0.1,-0.1)
(0,0)to[out= -135,in=45, looseness=1] (-0.1,-0.1);
\path[font=\small]
(-1,-1)[above] node{$x_1$}
(1.6,-0.05)[above] node{$x_2$}
(-0.13,1.5)[above] node{$t$};
\end{tikzpicture}\qquad\qquad
\begin{tikzpicture}
\draw[black]
(-0.5,-0.5)to[out= 180,in=-135, looseness=1] (-1,0)
(-0.5,-0.5)to[out= 0,in=-135, looseness=1] (1,0);
\draw[dashed, black]
(-1,0)--(-2,0)
(0.5,0.5)to[out= 180,in=45, looseness=1] (-1,0)
(0.5,0.5)to[out= 0,in=45, looseness=1] (1,0);
\draw[black, shift={(-0.5,-0.5)}]
(0,0)--(-0.75,-0.75);
\draw[black]
(0,1)--(0,2.5)
(1,0)--(2.5,0);
\draw[black,  shift={(0,1)}]   plot[smooth,domain=-1:0] (\x, {\x^2});
\draw[black,  shift={(0,1)}]   plot[smooth,domain=0:1] (\x, {-\x^2});
\draw[black, shift={(-1.25,-1.25)}, scale=1, rotate=135]
(0,0)to[out= -45,in=135, looseness=1] (0.1,-0.1)
(0,0)to[out= -135,in=45, looseness=1] (-0.1,-0.1);
\draw[black, shift={(2.5,0)}, scale=1, rotate=-90]
(0,0)to[out= -45,in=135, looseness=1] (0.1,-0.1)
(0,0)to[out= -135,in=45, looseness=1] (-0.1,-0.1);
\draw[black, shift={(0,2.5)}, scale=1, rotate=0]
(0,0)to[out= -45,in=135, looseness=1] (0.1,-0.1)
(0,0)to[out= -135,in=45, looseness=1] (-0.1,-0.1);
\path[font=\small]
(-1.5,-1.5)[above] node{$x_1$}
(2.5,-0.05)[above] node{$x_2$}
(-0.2,2.25)[above] node{$t$};
\end{tikzpicture}
\end{center}
\caption{From $Z=\mathbb{R}^+_t\times\mathbb{R}^2_z$ to $Z_h$}
\label{fig2}
\end{figure}

In particular, if $\gamma$ is an expander, so $\eta = \eta(s)$, this yields the dimensionally reduced soliton equation
\begin{equation}
\frac{\del_s^2 \eta}{|\del_s\eta|^2}  + (s\del_s - 1) \eta = 0.
\label{stationary}
\end{equation}
Observe that if $\eta(s)$ solves this last equation, then so does $\eta(cs)$ for any constant $c > 0$.
We are therefore free to normalize by assuming that
\begin{equation}
|\del_s \eta| \to 1\ \ \mbox{as}\ \ s \to \infty.
\label{normalization}
\end{equation}

\medskip

\noindent {\bf Boundary conditions:} 
We now explain the various boundary conditions and how they should be formulated along each of the boundary faces of 
$Q_h \sqcup P_h$.   Clearly we must demand that the restriction of each $\gamma^{(j)}$, $j \leq n$, to the bottom
face of $Q^{(j)}_h$ is the given parametrization $\gamma^{(j)}_0$ of that arc of the initial network.    Next, fix
any $j$ and consider either the left or right face of the appropriate domain region $Q^{(j)}_h$ or $P^{(j)}_h$.   If the
left face of this domain does not correspond to an exterior vertex, then it is paired to exactly two other faces (either
left or right) of some other $Q^{(i)}_h$ or $P^{(i)}_h$, corresponding to the evolving trivalent vertex of the regular
network $\Gamma(t)$. Let us denote this `cluster' by $N(\ell)$, where $\ell$ is a vertex of $\Gamma(t)$, and assume that 
the matching boundaries are the left faces of each.  Then in terms of the $(\tau,s)$ coordinate system, the boundary conditions
at these faces become
\[
\eta^{(i)}(\tau,0) = \eta^{(j)}(\tau,0)\ \ \forall\ i,j \in N(\ell), \qquad \sum_{i \in N(\ell)} \frac{\del_s \eta^{(i)}(\tau,0)}{|\del_s \eta^{(i)}(\tau,0)|} = 0.
\]
Every left face corresponding to an interior vertex of $\Gamma(t)$ arises exactly once in one of these clusters. This specifies
the boundary conditions along all of the side faces.

Finally, it is perhaps not obvious how to specify an `initial' condition along any one of the front faces. To determine this,
note that we expect the lifted map $\eta^{(j)}(\tau,s)$ to be bounded as $\tau \to 0$. We make the assumption
that $\eta^{(j)}$ is in fact smooth up to the front face.  This means that it has a boundary value $\eta^{(j)}_0(s)$.
Noting that $\tau\del_\tau \eta^{(j)}|_{\tau=0} = 0$, we deduce that 
\[
\frac{\del_s^2 \eta_0}{|\del_s \eta_0|^2} + (s\del_s - 1) \eta_0 = 0,
\]
which is precisely the soliton equation!  In other words, solitons arise naturally as the initial conditions for the flow
along these new front faces.

As a word of extra motivation, notice that the lifted flow equation \eqref{liftedeq1} is characteristic along the front face.
Change variables, setting $\bar{\tau} = \log \tau$. Then \eqref{liftedeq1} becomes
\[
(\del_{\bar{\tau}} + 1 - s\del_s) \eta = \frac{\del^2_s \eta}{|\del_s\eta|^2},
\]
which appears to be nondegenerate, but now of course, $\bar{\tau} \to -\infty$ along $\ff$.  In order for there to be
a reasonable short-time solution for the original problem, it is necessary that this solution restrict to
an ancient solution near the front face. As is well-known in evolution problems, the choice for limiting `initial' data
as $\bar{\tau} \to -\infty$ must be extremely rigid in order for the flow to exist on this semi-infinite time interval. 

The final point is to recall that we have chosen a particular expanding self-similar solution $\Gamma^{\ssim, (i)}_j$ 
for each interior vertex $p_i$ of $\Gamma_0$. The various arcs of this soliton correspond precisely to the various 
front faces of the domains $Q^{(j)}_h$ and $P^{(j)}_h$ which have left faces limiting to $p_i$. The $\tau=1$ slice of the
soliton itself should be regarded as the geometric limit of these maps along the corresponding front face of $Z_h$ over $(0, p_i)$.

\section{A model: the scalar heat equation}\label{Section:model case}
As a further step toward explaining our main theorem, we now carry out in full detail the method 
we are espousing for the scalar heat equation on the half-line, where a number of simplifications occur.  We describe
briefly at the end of this section the minor changes that need to be made to prove the corresponding
result for higher dimensional domains. 
\begin{theorem}\label{scalarheateq}
Suppose that $u$ is the (unique moderate-growth) solution of the heat equation
$$
u_t-u_{xx}=0
$$
on $Q := [0,\infty)_t \times [0,\infty)_x$ with boundary conditions
\[
u(0,x) = \phi(x),\ \ u(t,0) = \psi(t),
\]
where both $\phi$ and $\psi$ are $\calC^\infty$ on the closed half-line. for
simplicity we assume that each is supported in $[0,1)$.  Then $u$ is smooth 
on the space $Q_h = [ [0,\infty)^2; (0,0); \{dt\}]$ obtained by parabolically blowing
up the quarter-space at the origin.  This solution is $\calC^\infty$ on the
closed quarter-space only when infinitely many `matching conditions' are satisfied.
\end{theorem}
The last assertion is classical, but the fact that $u$ is always ``smooth''
on the blown up space if the matching conditions are not satisfied  is the novel part of this.  

The fact that the spatial domain is $\RR^+$ rather than a closed interval $[0,1]$ is immaterial, and the 
proof below could be extended to that case with only minimal change of notation. It is 
easy to adapt the proof below to the higher dimensional case, i.e., to solutions of the scalar heat
equation on $\RR^+_t \times \Omega_x$, where $\Omega$ is a domain in $\RR^n$ with smooth boundary. 
We comment further on this extension at the end of this section.

\medskip

\noindent \textit{Proof of Theorem~\ref{scalarheateq}}
There are two steps. In the first we show that the Taylor series of the boundary data at $x=0$ and $t=0$ respectively 
determine all terms in the expansion of $u$ on $Q_h$. In other words, we produce a smooth function 
$v$ on $Q_h$  such that $(\del_t - \del_x^2) v = f$ vanishes to all orders as $t \searrow 0$ in $\del Q_h$, and is supported
in $\{(t,x): 0 \leq t, x \leq 2\}$.   In the second step we use the heat kernel to produce a function $w$ which also 
vanishes to all orders at the boundary of $Q$ and which solves $(\del_t - \del_x^2) w = f$ with $w(t,0) = w(0,x) = 0$. 
The solution $u = v - w$ then has the desired properties.   There is one non-local aspect in this procedure which
only becomes apparent once we pass to the blowup. 

\medskip

\noindent {\bf Construction of the approximate solution $v$:}  We now construct an approximate solution $v$ to the 
homogeneous linear heat equation. For simplicity, assume that $v(0,x) = \phi(x) \in \calC^\infty(\RR^+)$
and $v(t,0) = 0$.  We comment on the more general case where $v(t,0) = \psi(t) \in \calC^\infty(\RR^+)$ later.  

We shall {\it assume} that the lift of $v$ to $Q_h$ is smooth, and show how the heat equation and boundary conditions 
determine its Taylor series at each boundary.  

First observe that the equation $(\del_t - \del_x^2)v \equiv 0$ (in Taylor series as $t \to 0$) determines the full
Taylor series of $v$ along $t=0$ for $x > 0$ in terms of the intial data $\phi$ and its derivatives. 

Now consider this equation near the front face $\ff$ of $Q_h$.  Using the $(\tau,s)$ coordinates, write 
\[
v \sim \sum_{j=0}^\infty v_j(s)\tau^j.
\]
Each $v_j(s)$ is a function on $\ff$. 
The boundary condition $v(t,0)  = 0$ forces each $v_j(0) = 0$.   Now substitute this expansion into \eqref{heatcoordstau}, and cancel
the extra factor of $\tau^{-2}$, to get 
\[ 
(\tau \del_\tau - s\del_s - \del_s^2) v \sim  \sum_{j=0}^\infty ( j v_j - s v_j' - v_j'') \tau^j \equiv 0,
\]
where the prime denotes a derivative with respect to $s$.  We conclude that 
\begin{equation}
 L_j v_j := v_j''(s) + s v_j'(s) - j v_j(s) = 0, \ \ j = 0, 1, 2, \ldots,
\label{eqvj}
\end{equation}
We bring in the boundary condition at $t=0$, i.e., as $s \to \infty$, momentarily.

The equation $L_0 v_0 = 0$ is explicitly solvable. Indeed, setting $z = v_0'$, then the equation becomes
$z' =- s z$, which means $z(s) = c e^{-s^2/2}$, and hence 
\[
v_0(s)  = c \int_0^s e^{-\sigma^2/2} \, d\sigma
\]
for some $c \in \RR$.  Observe that $v_0(s)$ converges to a limit $\alpha_0 > 0$ as $s \to \infty$. In fact, 
$v_0(s) = \alpha_0 + \calO(R(s) e^{-s^2/2})$ for some polynomial $R(s)$.   We do not need to keep careful track of
these exponential remainders, so in all the following, we write simply $v_0(s) = \alpha_0 + E$ where $E = \calO(e^{-\delta s^2})$ 
for some $\delta > 0$. We use similar expressions below when this leading part is more complicated than just a constant. 

A fortuitous identity allows us to solve the other equations in this sequence for $j > 0$. To do this,  observe that differentiating 
\eqref{eqvj} with respect to $s$ gives
\[
(L_j v_j)' = L_{j-1} v_j' = (v_j')'' + s (v_j')' - (j-1) v_j' = 0.
\]
Iterating $j-1$ times shows that $L_0 v_j^{(j)} = 0$.    As shown above, this yields that $v_j^{(j)} = \mbox{const.} + E$, 
and hence 
\begin{equation*}
v_j(s) = P_j(s) + E
\end{equation*}
where $P_j$ is a polynomial of degree $j$ in $s$. Writing this out explicitly, we have
\begin{equation}
v_j(s)=c_{j,j} s^j+c_{j, j-1}s^{j-1}+\ldots+c_{j, 1} s+c_{j, 0}+ E.
\label{polyj} 
\end{equation}

We now use the heat equation to determine recursion relationships between these coefficients.  First recall that $c_{0,0} \neq 0$ unless
$v_0 \equiv 0$, which we assume is not the case.  Inserting \eqref{polyj} into \eqref{eqvj}, and reindexing, we obtain
\[
- c_{j, j-1}s^{j-1}+\sum_{\ell=0}^{j-2}  \big( (\ell+2)(\ell+1) \, c_{j, \ell +2 }  + (\ell-j) \, c_{j, \ell} \big)\,  s^\ell = 0,
\]
hence
\begin{equation}
(\ell+2)(\ell+1) c_{ j, \ell+2} + (\ell-j) c_{j, \ell}  = 0, \ \ \ell = 0, \ldots, j. 
\label{recc}
\end{equation}
Since $c_{j,\ell} = 0$ for $\ell > j$, these recursion relationships show that $c_{j,j}$ may be chosen freely, while $c_{j, j-1} = 0$.
They also recursively determine each $c_{j, j-2k}$ as a nonzero multiple of $c_{j,j}$, and show that each $c_{j, j-2k-1} = 0$ 
since it is a multiple of $c_{j, j-1}$. 

The only point to check is that each $c_{j,j} \neq 0$, or more accurately, that $c_{j,j} = 0$ implies that $v_j$ vanishes identically. 
To see this, first note that $c_{j,j} = 0$ implies that all $c_{j,\ell} = 0$, so that $v_j$ decays exponentially. In addition, $v_j(0) = 0$.
Now observe that
\[
\int_0^\infty - (L_j v_j) v_j \, ds = \int_0^\infty  (v_j')^2 + (j + \frac12) (v_j)^2 \, ds  = 0
\]
since 
\[
\int_0^\infty  s v_j v_j' \, ds = \frac12 \int_0^\infty  s (v_j^2)'\, ds = - \frac12 \int_0^\infty (v_j)^2\, ds,
\]
which would show that $v_j \equiv 0$ as claimed. 

To determine $v_j$ completely, it suffices to determine the coefficient $c_{j,j}$, and this is where the 
initial condition at $t=0$ is required.   It is now convenient to switch to the other coordinate
system $(T,y)$, since this is nonsingular near the corner $\ff \cap \bff$ of $Q_h$.   

Using that $T=\frac{1}{2s^2}$, $y=s\tau$, or equivalently $s =(2T)^{-1/2}$, $\tau = y (2T)^{1/2}$, we see that 
$s^\ell \tau^j=y^j(2T)^{\frac{\ell-j}{2}}$. Altogether, modulo the exponentially decreasing error term,
\begin{equation}
v \sim \sum_{j=0}^\infty\sum_{\ell = 0}^j c_{j ,\ell}s^\ell\tau^j=\sum_{j=0}^\infty \sum_{\ell=0}^j 
2^{(\frac{j-\ell}{2})}c_{j,\ell}y^jT^{(\frac{j-\ell}{2})},
\label{expvnc}
\end{equation}
where the $c_{j,\ell}$ are the same coefficients as were discussed above.  Now suppose that
\[
v \sim \sum_{j, p = 0}^\infty A_{jp } y^j T^p, \ \ \mbox{as}\ y, T \to 0.
\]
Inserted into \eqref{ho2}, this leads to the recursion relation
\[
(p+1) A_{j, p+1} = \big[ j(j-1) + 2p (2p + 1 - 2j) \big] A_{jp},
\]
or more simply
\begin{equation}
A_{j, p+1}= \frac{1}{p+1} (j-2p)(j-2p-1) A_{j p}.
\label{recA}
\end{equation}

To compare with \eqref{expvnc}, set $p = (j-\ell)/2$, i.e., $\ell = j - 2p$, so 
\begin{equation}
2^p c_{j, j-2p} = A_{jp}.
\label{c=A}
\end{equation}
For this to be consistent, it is necessary that $A_{jp} = 0$ for $p > j/2$, but this is immediate from \eqref{recA}.  
Comparing the recursion formulas \eqref{recc} and \eqref{recA} more carefully, one can check that they are actually
the same, so that \eqref{c=A} is true for all indices. Furthermore, the coefficient
\[
c_{j,j} = A_{j0} = \left. \del_y^j v(T,y) \right|_{y=T=0}
\]
is nothing other than the $j^{\mathrm{th}}$ derivative of $\phi(x)$ at $x=0$.   This shows that the initial condition
$u_t|_{t=0} = \phi$ determines all of the $c_{j,j}$, and hence the entire set of coefficients $v_j(s)$, $j \geq 0$. 

We have now shown that the initial data $\phi$ and the heat equation itself determine the full set of Taylor coefficients of the solution 
near the front and bottom faces of $Q_h$. 

To determine the Taylor series of $v$ at $x=0$ for $t > 0$, note that the heat equation shows that all even derivatives
$\del_x^{2j} v(t,0)$ must vanish since we have specified that $v(t,0) \equiv 0$.  The odd Taylor coefficients 
$\del_x^{2j+1} v(t,0)$ are all determined by $\del_x v(t,0)$.  We can choose this normal derivative to be any smooth function, 
but for simplicity, let us assume that $\del_x v(t,0)$ vanishes identically.  This shows that $v$ vanishes 
to infinite order as $x \to 0$.   It is perhaps slightly more reassuring to repeat this calculation in the $(\tau,s)$ coordinates
in order to be sure that it really can be done uniformly up to the front face, but this is indeed the case.  

When the boundary condition $u(t,0) = \psi(t)$ is smooth, but not identically  zero, we see that all even Taylor coefficients 
of $v$ are determined recursively in terms of derivatives of $\psi$; as before, assuming that $\del_x v(t,0) = 0$, all
the odd coefficients vanish. In this case where $\psi(t) \neq 0$, we must also alter the prior step in determining the
solutions to $L_j v_j = 0$ along the front face, since now we must prescribe the initial condition $v_j(0)$ to 
equal the recursively determined value of $(1/j!)\del_s^j v(0,0)$. 

Having determined all of these Taylor coefficients at all boundaries of $Q_h$, we now invoke Borel's Lemma.
This classical result states that given any prescribed functions on the boundary components of a manifold with corners, there
exists a smooth function for which these functions are precisely the coefficients of the Taylor series with respect to
some fixed choice of boundary defining function at that face. We apply this to construct a smooth function $v$ on $Q_h$, the 
Taylor coefficients of which at all boundaries of $Q_h$ are the quantities we have determined above.  By construction, 
$(\del_t - \del_x^2) v = f$ vanishes to infinite order as $t \searrow 0$ in $Q_h$. 

\medskip

\noindent {\bf Construction of the correction term $w$:} The final step is to define the function
\[
w(t,x) = \int_0^\infty \int_0^\infty  H_D( t - \tilde{t}, x, \tilde{x}) f(\tilde{t}, \tilde{x})\, d\tilde{t} d\tilde{x},
\]
where $H_D(t, x, \tilde{x})$ is the Dirichlet heat kernel on $\RR^+$, i.e., the solution operator for 
$(\del_t - \Delta) u = 0$, $u(0,x) = \phi(x)$ and $u(t,0) = 0$. This heat kernel is known explicitly, of course,
by the method of reflections.   Using the rapid vanishing of $f$, it is easy to see that $w(t,x)$ vanishes to all 
orders as $t \to 0$, uniformly for $x \geq 0$.  In particular, the lift of $w$ to $Q_h$ vanishes to
all orders on the front face. In addition, $w(t,0) = 0$ since $H_D(t-\tilde{t}, 0, \tilde{x})$ vanishes.  On the other hand, 
$\del_x H_D(t, 0, \tilde{x})$ is a smooth nonvanishing function, as are each of the odd derivatives $\del_x^{2j+1} 
H_D(t, 0, \tilde{x})$.  Hence $\del_x^{2j+1} w$ is a smooth function on the left face; as remarked earlier,
all of these functions vanish rapidly as $t \searrow 0$.  This proves that $w(t,x)$ is $\calC^\infty$ 
on $[0,\infty)_t \times [0,\infty)_x$ and vanishes to all orders at $t = 0$, and all its even Taylor
coefficients at $x=0$ also vanish. Lifted to $Q_h$, it vanishes rapidly on the front and bottom faces,
and all even terms in its series vanish at the left face. 

The true solution $u(t,x) = v(t,x) - w(t,x)$ to the heat equation with specified initial and boundary conditions is
now smooth on $Q_h$ as claimed.   Its series expansion at the front face of $Q_h$ is determined only `semilocally'
on this space, since the coefficients in the expansion here are solutions to the induced ordinary differential
equations on that face. Note that the arbitrary choice of $\del_x v(t,0)$ is compensated by the nonvanishing
of $\del_x w(t,0)$. As expected, the normal derivative in $x$ of the final solution $u$ is determined globally
by $\phi$ and $\psi$. 
\qed 

\medskip

We conclude this section by remarking on the higher dimensional case.  Suppose that $\Omega \subset \RR^n$ is a 
smoothly bounded region, or more generally, a smooth manifold with boundary (not necessarily compact).  We assume
that there is a well-behaved Dirichlet heat kernel $H_D$ for $\Omega$, which is the case if $\Omega$ is compact
or has boundary having reasonable asymptotics at infinity.  Consider the problem 
\[
(\del_t - \Delta_\Omega) u = 0,\ \ u(0,z) = \phi(z) \in \calC^\infty(\overline{\Omega}),\ \ u(t, y) = \psi(t,y) \in 
\calC^\infty([0,\infty) \times \del \Omega),
\]
where $y \in \del \Omega$.  An almost identical result to the one proved above remains true. Namely, define
$Q = [0,\infty) \times \overline{\Omega}$ and $Q_h$ the parabolic blowup of $Q$ at $\{0\} \times \del \Omega$.  
\begin{theorem}
Suppose that the boundary data $\phi$ and $\psi$ are smooth up to the boundaries of
their respective domains of definition. Then the unique (moderate growth) solution $u$ to the
mixed initial/Dirichlet problem with this prescribed boundary data is smooth on $Q_h$. 
\end{theorem}
The proof is essentially identical to the one above. The extra $y$-dependence is straightforward to include
in the recursion formulas above, and the sequence of equations that must be solved on the front face
is still a sequence of ordinary differential equations in which $y$ appears as a parameter. The recursion relationships
now involve derivatives in $y$ of the previously determined coefficients.

\section{Approximate solutions for the nonlinear flow}
We now begin the proof of our short-term existence result for the network flow.   As earlier in the paper, 
assume that the initial network $\Gamma_0$ comprises $n$ parametrized arcs $\gamma^{(j)}_0(x)$, $j = 1, \ldots, n$,
meeting pairwise nontangentially at the origin. Following the logic and strategy explained in the previous section for the linear 
heat equation, we first construct an approximate solution to the problem, i.e., a collection of families of parametrized 
arcs $\widehat{\gamma}^{(j)}(t,x)$, $j = 1, \ldots, m$, or more properly the lifts of these curves to the blown up spaces,
which satisfy
\begin{equation}
(\tau \del_\tau + 1 - s \del_s) \widehat{\gamma}^{(j)} = \frac{ \del_s^2 \widehat{\gamma}^{(j)}}{|\del_s \widehat{\gamma}^{(j)}|^2} + f,
\label{motion2}
\end{equation}
along with the appropriate initial conditions on the bottom and front faces and matching conditions along the
left faces. Here $f$ is smooth on $Q_h$ or $P_h$  and vanishes to all orders at every boundary component.  This step will 
be similar in spirit, but unfortunately rather more intricate, than in the linear case.

In the next sections we take up the problem of correcting this approximate solution to an exact solution. To do 
that, we will need a detour to explain the structure of the heat kernel for networks. 
For simplicity, drop both the $\ \hat{} \ $ and the superscript $(j)$ from the notation and consider the lift of a single 
curve as a map from $Q_h$ to $Z_h$. (We focus for brevity on these outer arcs, rather than the inner ones which have domain
$P_h$.)   As in \S 3, write $\gamma(t,x)$ in terms of the $(\tau,s)$ and $(\tau,w)$ coordinate systems on $Q_h$ and $Z_h$,
respectively.  Thus $\gamma = \tau \eta(\tau,s)$.  We make the ansatz that $\eta$ is smooth,
\[
\eta \sim \sum_{j=0}^\infty \tau^j \eta_j(s)
\]
near $\ff$.  Inserting this series in \eqref{motion2} and collecting like powers of $\tau$ gives the sequence of equations
\begin{equation}
\left[ \frac{\del^2_s \eta}{|\del_s \eta|^2}\right]_j +  s\del_s \eta_j - (j+1) \eta_j = 0,\ \ j = 0, 1, 2, \ldots
\label{atj}
\end{equation}
The initial expression on the left, $[ \ldots ]_j$, denotes the coefficient of $\tau^j$ in the series expansion of the
term inside the brackets.  

When $j = 0$, \eqref{atj} reduces to the soliton equation \eqref{stationary}; hence, as already explained in \S 3, once we have
chosen a particular expanding soliton, then the arcs of this soliton appear as the leading asymptotic terms of each  arc of
this approximate solution at $\ff$.   

The study of \eqref{atj} for $j \geq 1$ is conducted in a series of steps.   In the first, we study the sequence of operators
\[
\calL_j \eta_j  = \del_s^2 \eta_j + (s\del_s - (j+1)) \eta_j \, . 
\]
These operators appear as the principal parts of \eqref{atj}, and their precise role in this analysis will be explained below. 

In the following we write $A \sim B$ if $A-B$ decays exponentially as $s \to \infty$. (In fact, the `error terms' 
encountered below all decay like $e^{-cs^2}$ for some $c > 0$.)   We also say that a function $u$ 
is {\sl of degree $\ell$} if $u \sim P$ where $P$ is a polynomial of degree $\ell$. All functions below are assumed
to be smooth. 

\begin{lemma}
The space of solutions to $\calL_j u = 0$ on $\RR^+$ is spanned by two functions, $U_j, V_j$, where $U_j$ is of degree $j+1$ and $V_j \sim 0$. 
If $R$ has degree $j-1$, and $a$ is any constant, then there exists a unique solution $u$ to $\calL_j u = R$ with $u$ of degree
$j+1$ and $u(0) = a$. 

The analogous statement is true for any operator of the form $\calL_j + E_1 \del_s^2 + E_2 \del_s + E_3$, where each
$E_j \sim 0$. 
\end{lemma}
\begin{proof}
Observe first that $\calL_{-1} u = u'' + s u' = 0$ has the explicit solutions $1$ and $G(s) := -\int_s^\infty e^{-\sigma^2/2}\, d\sigma$,
so a general solution has the form $A + B G(s)$ for some constants $A, B$. 

Next, applying the intertwining formula 
\[
\del_s \calL_j = \calL_{j-1} \del_s
\] 
$(j+1)$ times to $\calL_j u = 0$ gives $\calL_{-1} u^{(j+1)} = 0$,
so $u^{(j+1)} = A + B G(s)$, and hence $u$ has degree $j+1$ as claimed. Note that if $U_j \not \equiv 0$, then its leading coefficient 
must be nonzero.

To obtain a solution of the inhomogeneous equation $\calL_j u = R$, we can use the method of undetermined coefficients
to solve away the polynomial terms of $R$, and reduce to the case where $R \sim 0$.  Assuming this has been done
for the moment, then the explicit formula 
\[
w = U(s) \int_s^\infty U(s')^{-2} \int_{s'}^\infty U(s'') R(s'') \, ds'' ds'
\]
produces a solution $w \sim 0$.

Now suppose that $u(s) \sim a_{j+1} s^{j+1} + a_j s^j + \ldots + a_0 s^0$ and write $R \sim b_{j-1} s^{j-1} + \ldots + b_0 s^0$.  Then
\[
\calL_j u \sim  \sum_{i=0}^{j-1}  a_{i+2} (i+2)(i+1) s^i + \sum_{i=0}^{j+1}  ( i - (j+1)) a_i  s^i
\sim  \sum_{i=0}^{j-1}  b_i s^i,
\]
hence
\[
 (j+1 - (j+1)) a_{j+1} = 0,\ \  (j - (j+1)) a_j = 0,\ \ (i+2)(i+1) a_{i+2} + (i - j - 1) a_i = b_i, \ \ i = 0, \ldots, j-1.
\]
The first equation leaves $a_{j+1}$ free, and the second equation gives that $a_j = 0$. The remaining equations show
that each $a_i$, $i \leq j-1$, is determined by $b_i$ and $a_{j+1}, \ldots, a_{i+1}$.  Notice that if we had assumed
that the polynomial part of this solution has leading term different than $s^{j+1}$, there would have been no cancellation
and this ansatz would not have worked. 

To attain the correct boundary value, simply add on the appropriate homogeneous solution, which by the above has degree 
at most $j+1$. 




Finally, $(1 + E_1)u'' + s (1+ E_2) u'  + (-(j+1) + E_3 )u = R$ as $\calL_j u = f$
where $f = R -E_1 u'' - E_2 s u' - E_3 u$.  The same explicit integral formula, and standard perturbation arguments,
gives the existence of solutions with the specified asymptotics. 
\end{proof}

The relevance of the operators $\calL_j$ is explained by the following
\begin{lemma}
For $j \geq 1$, equation \eqref{atj} can be written in the form
\[
\calL_j \eta_j +  E_1 \del^2_s \eta_j  + E_2 \del_s \eta_j + Q_j(\eta_0, \ldots, \eta_{j-1}) = 0, 
\]
where $Q_j$ is a polynomial in $\del_s \eta_i$ and $\del_s^2 \eta_i$, $0 \leq i \leq j-1$ and each $E_j \sim 0$
depends only on $\eta_0$.
\end{lemma}

\begin{proof} Recalling again that $'$ is the same as $\partial_s$, write 
\[
\eta''/|\eta'|^2 = F( |\eta'|^2) \eta'' \quad \mbox{with}\ \ F(w) = 1/w.
\]

By \eqref{normalization}, $|\eta_0'|^2 \sim 1$, hence
\[
|\eta'|^2 \sim 1 + \sum_{\ell=1}^\infty \tau^\ell \left(2 \eta_0' \cdot \eta_\ell' + 
\sum_{i=1}^{\ell-1} \eta_i' \cdot \eta_{\ell-i}'\right). 
\]
Expanding $F$ about $w=1$ yields
\[
  F( |\eta'|^2) \sim  \sum_{p=0}^\infty  (-1)^p \bigg(\sum_{\ell=1}^\infty (2 \eta_0' \cdot \eta_\ell' +
  \sum_{i=1}^{\ell-1} \eta_i' \cdot \eta_{\ell-i}'  ) \tau^\ell  \bigg)^p 
  \sim 1  + \sum_{\ell=1}^\infty (-2\eta_0' \cdot \eta_\ell' + B_\ell) \tau^\ell  
\]
where each $B_\ell$ is a sum of products, each summand having at least two factors, of terms $\eta_i' \cdot \eta_j'$.  In fact, a moment's reflection
reveals that each monomial in $B_\ell$ is a constant multiple of $(\eta_{r_1}' \cdot \eta_{r_2}') \ldots (\eta_{r_{2i+1}}' \cdot \eta_{r_{2i+2}}')$
where $r_1 + \ldots + r_{2i+2} = \ell$ and $i \geq 0$.   Note that because of these last two conditions, $\eta_\ell'$ does not appear in $B_\ell$, so
we can write $B_\ell = B_\ell( \eta_0', \ldots, \eta_{\ell-1}')$. 

Next, 
\begin{equation*}
  \begin{split}
F(|\eta'|^2) \eta''  & \sim  \left( 1 + \sum_{\ell=1}^\infty (-2 \eta_0' \cdot \eta_\ell' + B_\ell) \tau^\ell \right)  \big(\tau \eta_1'' + \tau^2 \eta_2'' + \ldots\big) \\
& \sim \sum_{j=1}^\infty \tau^j \left( \eta_j'' + \sum_{i=1}^{j-1} \eta_i'' (B_{j-i} - 2 \eta_0' \cdot \eta_{j-i}')\right), 
\end{split}
\end{equation*}
where $\eta_0''$ has been omitted since it decays exponentially. We read off from this that 
\[
\left[\frac{\eta''}{|\eta'|^2}\right]_j = \eta_j'' + Q_j,
\]
where $Q_j$ is a polynomial in $\eta_i'$, $0 \leq i \leq j-1$ and $\eta_i''$, $1 \leq i \leq j-1$. 
\end{proof}

\begin{lemma}
For every $j \geq 0$, $\eta_j$ has degree $j+1$. 
\end{lemma}
\begin{proof}
We shall prove this by induction, making the hypothesis that
\[
\mbox{For each $i < j$, $\eta_i$ has degree $i+1$.}
\]
We have already shown that $\eta_0(s) \sim a s$, where $a$ is a constant vector, so this hypothesis is true for $j=0$.

We first show that $Q_j(\eta_0, \ldots, \eta_{j-1})$ has degree $j-1$. We have already shown that
\[
Q_j = \sum_{i=1}^{j-1} \eta_i'' (B_{j-i} - 2 \eta_0' \cdot \eta_{j-i}').
\]
Each term in $B_{j-i}$ is a product $(\eta_{r_1}' \cdot \eta_{r_2}') \ldots (\eta_{r_{2\ell+1}}' \cdot \eta_{r_{2\ell+2}}')$ with $r_1 + \ldots + r_{2\ell+2} = 
j-i$ and $\ell \geq 0$. By induction, each $\eta_s'$ has degree $s$, $s \leq j-1$, so the total degree of $B_{j-i}$ is $j-i$; the same is true for
the other term, $\eta_0' \cdot \eta_{j-i}'$. Therefore,  the product with $\eta_i''$ has degree $i-1 + j-i = j-1$.

Next, using the intertwining formula, we see that
\[
(\calL_j  \eta_j)^{(j+1)} = \calL_{-1}  \eta_j^{(j+1)} \sim - (Q_j)^{(j+1)} \sim 0, 
\]
which implies that $\eta_j^{(j+1)} \sim A$, hence $\eta_j\sim P_{j+1}(s)$, as claimed. This completes the inductive step.
\end{proof}

Collecting and summarizing these results, we have now proved the following
\begin{proposition}\label{solution etaj}
Given any constant vectors $a, b \in \RR^2$ and an $\RR^2$-valued function $R_{j-1}$ of degree $j-1$, there exists a unique solution to 
$\calL_j \eta_j  = R_{j-1}$  such that $\eta_j(0) = b$ and $\eta_j(s) \sim a s^{j+1} + \calO(s^j)$ as $s \to \infty$. 
\end{proposition}

The last in this sequence of results concerns the polynomial structure of $\eta_j$.  We say that a polynomial is even
if it only has terms with even powers of $s$, and odd if it only has terms with odd powers of $s$.
\begin{lemma}
 If $j$ is even, then the polynomial part $P_{j+1}$ of $\eta_j$ is odd, while if $j$ is odd, then $P_{j+1}$ is even.
\end{lemma}
\begin{proof}
We prove this again by induction, assuming the statement is true for all $i < j$.  Following the construction of
the solution above, it suffices to show that each $R_{j-1}$ (the polynomial part of $Q_j$) is even or odd,
depending on the parity of $j$.  Indeed, suppose $j$ is even and we have shown that $R_{j-1}$ is
odd.  Then the recursion formula for the coefficients of $P_{j+1}$ above, and the fact that the coefficients
$b_i$ vanish when $i$ is even, show that it too involves only odd powers of $s$; similarly when $j$ is odd,
then $P_{j+1}$ can only involve even powers.

To prove the assertion about $Q_j$, we analyze each summand $\eta_i'' B_{j-i}$ and $\eta_i'' \eta_0' \cdot \eta_{j-i}'$.
Suppose $j$ is even.  If $i$ is also even, then $\eta_i$ is odd, so the same is true for $\eta_i''$.   Furthermore, each
monomial in $B_{j-i}$ is composed of factors $\eta_{r_k}'$, which are polynomials of degree $r_k$. 
The sum of these degrees adds up to $j-i$, which is even, and so there must be an even number of factors
which are odd polynomials.  Hence each summand in $B_{j-i}$ is even,  so the product with $\eta_i''$ is odd.
On the other hand, if $i$ is odd, then $\eta_i''$ is even, while $B_{j-i}$ is a sum of products of polynomials,
each either even or odd, and since the overall product is odd, there must be an odd number of odd factors,
so $B_{j-i}$ is odd and the product with $\eta_i''$ is odd.   Similar considerations apply to the other
term $\eta_i'' \eta_0' \cdot \eta_{j-i}'$.  This proves the parity assertion when $j$ is even. The corresponding
assertion when $j$ is odd is proved similarly.
\end{proof}

Here is we have accomplished so far. Proposition \ref{solution etaj} shows that once we have specified
the coefficients $b_j$ and $a_j$ giving the boundary value at $s=0$ and the leading term of the
polynomial asymptotic as $s \to \infty$, then there exists an infinite sequence of functions $\eta_j(s)$ 
such that the asymptotic sum   $\sum \tau^j \eta_j(s)$ is a formal solution of the flow equation near $\ff$.  
Said slightly differently, if we choose a smooth function $\tilde{\eta}(\tau,s)$ with this Taylor series as 
$\tau \to 0$, then $N( \tilde{\eta})$ vanishes to infinite order at $\ff$.  To complete this construction
of the approximate solution, we must show how to choose these values $a_j$ and $b_j$.   As already
seen in our treatment of the linear heat equation, the values $a_j$ will be determined by the Taylor
coefficients of the arcs $\gamma^{(\ell)}$ at $x=0$. On the other hand, the Herring conditions only
specify half of the necessary Taylor coefficients along the left face of each $Q_h$ (or $P$), and hence
we will need to specify the other parts of the Taylor series (which we can do arbitrarily). 

Consider first the issue of matching coefficients at the intersection of the front and bottom faces.  In
this region we use the coordinates $T = t/x^2 = 1/2s^2$ and $y = x = s\tau$, and also revert to $\gamma = \tau \eta$. 
Then 
 \begin{equation*}
\begin{split}
\gamma(\tau, s) & \sim  \sum_{j=0}^\infty \tau^{j+1} \eta_j(s) \sim 
\sum_{j=0}^\infty  \sum_{i = 0}^{j+1}  a_{ij} s^{j+1- i} \tau^{j+1} \\ 
 & \sim  \sum_{j=0}^\infty \sum_{i=0}^{ [(j+1)/2]}  a_{ij}   (2T)^{- (j+1-2i)/2} ( y 2T^{1/2})^{j+1}  \sim
\sum_{j=0}^\infty \sum_{i=0}^{ [(j+1)/2]}  2^i a_{ij}   T^{i } y^{j+1}. \end{split}
\end{equation*}

Now dropping the superscript, each initial curve $\gamma_0(x) = \gamma_0^{(\ell)}(x)$ is smooth for $x \geq 0$, so in the $(y,T)$ coordinates,
\[
\gamma_0(x) \sim  \sum_{j=0}^\infty  A_j y^j.
\]
It is thus necessary that $a_{0j}$, the coefficient of $s^{j+1}$ in $\eta_j$, equals $A_j$.  As we have showed earlier, computing in
the $(\tau,s)$, the rest of the terms $a_{ij}$, $i \geq 0$, are determined by this leading coefficient.   Establishing the equivalent
recursion in the $(T,y)$ coordinates is more difficult, but is not necessary simply because the recursion is determined by
the underlying equation, and hence must come out the same in either coordinate system.

Now let us return to the expansion along the left face, where $s = 0$ and $\tau \geq 0$.  In the following, recall that
along this face we impose the boundary conditions that three curves meet at equal angles. For simplicity, call
these three curves $\eta^{(1)}$, $\eta^{(2)}$ and $\eta^{(3)}$.  The relevant endpoint of each $\eta^{(j)}$ is specified
by four real numbers (the coordinates of the position and tangent vector at that point), so the full space of parameters for these 
endpoints for the  three curves is $12$-dimensional.  The Herring conditions impose $6$ (nonlinear) constraints,
namely that the endpoints of these curves coincide (four real conditions) and that the unit tangents at this point of coincidence 
sum to $0$ (two further conditions).  We now describe the remaining six free parameters.   Two are the coordinates of
the vertex where the three curves come together, and three are the lengths of the tangents $(\eta^{(j)})'(0)$ at this point.
For the remaining coordinate, observe that, up to reflection, the normalized (unit length) triod is completely determined 
by specifying the direction of $(\eta^{(1)})'(0)$, and this in turn is determined by the angle it makes with some 
fixed direction (e.g.\ the positive $x_1$ axis). This angle might be called the phase of the triod.   

To formalize this, consider a collection of six vectors $v_1, v_2, v_3, w_1, w_2, w_3 \in \RR^2$, and write $\bf{v}$ and $\bf{w}$ for
the two triples of vectors respectively.  Now define
\[
F: \{ ( \mathbf{v}, \mathbf{w} ): |w_i| \neq 0, \ i = 1, 2, 3\}  \longrightarrow  \RR^6 \times \RR^2 \times (\RR^+)^3 \times S^1
\]
\[
F( \mathbf{v}, \mathbf{w} )  =  \big( v_1 - v_2, v_2 - v_3, \sum w_i/|w_i|, v_1, |w_1|, |w_2|, |w_3|,  \arccos (w_1 \cdot e_1/ |w_1|) \big). 
\]
The coordinate $\theta$ is only defined locally, where $\cos \theta = w_1 \cdot e_1/|w_1|$, but there is 
a local branch of this map which suffices for our purposes.  A straightforward calculation shows that this 
map is a local diffeomorphism.    

The Herring conditions are equivalent to 
\[
F( (\eta^{(1)}(0), \eta^{(2)}(0), \eta^{(3)}(0), (\eta^{(1)})'(0), (\eta^{(2)})'(0), (\eta^{(3)})'(0)) \in \{0\} \times (\RR^+)^3 \times S^1 \times \RR^2
\]
Now choose any smooth mappings $\beta(\tau), \zeta_1(\tau), \zeta_2(\tau), \zeta_3(\tau), \theta(\tau)$ 
where $\zeta_i(0) = |(\eta^{(i)}_0)'(0)|$ and the initial angle $\theta(0)$ that $(\eta^{(1)}_0)'(0)$ makes with the $e_1$ direction
are determined by the fixed soliton solution at time $\tau = 0$, and $\beta(\tau)$ represents the common vertex at time $\tau$. 
Then $F^{-1}( 0, 0, 0, \beta(\tau), \zeta_1(\tau), \zeta_2(\tau), \zeta_3(\tau), \theta(\tau))$ specifies the full
Cauchy data of the three curves $\eta^{(i)}(s)$ at $s=0$. 

Having made these choices of complementary data, it is now easy to construct the entire Taylor series of each $\eta^{(i)}(\tau,s)$ 
at $s=0$ for $0 \leq \tau < \tau_0$ using \eqref{liftedeq1}.  In particular,  this determines all of the constant initial 
values $b_j$ for $\eta_j(0)$. 

Altogether, we have determined the full Taylor series at all faces of each $Q_h$, and of each $P$, for any putative solution 
to the full nonlinear equation with coupled to the Herring boundary conditions, and also imposing the extra complementary 
boundary conditions as above.    We now use Borel's lemma to choose functions $\widehat{\gamma}^{(i)}$ which have 
precisely these Taylor series at all boundaries.   These constitute the evolving network $\widehat{\Gamma}(t)$. 
By construction,  $N( \widehat{\Gamma}) := f$ is $\calC^\infty$  and vanishes to all orders at all boundary faces.

\section{Linearization}\label{subsection:linearized operator} 
We next study the linearization of \eqref{evolution3} around the approximate solution $\widehat{\Gamma}$ constructed in
the last section. By the naturality of the process of computing linearizations, we may either compute the linearization
of the equation on the unblown up spaces and then lift to the blowups, or else first lift the nonlinear equation and
then compute blowups.  We do the former.

First suppose that $\Gamma_\kappa = \{ \gamma^{(j)}_\kappa\}$ is a family of networks, i.e., for each $\kappa$, $\Gamma_\kappa$ a
is a time-dependent network which does not necessarily satisfy the flow equations.  Differentiating  \eqref{evolution3}  with
respect to $\kappa$ leads to the operator
\begin{equation*}
\del_t u^{(j)} =  \frac{\del_x^2 u^{(j)}}{|\del_x \gamma^{(j)}|^2} -  2 \frac{ \langle \del_x \gamma^{(j)}, \del_x u^{(j)}\rangle}{ |\del_x \gamma^{(j)}|^4}
\del_x^2 \gamma^{(j)}, 
\end{equation*}
where $u^{(j)} = \del_\kappa\gamma^{(j)}|_{\kappa = 0} $.  The linearized boundary conditions are
\begin{equation*}
u^{(i)}(t,0) = u^{(j)}(t,0)\ \ \mbox{for all}\ i \neq j, \quad
\sum \left(\frac{ \del_x u^{(j)}}{|\del_x \gamma^{(j)}|} - \frac{\langle \del_x \gamma^{(j)}, \del_x u^{(j)}\rangle}{|\del_x \gamma^{(j)}|^3} \del_x \gamma^{(j)}\right) (t,0) = 0, 
\end{equation*}
where the sum is over the triples associated to any interior vertex.   
The latter boundary condition can be rewritten in the more revealing form
\begin{multline*}
\sum |\del_x \gamma^{(j)}(t,0)|^{-1} P^{(j)}(\del_x u^{(j)}(t,0)) = 0, \ \ \mbox{or more simply}\ \ \ \sum c_j(t) P^{(j)}( \del_x u^{(j)}(t,0)) = 0
\end{multline*}
where $c_j(t) = |\del_x \gamma^{(j)}(t,0)|^{-1}$ and $P^{(j)}$ is the orthogonal projection onto the unit normal to $\gamma^{(j)}$.  The interior and boundary
operators will be written as $\del_t - \calL$ and $\beta$, respectively. It will be useful to separate the boundary operator into
two separate terms, $\beta( u, \del_x u) = (\beta_0(u), \beta_1(\del_x u))$, where the first only depends on $u(t,0)$ 
and the second only depends on $\del_xu(t,0)$.  We give a brief derivation of the ellipticity of these boundary conditions later.
We will consider the lifts of these operators, or rather the operators $t (\del_t - \calL)$ and $x\beta$ to each $Q_h$ and $P$.

Our goal is to solve the nonlinear parabolic equation $\bbM(\widehat{\Gamma} + U) = 0$ and $B(\widehat{\Gamma} + U) = 0$. 
The operators $\bbM$ and $B$ are simply the full nonlinear operators corresponding to the parametrized flow and the Herring
boundary conditions, respectively. Expanding around the approximate solution, this becomes 
\[
(t\del_t - t \calL) U = f + Q_i(U), \quad \beta(U) = S_b(U),
\]
where $f = - \bbM(\widehat{\Gamma})$ is the error term and $Q_i$ and $S_b$ are the quadratically vanishing remainders.  By construction, $f$ 
vanishes to all orders at all boundaries of $Q_h,  P_h$, and similarly, $B(\widehat{\Gamma}) = 0$.  We shall prove the existence
and uniqueness of a solution $U$ which vanishes rapidly as $t \to 0^+$, and by the usual contraction mapping arguments, to do so it 
suffices to solve the linear system
\begin{equation}
(\del_t - \calL) U = H, \quad \beta(U) = K
\label{fulllinsys}
\end{equation}
with good estimates, where $H(t,x)$ and $K(t)$ both vanish to all orders in $t$, or equivalently, along the front and bottom faces.
Notice that we may as well assume that the correction term $U$ satisfies the matching condition $U^{(i)}(t,0) = U^{(j)}(t,0)$, and
the nonlinear part of the boundary term $S_b(U)$ only involves the second component of the nonlinear boundary operator $B(U)$.
Therefore, it suffices to consider boundary conditions of the form $\beta(U) = (\beta_0(U), \beta_1(\del_x U)) = (0, k)$, but 
since it is no harder, we consider the full system where $\beta_0(U) \neq 0$ as well. 

Our aim is to show that if the inhomogeneous terms $H$ and $K$ in \eqref{fulllinsys} are smooth and decay to some high order 
as $t \to 0$, then there exists a solution $U$ to this system with these same properties. We focus on obtaining solutions
with appropriate decay first; the higher regularity is standard and will be discussed at the end.  The main step is to recast 
this problem by showing that 
\begin{equation}
t(\del_t - \calL):  t^N \calX \to t^{N} \calY + t^{N}\calZ
\label{F0}
\end{equation}
is an isomorphism for all $N \gg 1$, where $\calX$, $\calY$ and $\calZ$ are appropriately defined weighted 
$L^2$-based Sobolev spaces (thus $U \in t^N \calX$, $H \in t^{N} \calY$ and $K \in t^N \calZ$). 
The uniqueness part of this shows that if $H$ and $K$ vanishes to all orders, then so does $U$. Writing 
$H = t^{N}h$, $K = t^{N} k$, and $U = t^N u$, then the problem is equivalent to showing that
\begin{equation*}
(t\del_t + N - t\calL) u = h,\ \ \beta(u) = k.
\end{equation*}
is an isomorphism for $N$ large. 

The space $\calZ$ is simply $L^2(Q_h; \frac{dt}{t} dx)$.  Away from the boundaries, elements of $\calX$ lie
in the parabolic Sobolev space $H^{1,2}$, but the precise weight factors at each boundary face of $Q_h$ are
not obvious at the outset. We also impose fractional Sobolev regularity on $k_0$ and $k_1$.  To show that \eqref{F0} 
is an isomorphism, we prove the a priori estimate
\begin{equation}
||u||_{\calX} \leq C ( ||h||_{ \calY} + ||k||_{\calZ}).
\label{F02}
\end{equation}
Concealed is the fact that the norms depend on $N$, again as spelled out carefully below. Given \eqref{F02}, a standard continuity 
argument implies the unique solvability of the original problem, and hence that \eqref{F0} is an isomorphism. 
We defer a precise statement of the estimate until the end of the proof.

This estimate is not standard because of the behavior of the coefficients of $\calL$ at the front faces of $Q_h$.
To handle this, we consider the lift of this operator to the blowup. We follow in broud outline a proof of a priori
estimates given by Krylov \cite{Krylov}, adapted to the present situation.  By standard localization and approximation 
arguments, it is enough to prove a priori estimates for each of the following simpler model operators near the different 
faces of each $Q_h$. In the following, we use either the  original coordinates $(t,x)$
or else the projective coordinates $(\tau,s)$. 
\begin{itemize}
\item[a)] $\tau\del_\tau + 2N - \del_s^2  - s\del_s$ on functions supported on the region $0 < \tau < T$, $s_0 < s < \infty$,
i.e., near $\ff \cap \bff$; 
\item[b)] $t\del_t + N - t\del_x^2$ on smooth functions compactly supported on $0 < t < T$, $0 < c_1 < x < c_2$, i.e., 
near $\bff$ but away from $\ff$; 
\item[c)] $\tau\del_\tau + 2N - \del_s^2$ on smooth functions on three copies of $Q_h$ joined along $\lf$ where $0 < \tau < T$ 
$0 \leq s_i < s_0$, $i =  1, 2, 3$, and satisfying the boundary conditions $\beta(u) = k$, i.e., near $\ff \cap \lf$. 
\end{itemize}
We call these regions I, II and III, respectively. These are models in the sense that the actual lifted operator in each of these
regions can be considered as a lower order perturbation of these models, and those lower order terms are readily absorbed into
the estimates below. The appearance of $2N$ rather than $N$  in a) and c) is because we are conjugating 
by $t^N = \tau^{2N}$, but we can obviously replace $2N$ by $N$ here. In each of these cases, by density arguments, 
it suffices to prove the estimate for smooth compactly supported functions whose support does not intersect $t=0$ or $\tau = 0$. 

\medskip

\noindent {\bf Region I:}  First suppose that $u$ is supported near the intersection of the front and bottom faces.  Although
the nondegenerate coordinate system here are $T = t/x^2$ and $y = x$, it is simpler to derive estimates using the
$(\tau,s)$ coordinates. An auxiliary argument, which gives the same estimate but in the coordinates $(t,x)$, is included
here to show how to join this estimate to the one in region II. 

Suppose then that 
\begin{equation}
\tau\del_\tau u - \del_s^2u - s\del_s u + Nu = \tau^2 h.
\label{F1} 
\end{equation}
First multiply both sides of \eqref{F1} by $Nu$ and integrate in both variables with respect to the measure $d\tau ds$. 
Given the support properties of $u$,  integrations by parts leave no boundary terms. We obtain the identities:
\[
N \int  \tau (\del_\tau u) u \,d\tau ds = \frac{N}{2}\int \tau\del_\tau u^2 \,d\tau ds = -\frac{N}{2} \int u^2 \,d\tau ds,
\]
\[
-N \int u \del_s^2 u  \,d\tau ds= N \int (\del_s u)^2 \,d\tau ds,
\]
and 
\[
-N \int su \del_s u  \,d\tau ds= -\frac{N}{2} \int s \del_s u^2  \,d\tau ds= \frac{N}{2} \int u^2 \,d\tau ds.
\]
Using these, and applying the Cauchy-Schwarz and arithmetic/geometric mean inequality on the right hand side, we arrive at
\begin{equation}
N \int (\del_s u)^2  \,d\tau ds+ N^2 \int u^2  \,d\tau ds\leq C \int \tau^4 h^2 \,d\tau ds. 
\label{F2}
\end{equation}

Next, multiply \eqref{F1} by $s^2 u$ and integrate with  respect to the same measure to get
\[
\int \tau\del_\tau u (s^2 u) - \del_s^2 u (s^2 u) - s\del_s u (s^2 u) + N s^2 u^2   \,d\tau ds = \int \tau^2 h s^2 u \,d\tau ds.
\]
Here we observe that
\[
\int (\tau \del_\tau u) s^2 u  \,d\tau ds= \frac12 \int \tau \del_\tau (u^2) s^2 \, d\tau ds =  -\frac12 \int s^2 u^2 \,d\tau ds, 
\]
\[
- \int (\del_s^2 u) (s^2 u)  \,d\tau ds= \int (\del_s u) (s^2 \del_s u + 2s u)  \,d\tau ds= \int s^2 (\del_s u)^2  \,d\tau ds- \int u^2 \,d\tau ds,
\]
and
\[
- \int s^3 u \del_s u  \,d\tau ds= -\frac12 \int s^3 \del_s(u^2)  \,d\tau ds= \frac32 \int s^2 u^2 \,d\tau ds.
\]
Putting these together and estimating the right hand side the  same way gives
\begin{equation}
N \int s^2 u^2  \,d\tau ds+ \int s^2 (\del_s u)^2 \,d\tau ds \leq C \int s^2 \tau^4 h^2  \,d\tau ds + \int u^2 \,d\tau ds \leq C \int s^2\tau^4 h^2  \,d\tau ds
\label{F3}
\end{equation}
since $s \geq s_0 \geq 1$. 

Now rewrite the equation as
\[
(\tau\del_\tau - \del_s^2)u = (s\del_s u - Nu + \tau^2 h).
\]
Squaring both sides and integrating, we get
\[
\int \tau^2 (\del_\tau u)^2 + (\del_s^2 u)^2 - 2 \tau\del_\tau u \del_s^2 u  \,d\tau ds\leq 
C \int s^2 (\del_s u)^2 + N^2 u^2 + \tau^4 h^2 \,d\tau ds.
\]
The cross-term equals
\[
-2 \int \tau\del_\tau u \del_s^2 u  \,d\tau ds= 2 \int \tau\del_\tau (\del_s u)  \del_s u \,d\tau ds =
\int {\color{green}\tau}\del_\tau (\del_s u)^2 \, d\tau ds= - \int (\del_s u)^2 \,d\tau ds. 
\]
Thus using \eqref{F2} and \eqref{F3}, we have 
\begin{equation}
\int \tau^2 (\del_\tau u)^2 + (\del_s^2 u)^2  \,d\tau ds \leq C \int s^2 \tau^4 h^2 \,d\tau ds.
\label{F4}
\end{equation}

Finally, \eqref{F2}, \eqref{F3} and \eqref{F4} together prove that
\begin{equation}
\int (\tau \del_\tau u)^2 + (\del_s^2 u)^2  \,d\tau ds+ \int s^2 (\del_s u)^2 + N^2u^2   \,d\tau ds+ N \int s^2 u^2  \,d\tau ds
\leq C \int s^2 \tau^4 h^2 \,d\tau ds.
\label{F5}
\end{equation}
This completes the main estimate in this region.  

We now translate this estimate to the $(t,x)$ coordinate system. As explained earlier, this is necessary to show how to 
join this estimate to the one for region II. We have
\[
d\tau ds = \det \begin{pmatrix}  1/\sqrt{2t}  & 0 \\ -x/2\sqrt{2}t^{3/2} & 1/\sqrt{2t} \end{pmatrix}  dt dx = \frac{1}{2t} dt dx \quad \mbox{and}
\]
\[
\tau \del_\tau = 2t\del_t + x \del_x, \quad s\del_s = x \del_x, \quad \del_s^2 =2 t \del_x^2,
\]
and therefore, \eqref{F5} is the same as
\[
\int \left[ ( (2t\del_t + x \del_x)u)^2 + 4t^2 (\del_x^2 u)^2 + (x\del_x u)^2
+N (x^2/2t) u^2  + N^2 u^2 \right] \, \frac{dt dx}{2t}
\leq C \int x^2 h^2 \, dt dx,
\]
or equivalently, 
\begin{equation}
\int \frac{1}{t} (2t\del_t u + x \del_xu)^2 + t (\del_x^2 u)^2 + \frac{x^2}{ t} (\del_x u)^2 + (N \frac{x^2}{t^2}  +  N^2 \frac{1}{t})u^2 \, dtdx
\leq C \int x^2 h^2 \, dt dx. 
\label{E0}
\end{equation} 

We now prove this latter estimate by similar means.  It suffices now, by approximation,  to consider the simple 
model equation $(t\del_t + N - t\del_x^2) u = th$.  For simplicity, we suppose that $u$ is supported  in the region where $t/x^2 \leq 1$.

First multiply the equation by $x^2 u/t^2$ and integrate to get
\[
\int \frac{x^2}{t} u \del_t u +  N\frac{x^2}{t^2} u^2  - (\del_x^2 u) \left(\frac{x^2}{t} u\right) \, dt dx  = \int \left(\frac{x^2}{t}\right) h u \, dt dx.
\]
The first term on the left equals
\[
\frac{1}{2} \int \frac{x^2}{t} \del_t (u^2) \, dt dx= \frac12 \int \frac{x^2}{t^2} u^2\, dt dx.
\]
The third term on the left becomes
\[
\int \frac{1}{t} \del_x u  \del_x ( x^2 u) \, dt dx= \int \frac{x^2}{t} (\del_x u)^2 + \frac{2x}{t} u \del_x u \, dt dx= \int \frac{x^2}{t} (\del_x u)^2 - \frac{1}{t} u^2 \, dt dx..
\]
The right hand side of this equality is estimated by
\[
\frac12\int \frac{x^2}{t^2} u^2 \, dt dx+ \int x^2 h^2\, dt dx.
\]
Altogether, these prove that
\begin{equation*}
\int  \frac{x^2}{t} (\del_x u)^2 + N \frac{x^2}{t^2} u^2\, dt dx \leq  \int x^2 h^2\, dt dx + \int \frac{1}{t} u^2\, dt dx.
\end{equation*}

Now multiply the equation by $N u/t$ and integrate, to obtain
\[
\int N u \del_t u +  N^2 \frac{1}{t} u^2 + N(\del_x u)^2 \, dt dx= \int N h u\, dt dx \leq C \int x^2 h^2 + \frac{1}{4} \int N^2 \frac{1}{x^2} u^2\, dt dx.
\]
The first term integrates to $0$. Furthermore, in this region, $(1/x^2) = (1/t)(t/x^2) \leq 1/t$, so we can move the second term on the right 
over to the left, leading to
\begin{equation}
\int N(\del_x u)^2  + N^2\frac{1}{t} u^2 \, dt dx \leq C \int x^2 h^2\, dt dx.
\label{E2}
\end{equation}

For the next step, rewrite the equation as $\del_t u - \del_x^2 u = h - (N/t) u$, square both sides and multiply by $t$. This gives
\[
\int t (\del_t u)^2 + t (\del_x^2 u)^2 - 2t \del_t u \del_x^2 u \, dt dx= \int t h^2 - 2N hu + \frac{N^2}{t} u^2\, dt dx.
\]
The third term on the left can be transformed as
\[
2 \int t \del_t \del_x u \del_x u \, dt dx= \int t \del_t  (\del_x u)^2\, dt dx = - \int (\del_x u)^2\, dt dx.
\]
For the first term on the right, use $t \leq x^2$. The middle term on the right can be estimated by
\[
\int N^2 \frac{1}{x^2} u^2 + \int x^2 h^2 = \int N^2 (1/t)(t/x^2) u^2 + \int x^2 h^2 \leq \int N^2 \frac{1}{t} u^2 + \int x^2 h^2
\leq C \int x^2 h^2,
\]
using \eqref{E2}. 

Altogether, this step shows that
\begin{equation*}
\int t (\del_t u)^2 + t (\del_x^2 u)^2 \, dt dx\leq C \int x^2 f^2 + \int (\del_x u)^2 \, dt dx \leq C' \int x^2 h^2 \, dt dx
\end{equation*}
the last inequality following from \eqref{E2}.

We have now estimated almost all the terms on the left in \eqref{E0}.  To complete the proof, note that
\[
\int \frac{1}{2t} (2t\del_tu + x \del_x u)^2 \, dt dx= \int 2t (\del_t u)^2 + \frac{x^2}{2t} (\del_x u)^2 + 2 \del_t u x \del_x u\, dt dx.
\]
The first and second terms have already been estimated, while 
\[
\int \del_t u (x \del_x u) \, dt dx\leq \int  t (\del_t u)^2 + \frac{x^2}{t} (\del_x u)^2\, dt dx
\]
and both of these terms have been estimated.  This proves
\begin{equation}
\int \left[ t (\del_t u)^2 + t (\del_x^2u)^2 + (N + \frac{x^2}{t}) (\del_x u)^2 + (N \frac{x^2}{t^2} + N^2\frac{1}{t}) u^2 \right]\, dt dx \leq
\int x^2 h^2\, dt dx.
\label{E4} 
\end{equation}

\medskip

\noindent{\bf Region II:}  The previous calculations have been written out carefully, but these all adapt readily when $0 < c \leq x \leq C$. 
Indeed, in such a region, \eqref{E4} becomes
\begin{equation*}
\int \left( t (\del_t u)^2 + t (\del_x^2 u)^2 + (N + \frac{1}{t})(\del_x u)^2 + (N/t)(N + 1/t) u^2 \right) \, dt dx
\leq C \int h^2 \, dt dx.
\end{equation*}
This is not a particularly familiar estimate, but is easy enough to derive. One proceeds through essentially the same steps as
above for the equation $(\del_t + (N/t) - \del_x^2 )u = h$: first multiply by $N u$ and integrate, then multiply by $u/t$, and
finally, square the equation $(\del_t - \del_x^2) u = h - (N/t) u$, multiply by $t$ and integrate; in each case one has
to apply Cauchy-Schwarz judiciously. 

\medskip

\noindent {\bf Region III:}  Finally consider a neighbourhood of the intersection of the left and front faces. We use coordinates $(\tau, s)$ 
but now must take the linearized Herring boundary conditions and the interaction between the different evolving arcs into account.
So consider three separate functions $U = (u^{(j)})$, one for each of the three intersecting arcs, and each supported in $s \leq 2$. 
(This use of $U$ is different than the  notation earlier in this section.) 
The model operator in this region is $\tau\del_\tau + N - \del_s^2$, and we write $(\tau\del_\tau + N - \del_s^2)u^{(j)} = \tau^2 h^{(j)}$
and $\beta(U) = K = (k_0, k_1)$, where each $u^{(j)}$ (and hence each $h^{(j)}$ and $K$ as well) vanishes to all orders as $\tau \to 0$, but is 
nonvanishing at $x=0$.   

We next choose a decomposition $u^{(j)} = v^{(j)} + w^{(j)}$, or $U = V + W$, as follows.  The functions $v^{(j)}$ are chosen to
satisfy $(\tau \del_\tau + N - \del_s^2) v^{(j)} = \tau^2 h^{(j)}$, $v^{(j)}(\tau,0) = 0$ (and, say, $v^{(j)}(\tau, s) = 0$ for $s \geq 2$ as well). 
The $w^{(j)}$ are then chosen so that $(\tau\del_\tau +N - \del_s^2)w^{(j)} = 0$ and $\beta( W) = K - \beta(V)$.  
Although $W$ is not supported in $s \leq 1$, we can cut both it and $V$ off to have support in $s \leq 3/2$, and estimate 
the terms coming from the cutoff functions using the earlier estimates for Region I.    The estimates for $V$ and $W$ are handled differently.

We assume for simplicity that the coefficients $c_j(\tau) = |\del_s\gamma^{(j)}(\tau,0)|^{-1}$ of the projectors in the definition of $\beta$ are 
constant; the general case can be handled by approximation. The steps in the estimates for each $v^{(j)}$
are essentially the same as those in Region I, but slightly simpler.  First multiply the equation by $N v^{(j)}$ and integrate to get the estimate 
\eqref{F2}.   The boundary terms in the various integrations by parts all vanish since 
each $v^{(j)}(\tau,0) = 0$.  Next write the equation as  $(\tau\del_\tau - \del_s^2) v^{(j)} = (\tau^2 h^{(j)} - N v^{(j)})$ 
square both sides and integrate, just as before. We handle the cross-term just as in region I. 
This proves that
\begin{equation}
\int (\tau \del_\tau v^{(j)})^2 + (\del_s^2 v^{(j)})^2 + N (\del_s v^{(j)})^2 + N^2 (v^{(j)})^2\, d\tau ds \leq C \int \tau^4 |h^{(j)}|^2\, d\tau ds.
\label{G1}
\end{equation}

Finally, let us turn to the estimates for $W = (w^{(j)})$. Recall the Mellin transform and its inverse
\[
\eta(\tau) \mapsto \calM(\eta)(\mu) := \tilde{\eta}(\mu) = \int_0^\infty \tau^{i\mu} \eta(\tau)\, \frac{d\tau}{\tau}, \qquad 
\eta(\tau) = \frac{1}{2\pi} \int \tilde{\eta}(\mu)\, d\mu.
\]
Taking the Mellin transform of the equation gives
\[
 (-\del_s^2 + (N + i\mu)) \tilde{w}^{(j)}(\mu,s) = 0,
\]
so, throwing out the exponentially growing solution, we find that
\begin{equation*}
\tilde{w}^{(j)}(\mu, s) = \tilde{b}_{j} e^{-\sqrt{N + i\mu} \, s}
\end{equation*}
for some $\tilde{b}_{j}(\mu)$ to be determined.  Note that $\tilde{w}^{(j)}(\mu,0)  = \tilde{b}_{j}$ and $\del_s \tilde{w}^{(j)}(\mu,0) = -\sqrt{N+i\mu} \, \tilde{b}_j$,
so if we set $b_{j}(\tau) = \calM^{-1}(\tilde{b}_{j})$, then 
\[
w^{(j)}(\tau,0) = b_j, \qquad \del_sw^{(j)}(\tau,0) = S_N b_j,
\]
where $S_N b := \calM^{-1} ( \sqrt{N+i\mu} \, \tilde{b}(\mu))$ is an invertible elliptic pseudodifferential operator of order $1/2$. 

Replacing $K - \beta(V)$ by $K$ for simplicity, the boundary conditions are:
\begin{align*}
\beta_0(W) & = (w^{(2)}(\tau,0) - w^{(1)}(\tau,0), w^{(3)}(\tau,0) - w^{(1)}(\tau,0)) = k_0 = (k_0', k_0''),\ \ \ \mbox{and}, \\
\beta_1(W) & = \sum_{j=1}^3 P^{(j)} \del_s w^{(j)}(\tau,0) = k_1. 
\end{align*}
Using these, we can determine the $b_j$ from $K$ as follows. Write $w^{(1)}(\tau,0) = b_1$ and $k_0 = (k_0',k_0'')$ (these primes do not denote
differentiation) so that $w^{(2)}(\tau,0) = b_2 = k_0' + b_1$, and 
$w^{(3)} (\tau,0)= b_3 = k_0'' + b_1$.   On the other hand, by a short calculation, the second boundary condition determines the unknown coefficient $b_1$ by
the relationship
\begin{equation*}
\mathbb P (b_1) =  S_N^{-1} k_1 - P_2 k_0' - P_3 k_0',
\end{equation*}
where $\mathbb P = \sum c_j P^{(j)}$. (Recall that $c_j = |\del_s \gamma^{(j)}(0,0)|^{-1}$.) To solve the problem completely, we need to 
show that $\mathbb P$ is invertible.  (This invertibility is of course precisely the Lopatinski-Shapiro condition for the linearized Herring 
boundary condition.)  In the special case where all $c_j = 1$, an explicit calculation using that the $P^{(j)}$ project onto three unit 
vectors making a mutual angle $2\pi/3$ yields
\[
\mathbb P = \begin{pmatrix}  3/4 & 0 \\ 0 & 3/4 \end{pmatrix}.
\]
For the general case, if some or all of the $c_j\neq 1$, we note that $\mathbb P = \sum c_j (P^{(j)})^2$, so $\mathbb P b_1 = 0$ implies
\[
\langle \mathbb P b_1, b_1\rangle = \sum c_j | P^{(j)} b_1|^2 = 0,
\]
hence $b_1 = 0$. 

\begin{remark}
The invertibility of the matrix $\mathbb P$ is precisely the Lopatinski-Shapiro condition for this boundary problem. 
\end{remark}

It remains to estimate the $L^2$ norm of $W$ and its derivatives.   We first carry this out assuming that $N = 1$, and after that obtain the final
estimates by a rescaling argument.   Let us focus on any one of the components $w^{(j)}$, and drop the superscript $j$ for the moment.
In the following we shall use the measures $d\tau/\tau$ and $(d\tau/\tau)ds$, which are more natural when using the Mellin
transform. We can convert back to the measures $d\tau$ and $d\tau  ds$ at the end of the proof if we replace
$w$ by $\tau^{1/2} w$, $N$ by $N+\frac12$, etc. 

By the Plancherel theorem, 
\begin{equation*}
||w||_{L^2}^2  = \int_{-\infty}^\infty \int_0^\infty |\tilde{w}(\mu,s)|^2 \, d\mu ds \leq C \int_{-\infty}^\infty  |\tilde{b}(\mu)|^2 
\left(\int_0^\infty \left|e^{- s \sqrt{1+i\mu}}\right|^2 \, ds\right) \, d\mu. 
\end{equation*}
Now $1 + i\mu = Re^{i\theta}$, where $R = \sqrt{1+\mu^2}$ and $\tan \theta = \mu$, so 
\[
\left| e^{-s \sqrt{1+i\mu}}\right|^2 = \exp \left(-2s (1+\mu^2)^{1/4} \cos(\theta/2)\right) \leq C \exp \left( -C s (1+\mu^2)^{1/4}\right),
\]
since $\cos(\theta/2) \to 1/\sqrt{2}$ as $|\mu| \to \infty$.  Integrating in $s$, we conclude that 
\[
||w||_{L^2}^2 \leq C \int_{-\infty}^\infty  |\tilde{b}(\mu)|^2 (1 + \mu^2)^{-1/4}\, d\mu = ||b||_{-1/4}^2,
\]
where $|| \cdot ||_{s}$ denotes the $H^s_b$ norm of the function $b(\tau)$.  For completeness, recall that 
\[
\|b||^2_s = \int |\tilde{b}(\mu)|^2 (1 + \mu^2)^s\, d\mu.
\]
In the special case $s = \ell \in \mathbb N$, 
\[
||b||_{\ell}^2 = \sum_{i = 0}^\ell \int |(\tau\del_\tau)^i b|^2\, \frac{d\tau}{\tau}.
\]
Since $b$ is a linear combination of $S_1^{-1}k_1$ and $k_0$, we conclude finally that
\[
||w||_{L^2}^2 \leq C( ||k_0||_{-1/4}^2 + ||k_1||_{-3/4}^2).
\]

A similar calculation can be done to estimate the $L^2$ norms of $\del_s w$, $\del_s^2 w$ and $\tau\del_\tau w$, leading  to the estimate
\begin{equation*}
\int |\tau\del_\tau w|^2 + |\del_s^2 w|^2 + |\del_s w|^2 + |w|^2  \leq C ( ||k_0||_{3/4}^2 + ||k_1||_{1/4}^2).
\end{equation*}
(The estimates for $w$ and $\del_s w$ require norms of $k_0$ and $k_1$ which are $1$ and $1/2$ orders weaker, respectively.) 

Now consider how these estimates depend on $N$.  Changing variables by $\hat{\tau} = \tau^N$
and $\hat{s} = \sqrt{N} s$ gives $\tau\del_\tau - \del_s^2 = N (\hat{\tau}\del_{\hat\tau} + 1 - \del_{\hat{s}}^2)$.
For want of better notation, write ${\bf w}(\hat{\tau}, \hat{s}) = w( \hat{\tau}^{1/N}, N^{-1/2} \hat{s})$. Then
\[
\widehat{\calL} {\bf w} (\hat{\tau}, \hat{s}) =  (\calL w)( \hat{\tau}^{1/N}, N^{-1/2} \hat{s}) = 0,
\]
and in addition
\[
{\bf w}(\hat{\tau}, 0) = b( \hat{\tau}), \quad  \del_{\hat{s}} {\bf w} (\hat{\tau}, \hat{s}) = N^{-1/2} (\del_s w)( \hat{\tau}^{1/N}, N^{-1/2} \hat{s}).
\]
Hence the boundary data appropriate for the problem in the $(\hat{\tau}, \hat{s})$ coordinate system is:
\[
\beta( {\bf W}) = ( {\bf k}_0(\hat{\tau}^{1/N}), {\bf k}_1(\hat{\tau}^{1/N})) = ( k_0( \hat{\tau}^{1/N}),  N^{-1/2} k_1(\hat{\tau}^{1/N}) ).
\]

It remains to compute how the various norms change under this rescaling.  For this we first note that
\begin{multline*}
\frac{d\hat{\tau}}{\hat{\tau}} d\hat{s} = N^{3/2} \frac{d\tau}{\tau} ds, \quad \del_{\hat{s}}{\bf w}( \hat{\tau}, \hat{s}) =N^{-1/2} (\del_s w)( \hat{\tau}^{1/N}, N^{-1/2} \hat{s}), 
\\ \del_{\hat{s}}^2{\bf w}( \hat{\tau}, \hat{s}) =N^{-1} (\del_s^2 w)( \hat{\tau}^{1/N}, N^{-1/2} \hat{s}),\ 
\hat{\tau}\del_{\hat{\tau}}{\bf w}( \hat{\tau}, \hat{s}) =N^{-1} ( \tau\del_\tau w)( \hat{\tau}^{1/N}, N^{-1/2} \hat{s}).
\end{multline*}
Thus
\[
\int |{\bf w}(\hat{\tau}, \hat{s})|^2\, \frac{ d\hat{\tau}}{\hat{\tau}} d\hat{s} =  \int |w(\hat{\tau}^{1/N}, N^{-1/2} \hat{s})|^2 \, \frac{ d\hat{\tau}}{\hat{\tau}} d\hat{s} = 
N^{3/2} \int |w(\tau,s)|^2\, \frac{d\tau}{\tau} ds,
\]
\[
\int | \del_{\hat{s}} {\bf w}(\hat{\tau}, \hat{s})|^2 \frac{ d\hat{\tau}}{\hat{\tau}} d\hat{s} = N^{1/2} \int |\del_s w(\tau,s)|^2\,  \frac{d\tau}{\tau}ds,
\]
\[
\int |\del_{\hat{s}}^2 {\bf w}(\hat{\tau}, \hat{s})|^2 \, \frac{ d\hat{\tau}}{\hat{\tau}} d\hat{s}  = N^{-1/2} \int |\del_s^2 w(\tau,s)|^2\, \frac{d\tau}{\tau}ds,
\]
and
\[
\int | \hat{\tau}\del_{\hat{\tau}} {\bf w}(\hat{\tau},\hat{s})|^2\, \frac{ d\hat{\tau}}{\hat{\tau}} d\hat{s}  = N^{-1/2} \int |\tau\del_\tau w(\tau,s)|^2\, \frac{d\tau}{\tau}ds.
\]
As for the boundary values, first note that if ${\bf k}(\hat{\tau}) = k( \hat{\tau}^{1/N})$, then
\[
\int_0^\infty k(\hat{\tau}^{1/N}) \hat{\tau}^{i\mu} \, \frac{d\hat{\tau}}{\hat{\tau}} = N \int_0^\infty k(\tau) \tau^{iN\mu}\, \frac{d\tau}{\tau},
\]
or in other words,
\[
\calM( {\bf k})(\mu)  = N \calM(k)( N\mu).
\]
Therefore,
\begin{multline*}
||{\bf k}_0||^2_{3/4} = \int |\tilde{{\bf k}}_0|^2 (1 + \mu^2)^{3/4}\, d\mu = N^2 \int |\tilde{k}_0(N \mu)|^2 (1 + \mu^2)^{3/4} \, d\mu \\
= N \int |\tilde{k}_0(\nu)|^2 (1 + N^{-2} \nu^2)^{3/4}\, d\nu \leq N \int |\tilde{k}_0(\nu)|^2 (1 + \nu^2)^{3/4} \, d\nu,
\end{multline*}
and 
\[
||{\bf k}_1||^2_{1/4} =  N \int |\tilde{k}_1(N\mu)|^2 (1 + \mu^2)^{1/4}\, d\mu \leq \int |\tilde{k}_1(\nu)|^2 (1 + \nu^2)^{1/4}\, d\nu.
\]
We used here that if $a > 0$, then 
\[
(1 + N^{-2}\nu^2)^a = N^{-2a}( N^2 + \nu^2)^a = N^{-2a}(1 + \nu^2)^a  \left( \frac{N^2 + \nu^2}{1 + \nu^2} \right)^a \leq (1 + \nu^2)^a.
\]
Multiplying all terms by $N^{1/2}$, we derive, at long last, the final estimate
\begin{equation*}
\int (\tau\del_\tau w)^2 + (\del_s^2 w)^2 + N (\del_s w)^2 + N^2 |w|^2 \leq C \left( N^{3/2} ||k_0||_{3/4}^2 + N^{1/2} ||k_1||_{1/4}^2\right).
\end{equation*}

Recall that we had replaced $k_1 - \beta_1(V)$ by $k_1$. This means that  the correct full estimate in region III is
\begin{equation}
\int (\tau\del_\tau u)^2 + (\del_s^2 u)^2 + N (\del_s u)^2 + N^2 |u|^2 \leq C \left( N^{3/2} ||k_0||_{3/4}^2 + N^{1/2} ||k_1||_{1/4}^2 + 
||\tau^2 h||_{L^2}^2\right).
\label{G4}
\end{equation}
Again we remind that this is with respect to the measures  $d\tau  ds/\tau$ and $d\tau/\tau$, but we can  obtain  an
identical-looking estimate by replacing $u$ by $\sqrt{\tau} u$. 

\medskip

We are finally in a position to define the spaces $\calX$, $\calY$ and $\calZ$.   First, the measures on the right hand sides 
of \eqref{F5}, \eqref{E0}, \eqref{G1} and \eqref{G4} are
\[
s^2 \tau^4 d\tau ds,\ \ x^2 dt dx, \ \ \mbox{and}\ \  \tau^4 d\tau ds;
\]
the first of these holds in the region where $s \geq s_0$ and the latter when $s \leq s_0$.  After changing coordinates, these
measures are all equivalent to
\[
d\mu_{\calY} = (1 + s^2)\tau^4\, d\tau ds \cong (t + x^2) \, dt dx;
\]
either of these expressions localize to the correct measure in each of the regions. 

Next, the space $\calX$ is defined by the ($N$-dependent) conditions that 
\begin{multline*}
||u||_{\calX}^2 := \int (\tau \del_\tau u)^2 + (\del_s^2 u)^2 + (N+s^2)(\del_s u)^2 + (Ns^2 + N^2) u^2 \, d\tau ds  \\
\cong \int t (\del_t u)^2 + t (\del_x^2 u)^2 + (N + \frac{x^2}{t})(\del_x u)^2 + (N \frac{x^2}{t^2} + N^2\frac{1}{t}) u^2\, dt dx < \infty.
\end{multline*}

Finally, $\calZ = \calZ_0 \oplus \calZ_1$ is defined as $H^{3/4}_b$ and $H^{1/4}_b$, the fractional Sobolev
spaces on $\mathbb R^+$ defined in terms of the vector field $\tau \del_\tau$, and with respect to the measures 
$N^{3/2}d\tau$ and $N^{1/2} d\tau$, respectively.

With these very extensive calculations, we can now have the
\begin{theorem}
The mapping
\begin{equation*}
\begin{aligned}
(t\del_t  - t \calL, \beta) :  \  t^N \calX & \longrightarrow t^N\calY \oplus t^N\calZ  \\
U &\longmapsto  ( (t\del_t - t \calL) U, \beta(U)) 
\end{aligned}
\end{equation*}
is an isomorphism when $N$ is sufficiently large.
\label{linexist}
\end{theorem}

We sketch the proof briefly.  First, recall that the mapping here is equivalent to
\[
(t\del_t + N - t\calL, \beta): \ \calX \longrightarrow \calY \oplus \calZ,
\]
so we consider this latter operator instead.  Next, if $N$ is sufficiently large, we may use the model estimates proved 
above to obtain the corresponding estimate for this operator on $\widehat{\Gamma}$. This is done by the usual 
process of approximation, passing from constant coefficient to variable coefficient operators, and then 
pasting together these local a priori estimates to get
\[
||U||_{\calX} \leq C ( || t\del_t +N - t\calL)U||_{\calY} + ||\beta(U)||_{\calZ})
\]
on the entire network. The constants in this estimate depend only on $\calC^0$ norms of the coefficients of 
the operator, and in particular are invariant under continuous deformation of the approximate solution and the network. 

This estimate implies, first of all, that any solution to $(t\del_t +N - t\calL)U = H$, $\beta(U) = K$, if it exists, must be unique.
In addition, the range of this boundary problem is closed, so the theorem will be proved if we can show that its range
is dense. However, this is straightforward: the spaces here are all $L^2$-based, and by standard parabolic theory, the
range of this operator in any region $\{t \geq \epsilon\}$ is onto.  

We next define spaces suitable for the higher regularity version. If $X$ is any manifold with corners, we define $\calV_b(X)$ to 
be the space of all smooth vector fields on $X$ which are tangent to all boundary faces. Thus, for example, $\calV_b(Q_h)$ is 
spanned over $\calC^\infty(Q_h)$ by $\tau\del_\tau$ and $s\del_s$ near $\lf \cap \ff$  and by $y\del_y$ and $T\del_T$ near 
$\bff \cap \ff$; similarly, $\calV_b(\RR^+)$ is spanned over $\calC^\infty(\RR^+)$ by $t\del_t$. These spaces of vector fields are 
closed under Lie bracket, and are the natural starting point in the analysis of many classes of degenerate differential 
operators on such spaces.  In particular, note that $t\del_t - t\calL$ is a sum of products of elements of $\calV_b(Q_h \cup P_h)$.
Their main use here is based on the fact that if $V \in \calV_b$, then the commutator $[V, t\del_t - t\calL]$ is of the form 
$t\del_t - t\calL'$ where $\calL'$ is an operator of exactly the same type as $\calL$.  Thus 
\[
V (t\del_t - t\calL) U =  (t\del_t - t\calL') (VU) = V(tH) = t (VH + [V,t] H) = tH',
\]
with a similar expression for the boundary values. 

Now for any $\ell \in \NN$, define 
\[
\calX^k = \{U \in \calX:  V_1 \ldots V_j U \in \calX\ \forall\, j\leq k\ \mbox{and}\ V_i \in \calV_b\},
\]
with similar definitions for $\calY^k$ and $\calZ^k$.   Applying a standard commutator argument and the same conjugation
trick as above, we can reduce the analysis of $(t\del_t  - t\calL, \beta): t^N \calX^\ell \longrightarrow t^N\calY^\ell \oplus 
t^N\calZ^\ell$ to that of a similar operator $(t\del_t  + N - t \calL', \beta): \calX \longrightarrow \calY \oplus \calZ$,
and for this we obtain estimates and an existence and uniqueness theorem exactly as above.  

\begin{corollary}
For every $\ell \geq 0$ and $N \gg 1$, 
\[
(t\del_t - t\calL, \beta):  t^N \calX^{\ell} \longrightarrow t^N \calY^{\ell} \oplus t^N \calZ^{\ell} 
\]
is an isomorphism.    

In particular, if $H$ and $K$ are smooth and vanish to infinite order, then there exists a solution $U$ to this
problem which is also smooth and vanishes to all orders. This solution is unique. 
\label{finlinest}
\end{corollary}   

\section{Short-time existence}\label{short-time existence}
We now have all the tools in place to prove our main result:

\begin{theorem}\label{existence-uniqueness-special-flow}
Fix an initial network $\Gamma_0 = \{\gamma_0^{(j)}\}_{j=1}^n$ which consists of $n$ smooth
curves $\gamma_0^{(j)}\in C^{\infty}([0,1])$, $j = 1, \ldots, n$. Suppose that at each interior vertex $p$, there are $m_p$ curves
which come together there, and that these intersections are all nontangential.  For each interior vertex $p$, choose an 
expanding self-similar solution $S_p$. Then there exists a family of networks $\Gamma(t)$  which are regular for each $t > 0$, 
and which vary smoothly in the sense that the worldsheet $\{ (t, \Gamma(t)): 0 \leq t < \epsilon\}$ is the imagine of a 
smooth function with domain $Q_h$.  The topology of $\Gamma(t)$ for $t > 0$ is obtained by excising a small
ball around each interior vertex $p$ and inserting a truncated copy of $S_p$.   For a given topology, the solution is unique. 
\end{theorem}

\begin{proof}
Construct an approximate solution $\widehat{\Gamma}(t) = \{ \hat{\gamma}^{(i)}\}_{i=1}^m$; as per our
earlier convention, suppose that each $\hat{\gamma}^{(j)}$, $j \leq n$, limits to $\gamma^{(j)}_0$, whereas
the $\hat{\gamma}^{(j)}$ with $j > n$ are the curves which disappear into the vertices as $t \to 0$. 
By abuse of language, we consider each $\hat{\gamma}^{(j)}$ as a map from $Q_h$ into $Z_h$.

We seek a family of perturbations $U = \{u^{(j)}\}$, where each $u^{(j)}$ is a function on $Q^{(j)}$ or $P^{(j)}$, as appropriate, 
and this collection of functions is chosen so that the flow equations and boundary conditions hold, i.e., 
\begin{equation*}
\bbM( \widehat{\Gamma} + U) = 0,  \qquad B( \widehat{\Gamma} + U) = 0.
\end{equation*}
We approach this in the standard way. Expand the nonlinear operators $\calM$ and $B$ around the approximate solution,
and denote their linearizations by $\del_t - \calL$ and $\beta$, respectively, as we have done in the first part of
\S 6.  Now set $H = -\calM( \widehat{\Gamma})$, $J = - B(\widehat{\Gamma})$.  We can then write this nonlinear
initial-boundary problem as
\begin{equation}
(\del_t - \calL) U = H + Q(U), \quad \beta(U) = J + S(U),
\label{nonlin-Taylor}
\end{equation}
where $Q$ and $S$ are the quadratically vanishing remainder terms in these two Taylor approximations.

Our careful choice of approximate solution ensures that $H$ and $J$ are smooth and vanish to all orders as $t \to 0$.
We therefore seek to find a solution $U$ which is also smooth and vanishing to all orders.  
Choose $\ell$ large enough so that multiplication $t^N \calX^\ell \times t^N \calX^\ell \to t^{2N} \calX^\ell$ is bounded. 
Next, suppose that we are solving the problem on some interval $0 \leq t < t_0$.  If $Q$ is any quadratically vanishing
expression in the components of $U$ and $N \geq 2$, then 
\begin{equation}
|| Q(U) ||_{ t^N \calX^\ell} \leq  t_0^N ||U||_{t^N \calX^\ell}
\label{vanish}
\end{equation}
for any $U \in t^N \calX^{\ell}$. 

Now let $\mathbb H$ denote the solution operator for $(t\del_t - t\calL, \beta)$, as in Theorem \ref{linexist}, and
rewrite \eqref{nonlin-Taylor} as the integral equation. 
\begin{equation}
U = \mathbb H( H + Q(U),  J + S(U)).
\label{inteq}
\end{equation}
It is now quite standard, using Corollary \ref{finlinest} and taking advantage of \eqref{vanish}, to check that
\eqref{inteq} is a contraction mapping.  Hence there is a unique solution $U$, and clearly $U$ is smooth
and vanishes to all orders as $t \to 0$.

This completes the proof of short-time existence.  

As for uniqueness, notice that if two solutions $\Gamma(t)$ and $\Gamma'(t)$ are two smooth flowouts from the
same initial network $\Gamma_0$ and both have the same topology, then all Taylor coefficients of the constituent
curves $\gamma^{(j)}(t,x)$ and $\gamma^{(j)\prime}(t,x)$ must agree.  This means that the differences
$\gamma^{(j) \prime} - \gamma^{(j)}$ necessarily vanish to all orders as $t \to 0$.  By the uniqueness inherent
in the contraction mapping step of the argument, this implies that $\gamma^{(j)}(t,x) \equiv \gamma^{(j)\prime}(t,x)$. 
\end{proof}

\begin{proposition}\label{C2datum}
Theorem~\ref{existence-uniqueness-special-flow} also holds for $\calC^2$ irregular networks $\Gamma_0=\{\gamma_0^{(j)}\}_{j=1}^n$.  
\end{proposition}

\begin{proof}
We prove this by taking a $\calC^2$ approximation of $\Gamma_0$ by a sequence of $\calC^\infty$ networks $\Gamma_{0,\ell}$,
and showing that the local existence theorem around any of these is valid in a fixed-sized neighborhood in the space
$\calX^0$. For $\ell$ large enough, this neighborhood includes an evolving network with initial condition $\Gamma_0$. 

First consider the approximate solution construction, where to each smooth initial network $\Gamma_{0,\ell}$ we assign
a flowout $\widehat{\Gamma}_\ell$ of some fixed topology.  It is possible to choose these approximate solutions
to depend continuously in the $\calC^\infty$ topology on the initial network. Indeed, the coefficient functions in the series
expansion at each front face clearly depend continuously on the derivatives of each $\gamma_{0,\ell}^{(j)}$ at $x=0$ or $x=1$.
We then use the Seeley extension theorem, which is a quantitative form of the Borel lemma. It states specifically that
the extension map, carrying the coefficients to all orders of these expansions at each boundary face of $Q_h \sqcup P_h$ 
to the approximate solution $\widehat{\Gamma}_\ell$, can be chosen to depend continuously on these coefficients. 

Now turn to the contraction argument. We may as well work in some fixed space $t^N \calX^0$. We seek a solution 
$U_\ell$ to  $\mathbb M( \widehat{\Gamma}_\ell + U_\ell) = 0$, $B(\widehat{\Gamma}_\ell + U_\ell) = 0$, which we
set up as the equivalent integral equation \eqref{inteq}.   The usual argument shows that this is a contraction on a ball
of size $R_0$ in the space $t^N \calX^0$, where the radius $R_0$ depends on the norm of the solution operator $\mathbb H_\ell$
on $[0,t_0)$.  In particular, this proof also shows that $t_0$ and hence $R_0$ can be chose to be uniform in $\ell$. 

Recalling that each arc in $\Gamma_{0,\ell}$ is $\calC^2$, we see that the coefficients in the linearized operators $\calL_\ell$ 
are convergent in the $\calC^0$ topology.   We have remarked that the constants in the main a priori estimate leading
to the fact that $\calL_\ell: t^N \calX \to t^N \calY \oplus t^N \calZ$ depend only on the size in $\calC^0$ of the
coefficients of $\calL_\ell$, and these are uniformly controlled.   Thus for this particular sequence of smooth
initial networks $\Gamma_{0,\ell}$, the solution $\widehat{\Gamma}_\ell + U_\ell$ exists and is unique in a ball 
around $\widehat{\Gamma}_\ell$ of fixed radius $R_0$. Furthermore, the $t^N \calX^0$ norm of $U_\ell$ is
uniformly controlled. 

We now take a weak limit of these solutions $U_\ell$ to some $U \in t^N \calX^0$. By classical parabolic regularity
theory, this convergence can be assumed to be strong in the $\calC^\infty$ topology in any region $t > \epsilon$. 
This convergence is also in $\calC^2$ on the bottom face, i.e., as $t \searrow 0$ but $x \neq 0,1$, and in $\calC^\infty$
on the front face in any region where $s \leq C$. The one place where classical methods do not apply is near
the corner where the front and bottom faces meet. We obtain in this way a nonzero solution $U$ to the flow equations 
on some fixed time-interval $[0,t_0)$ which have initial $\calC^2$ network $\Gamma_0$.   This proves existence of
a flowout for any admissible $\calC^2$ network $\Gamma_0$.
\end{proof}

\begin{remark}
We do not assert uniqueness in this last result, but note that it is highly likely to be true. 
\end{remark}

\section{Final comments and further directions}

In this last section we present a number of results which round out and provide further
context to our main theorem. 

\subsection{The unparametrized evolution equation}
The analysis in this paper has focused on the parametrized network flow, i.e., the specific parabolic equation
\[
\del_t \gamma = \frac{ \del_x^2 \gamma}{|\del_x \gamma|^2}
\]
with the Herring boundary conditions. As explained in \S  2, this formulation is equivalent to imposing a specific tangential term in the velocity vector.
We study here the question of whether the geometric objects, i..e. network of curves, corresponding to the solution we have 
found, which of course depends on the choice of expanding soliton at  each interior vertex, also depends on this tangential term.  
To do so, we introduce a  ``geometric''  formulation of this curvature flow that only involves the evolution of the set, but does not 
fix a choice of parametrization.  In the following we denote networks by $\calN$ rather than $\Gamma$ to emphasize the 
difference.

\begin{definition}
A network of $n$ curves $\calN_0 = \{n_0^{(j)}\}$ is an admissible initial network if $n_0^{(j)}$ admits a 
regular $\calC^2$ admissible parametrization $\gamma_0^{(j)}$. 
\end{definition}

\begin{definition}[Motion by curvature of networks]
Let $\calN_0 = \{n^{(j)}_0\}$, $j = 1, \ldots, n$ be an admissible initial network, and $\gamma^{(j)}$ a regular parametrization of
$n_0^{(j)}$. A time dependent family of networks $\calN(t)$ with $m \geq n$ constituent curves for each $t \in (0, T_{\max})$ is 
a maximal geometric solution of the geometric motion by curvature  in $[0,T_{\max})$  with Dirichlet boundary conditions and initial 
condition $\calN_0$ if each curve $n^{(j)}$ in $\calN$ has a $\calC^2$ parametrization $\gamma^{(j)}$ and the following is true.
First, for each $j$, 
\begin{equation*}
\left\langle\partial_t\gamma^{(j)}(t,x), \nu^{(j)}(t,x) \right\rangle = {\kappa}^{(j)}(t,x), 
\end{equation*}
where $\nu^{(j)}$ is the unit normal to $\gamma^{(j)}$; second, $\calN(t)$ is a regular network for each $t > 0$, i.e.,
has only regular triple junctions; third, its exterior vertices are fixed points on $\del \Omega$; finally, 
for $j \leq n$, the curves $n^{(j)}$ converge uniformly to $n^{(j)}_0$ as $t\to 0$, while for $j > n$, 
the curves $n^{(j)}$ converge to one of the interior vertices of $\calN_0$.  We say that $T_{\max}$ is the maximal 
time of existence for this regular evolution if this solution exists for $0 \leq t < T_{\max}$, but there exists 
no other regular network on a longer time interval.
\end{definition}

\begin{proposition}\label{geomex}
For any admissible initial network $\calN_0$, there exists a geometric solution $\calN(t)$, $0 \leq t \leq T$, to 
the network flow with with initial datum $\calN_0$. 
\end{proposition}
\begin{proof}
If $\{\gamma^{(j)}_0\}$ is a $\calC^2$ parametrization of the curves in $\calN_0$, then by Theorem~\ref{existence-uniqueness-special-flow} 
there exists a solution $\Gamma(t) = \{\gamma^{(j)}\}$ of the parametrized flow with initial datum $\Gamma_0:=\{\gamma_0^{(j)}\}$. 
Clearly this is a solution of the unparametrized geometric flow as well. 
\end{proof}

\begin{definition}
Two admissible networks $\calN$ and $\calN'$ in the smoothly bounded set $\Omega$ with the same outer vertices are said to have the same 
topology if there is a diffeomorphism $\Phi$ of $\Omega$ which fixes the outer vertices and is such that $\Phi(\calN) = \calN'$ and 
$\Phi( \Omega \setminus \calN) = \Omega \setminus \calN'$. In particular, two such networks have the same number of interior and exterior curves, vertices,
and overall combinatorial type. 
\end{definition}

\begin{proposition}\label{fixedtop}
Let $\calN_0$ be a smooth admissible irregular initial network, and $\calN(t)$, $0 \leq t \leq T$ a smooth
solution of the unparametrized geometric flow starting from $\calN_0$. Choose any smooth parametrization
$\Gamma(t) = \{\gamma^{(j)}\}$ of $\calN(t)$. Now suppose that $\widetilde{\Gamma}(t) = \{\tilde{\gamma}^{(j)}\}$, 
$0 \leq t \leq \widetilde{T}$ is any 
other solution of the parametrized flow which limits geometrically to $\calN_0$ and which has the same topology 
as $\Gamma(t)$ for any $t \in (0, \min\{T, \widetilde{T}\})]$. Then the curves in $\widetilde{\Gamma}(t)$ are
reparametrizations of those in $\Gamma(t)$. 
\end{proposition}

\begin{remark}
  If $\calN_0$ is a regular network containing precisely $3$ curves, then this result reduces to~\cite[Theorem 3.16]{GMP}. 
\end{remark}

\begin{proof}
We must prove that there exists a collection of functions $\psi^{(j)}:[0, T]\times[0,1]\to [0,1]$ such that 
$\widetilde{\gamma}^{(j)}=\gamma^{(j)}(t, \psi^{(j)}(t,x))$.   We shall derive evolution equations for
the $\psi^{(j)}$, and in fact also for their inverses, $\xi^{(j)}: [0,\widetilde{T}] \times [0,1] \to [0,1]$,
$\psi^{(j)} (t, \xi^{(j)}(t,x)) = x$, $\xi^{(j)}(t, \psi^{(j)}(t,x)) = x$.  Not unexpectedly, it is necessary to pass to
the blowups $Q_h \cup P_h$ to write these equations.


Suppose for the moment that $t>0$, and we derive the equations away from the front faces.  
Obviously, reparametrizations do not affect the normal velocities, so the normal components of
$\del_x \widetilde{\gamma}^{(j)}$ and $\del_x \gamma^{(j)}$ are the same at corresponding points. 
The reparametrization must be chosen so that the tangential components of these two tangent
vectors also correspond.  This will be the case if 
\begin{equation*}
\begin{cases}
\begin{array}{ll}
\psi^{(j)}_t(t,x) &= \displaystyle{\frac{\psi_{xx}^{(j)}\left(t,x\right)}{\left|\displaystyle{\widetilde{\gamma}_{x}^{(j)} 
\circ \psi^j(t,x)}\right|^2 \psi^{(j)}_x(t,x)^2}} -
\frac{\langle\widetilde{\gamma}_t^{(j)}\circ\psi^{(j)}(t,x),\widetilde{\gamma}^{(j)}_x \circ \psi^{(j)}(t,x)\rangle}{\vert\widetilde{\gamma}^{(j)}_x \circ \psi^{(j)}(t,x)\vert^2}\\ 
& \qquad + \displaystyle{\frac{1}{|\widetilde{\gamma}_x^{(j)}\circ \psi^{(j)}(t,x))|}
\left\langle\frac{\widetilde{\gamma}_{xx}^{(j)}\circ \psi^{(j)}(t,x)}{|\widetilde{\gamma}_{x}^{(j)}\circ \psi^{(j)}(t,x)|^{2}}, 
\frac{\widetilde{\gamma}_x^{(j)}\circ \psi^{(j)}(t,x))}{|\widetilde{\gamma}_{x}^{(j)}\circ \psi^{(j)}(t,x)|}\right\rangle
}  \\
\psi^{(j)}(t,0)&=0
\\
\psi^{(j)}(t,1)&=1
\\
\widetilde{\gamma}\left(0,\psi^{(j)}(0,x)\right)&=\gamma^{(j)}_0(x)
\end{array}
\end{cases}
\end{equation*}
here we have used the notation $\tilde{\gamma}^{(j)}\circ \psi^{(j)}(t,x)$ to stand for $\tilde{\gamma}^{(j)}(t, \psi^{(j)}(t,x))$. 
Unfortunately, the coefficients of this system depend on the $\psi^{(j)}(t,x)$. To remedy this, we consider instead
the evolution problem for the functions $\xi^{(j)}$ instead. These are
\begin{align*}
\partial_t\xi^{(j)}(t,y)&
=-\partial_t\psi^{(j)}(t,\xi^{(j)}(t,y))\partial_y\xi^{(j)}(t,y)\\
&=-\frac{\psi_{xx}^{(j)}\left(t,x\right)}{|\widetilde{\gamma}_{x}^{(j)}\left(t,y\right)|^{2}}
\partial_y\xi(t,y)^3
+\frac{\left\langle\widetilde{\gamma}_t^{(j)}(t,y),\widetilde{\gamma}^{(j)}_x(t,y)\right\rangle}{\vert\widetilde{\gamma}^{(j)}_x(t,y)\vert^2}\partial_y\xi(t,y)\\
&\,\,\,\,\,-\frac{\partial_y\xi(t,y)}{|\widetilde{\gamma}_{x}^{(j)}\left(t,y\right)|}
\left\langle\frac{\widetilde{\gamma}_{xx}^{(j)}\left(t,y\right)}{|\widetilde{\gamma}_{x}^{(j)}\left(t,y\right)|^{2}}\,,\,
\frac{\widetilde{\gamma}_x^{(j)}(t,y)}{|\widetilde{\gamma}_{x}^{(j)}\left(t,y\right)|}\right\rangle\,,
\end{align*}
and
\begin{align*}
\partial_y\xi^{(j)}(t,y)&= 1/\partial_y\psi^{(j)}(t,\xi^{(j)}(t,y))\,,\\
\partial^2_y\xi^{(j)}(t,y)&=-\frac{\partial_y^2\psi^{(j)}(t,\xi^{(j)}(t,y))}{\partial_y\psi^{(j)}(t,\xi^{(j)}(t,y))^3}
=-\partial_y\xi^{(j)}(t,y)^3\partial_y^2\psi^{(j)}(t,\xi^{(j)}(t,y))\,.
\end{align*}
this leads to the {\it linear} system
\begin{equation*}
\begin{cases}
\begin{array}{ll}
\partial_t\xi^{(j)}(t,y)&
= \displaystyle{ \frac{\partial_y^2\xi(t,y)}{|\widetilde{\gamma}_{x}^{(j)}(t,y)|^{2}}
+\frac{\langle\widetilde{\gamma}_t^{(j)}(t,y),\widetilde{\gamma}^{(j)}_x(t,y)\rangle}{\vert\widetilde{\gamma}^{(j)}_x(t,y)\vert^2}
\partial_y\xi(t,y)}\\ 
&\,\,\,\,\, \displaystyle{-\frac{\partial_y\xi(t,y)}{|\widetilde{\gamma}_{x}^{(j)}(t,y)|}
\left\langle\frac{\widetilde{\gamma}_{xx}^{(j)}(t,y)}{|\widetilde{\gamma}_{x}^{(j)}(t,y)|^{2}}\,,\,
\frac{\widetilde{\gamma}_x^{(j)}(t,y)}{|\widetilde{\gamma}_{x}^{(j)}(t,y)|}\right\rangle }
 \\
\xi^{(j)}(t,0)&=0
\\
\xi^{(j)}(t,1)&=1
\\
\xi^{(j)}(x)&=\widetilde{\gamma}^{-1}(y)\,.\\
\end{array}
\end{cases}
\end{equation*}
The coefficients here are at least continuous since $\Gamma(t)$ solves the network flow and its curves are of class $\calC^{1,2}$.
We write this main equation as
\begin{equation}\label{main-repara}
\partial_t\xi^{(j)}(t,x)=a^{(j)}(t,x)\partial_x^2\xi^{(j)}(t,x)+b^{(j)}(t,x)\partial_x \xi^{(j)}(t,x) ;
\end{equation}
here $a^{(j)}(t,x)$, $b^{(j)}(t,x) \in \calC^0$. 

Now a bit of calculation shows that in the $(\tau,s)$ coordinates, \eqref{main-repara} takes the form
\[
\tau \partial_\tau\xi^{(j)}-s\partial_s \xi^{(j)}-\tilde{a}\partial_s^2\xi^{(j)}-\tilde{b}\tau\partial_s\xi^{(j)}=0\,.
\]
Taking into account the explicit formul\ae\ for $a$ and $b$, and writing $\xi=\tau\varphi$ and $\gamma=\tau\eta$,
we obtain finally that
\begin{multline*}
\left(\tau\partial_\tau+1-2\partial_s\right)\varphi^{(j)}
= \\ 
\frac{\partial_s^2\varphi^{(j)}}{\vert \partial_s\eta^{(j)}\vert^2}
+\frac{\partial_s\varphi^{(j)}}{\vert\partial_s\eta^{(j)}\vert}
\left\langle\left(\tau\partial_\tau+1-2\partial_s\right)\eta^{(j)}, 
\frac{\partial_s\eta^{(j)}}{\vert\partial_s\eta^{(j)}\vert}\right\rangle
-\frac{\partial_s \varphi^{(j)}}{\vert\partial_s\eta^{(j)}\vert}\left\langle\frac{\partial_s^2\eta^{(j)}}{\vert\partial_s\eta^{(j)}\vert^2},
\frac{\partial_s\eta^{(j)}}{\vert\partial_s\eta^{(j)}\vert}\right\rangle\,.
\end{multline*}

This system has the same structure as the main one we have studied earlier in this paper, and following the steps
of that argument, we can build first an approximate solution to infinite order, then correct this to an exact solution.
We omit the details. 
\end{proof}

\begin{corollary}
The number of geometric flowouts of an initial geometric network $\calN_0$, up to reparametrization, is
precisely the same as the product of the number of expanders at each internal vertex.
\end{corollary}
\begin{proof}
The claim is that for every possible topology for a flowout, there exists exactly one unparametrized geometric
solution starting from $\calN_0$.    First, choose a compatible expander at each internal vertex of $\calN_0$.
This determines the topology of an approximate solution to the parametrized flow equations, and thus
an exact solution of these equations.  On the other hand, we have proved that any two geometric solutions 
with the same topology coincide up to reparametrization.  This establishes the bijection. 
\end{proof}

This answers the two Open Problems, 11.19 and 11.20, as stated in~\cite{Man}.

\begin{corollary}
If the initial network has precisely one triple junction and no other internal vertices, then
there is a unique solution of the network flow. If the initial datum has one quadruple junction 
forming angles of $\pi/3$ and $2\pi3$, then there is a unique solution of the network flow.
\end{corollary}

\subsection{Disconnected and unstable solutions}
The solutions we have found in this paper are regular networks when $t > 0$; in particular, every vertex
of the original network resolves into an expanding cluster of vertices, connected by a system of 
new curves, in such a way that the evolving vertices are all regular, i.e., points where precisely three
curves meet at $2\pi/3$.  In this section we describe briefly two other possible evolutions, both
of which are immediately accessible by our methods.  The first is where the network instantaneously
disconnects (locally) at one or more of the interior vertices, and the second is where the network 
evolves but maintaining irregular vertices modelled on unstable expanding solutions.   
The fact that this network flow is the gradient flow of the overall length of the network indicates
that the `physically relevant' solutions are those which are locally and globally stable. This suggests
that the second of these `new' flows is an oddity. On the other hand, there are situations where the
physics suggest that the network should disconnect.  We explain all of this now. 

Regarding the first of these issues, we recall from \cite{MS} that disconnected self-similar
expanders exist.  Indeed, as described in that paper, such expanders are in bijective
correspondence with geodesic Steiner trees in the negatively curved space
$(\RR^2,g)$, where $g=e^{\vert x\vert^2}\vert dx\vert^2$. The initial condition that the expander
limits to a particular union of $n$ rays meeting at the origin corresponds to the fact
that these Steiner trees have $n$ exterior edges limiting to points $q_1, \ldots, q_n$ lying
on the boundary of the geodesic compactification of this space. By virtue of the negative
curvature of $g$, this geodesic compactification is naturally identified with the closed
unit disk, so we may consider these points as arranged in cyclic order around the circle. 

Now suppose that we decompose the set $\{1, \ldots, n\}$ into disjoint subsets $S_1 \sqcup S_2 \sqcup \ldots S_k$,
where each $S_j$ consists of a contiguous set of elements (mod $n$) and $|S_j| \geq 2$ for all $j$. 
In other words, each $S_i$ is of the form $\{ \ell, \ell+1, \ldots, \ell'\}$.  The general methods there (which are 
based on minimization of $1$-rectifiable currents with coefficients in $\mathbb Z_r$ (for various values of $r$)
show that there is a connected geodesic Steiner tree $\calT_j$ with asymptotic boundary $S_j$ for each $j$. It follows 
by the regularity theory for these varifold minimizers that $\calT_i \cap {\cal T}_{j} = \emptyset$ for $i \neq j$. In other words,
this produces a geodesic Steiner tree with precisely $k$ components. 

Now let $\Gamma_0$ be any irregular geodesic network such that at least one vertex has valence
strictly greater than $3$. According  to our main theorem, the choice of a self-similar expander
at each interior vertex of $\Gamma_0$ leads to a unique evolution $\Gamma(t)$, $0 < t < T$.
If the expander is disconnected, the corresponding flowout $\Gamma(t)$ becomes immediately
locally disconnected near that vertex. Depending on the overall topology, $\Gamma(t)$ itself
may become disconnected. In some situations this may be the more physically appropriate situation.  

In each of the pictures below, the initial irregular network is the same. In the first picture it is 
the boundary of the blue set, in the second one the boundary of the yellow set, and in 
the third, it is the boundary for four different phases.  In the first two cases, the disconnected
flowouts are physically realistic, whereas in the last case, the four phases must remain separate
so the evolving network should also remain connected, but there are two possible combinatorial
arrangements.  

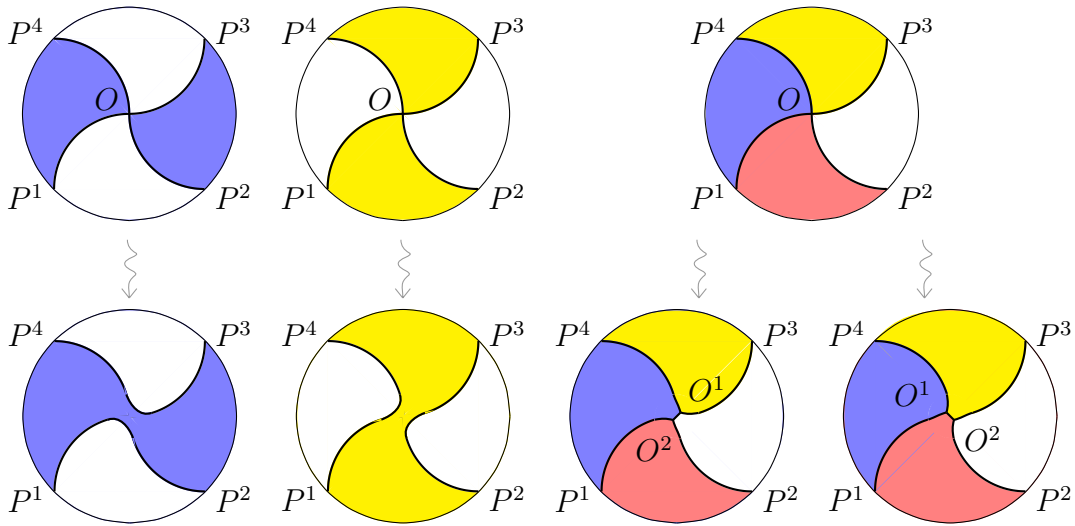
\begin{figure}[h]
\begin{center}
\begin{tikzpicture}
\filldraw[color=blue!50!white,scale=1,domain=-3.141: 3.141,
smooth,variable=\t,shift={(-5,2)},rotate=0]plot({1.42*sin(\t r)},
{1.42*cos(\t r)});
\filldraw[color=white]
(-6,3)to[out= 40,in=180, looseness=1](-5,3.4)--
(-5,3.4)to[out= 0,in=140, looseness=1](-4,3);
\filldraw[color=white]
(-6,3)to[out= 0,in=180, looseness=1](-4,3)--
(-4,3)to[out= -135,in=45, looseness=1](-5,2)--
(-5,2)to[out= 130,in=-40, looseness=1](-6,3);
\filldraw[color=blue!50!white,thick,scale=1,domain=0: 1.5708,
smooth,variable=\t,shift={(-6,2)},rotate=0]plot({1*sin(\t r)},
{1*cos(\t r)});
\filldraw[color=white,thick,scale=1,domain=0: 1.5708,
smooth,variable=\t,shift={(-5,3)},rotate=-90]plot({1*sin(\t r)},
{1*cos(\t r)});
\filldraw[color=white]
(-6,1)to[out= -40,in=180, looseness=1](-5,0.6)--
(-5,0.6)to[out= 0,in=-140, looseness=1](-4,1);
\filldraw[color=white]
(-6,1)to[out= 0,in=180, looseness=1](-4,1)--
(-4,1)to[out= 150,in=-50, looseness=1](-5,2)--
(-5,2)to[out= -135,in=45, looseness=1](-6,1);
\filldraw[color=white,thick,scale=1,domain=0: 1.5708,
smooth,variable=\t,shift={(-5,1)},rotate=90]plot({1*sin(\t r)},
{1*cos(\t r)});
\filldraw[color=blue!50!white,thick,scale=1,domain=0: 1.5708,
smooth,variable=\t,shift={(-4,2)},rotate=180]plot({1*sin(\t r)},
{1*cos(\t r)});
\draw[color=black,scale=1,domain=-3.141: 3.141,
smooth,variable=\t,shift={(-5,2)},rotate=0]plot({1.42*sin(\t r)},
{1.42*cos(\t r)});
\draw[color=black,thick,scale=1,domain=0: 1.5708,
smooth,variable=\t,shift={(-6,2)},rotate=0]plot({1*sin(\t r)},
{1*cos(\t r)});
\draw[color=black,thick,scale=1,domain=0: 1.5708,
smooth,variable=\t,shift={(-4,2)},rotate=180]plot({1*sin(\t r)},
{1*cos(\t r)});
\draw[color=black,thick,scale=1,domain=0: 1.5708,
smooth,variable=\t,shift={(-5,1)},rotate=90]plot({1*sin(\t r)},
{1*cos(\t r)});
\draw[color=black,thick,scale=1,domain=0: 1.5708,
smooth,variable=\t,shift={(-5,3)},rotate=-90]plot({1*sin(\t r)},
{1*cos(\t r)});
\path[font=\large]
(-6,3.1) node[left] {$P^4$}
(-4,3.1) node[right] {$P^3$}
(-6,.9) node[left] {$P^1$}
(-4,.9) node[right] {$P^2$}
(-5.3,2.47) node[below] {$O$};
\draw[color=black!40!white,scale=0.5,shift={(-10,-2)}]
(-0.05,2.65)to[out= -90,in=150, looseness=1] (0.17,2.3)
(0.17,2.3)to[out= -30,in=100, looseness=1] (-0.12,2)
(-0.12,2)to[out= -80,in=40, looseness=1] (0.15,1.7)
(0.15,1.7)to[out= -140,in=90, looseness=1.3](0,1.1)
(0,1.1)--(-.2,1.35)
(0,1.1)--(+.2,1.35);
\end{tikzpicture}
\begin{tikzpicture}
\filldraw[color=yellow]
(-6,3)to[out= 40,in=180, looseness=1](-5,3.4)--
(-5,3.4)to[out= 0,in=140, looseness=1](-4,3);
\filldraw[color=yellow]
(-6,3)to[out= 0,in=180, looseness=1](-4,3)--
(-4,3)to[out= -135,in=45, looseness=1](-5,2)--
(-5,2)to[out= 130,in=-40, looseness=1](-6,3);
\filldraw[color=white,thick,scale=1,domain=0: 1.5708,
smooth,variable=\t,shift={(-6,2)},rotate=0]plot({1*sin(\t r)},
{1*cos(\t r)});
\filldraw[color=yellow,thick,scale=1,domain=0: 1.5708,
smooth,variable=\t,shift={(-5,3)},rotate=-90]plot({1*sin(\t r)},
{1*cos(\t r)});
\filldraw[color=yellow]
(-6,1)to[out= -40,in=180, looseness=1](-5,0.6)--
(-5,0.6)to[out= 0,in=-140, looseness=1](-4,1);
\filldraw[color=yellow]
(-6,1)to[out= 0,in=180, looseness=1](-4,1)--
(-4,1)to[out= 150,in=-50, looseness=1](-5,2)--
(-5,2)to[out= -135,in=45, looseness=1](-6,1);
\filldraw[color=yellow,thick,scale=1,domain=0: 1.5708,
smooth,variable=\t,shift={(-5,1)},rotate=90]plot({1*sin(\t r)},
{1*cos(\t r)});
\filldraw[color=white,thick,scale=1,domain=0: 1.5708,
smooth,variable=\t,shift={(-4,2)},rotate=180]plot({1*sin(\t r)},
{1*cos(\t r)});
\draw[color=black,scale=1,domain=-3.141: 3.141,
smooth,variable=\t,shift={(-5,2)},rotate=0]plot({1.42*sin(\t r)},
{1.42*cos(\t r)});
\draw[color=black,thick,scale=1,domain=0: 1.5708,
smooth,variable=\t,shift={(-6,2)},rotate=0]plot({1*sin(\t r)},
{1*cos(\t r)});
\draw[color=black,thick,scale=1,domain=0: 1.5708,
smooth,variable=\t,shift={(-4,2)},rotate=180]plot({1*sin(\t r)},
{1*cos(\t r)});
\draw[color=black,thick,scale=1,domain=0: 1.5708,
smooth,variable=\t,shift={(-5,1)},rotate=90]plot({1*sin(\t r)},
{1*cos(\t r)});
\draw[color=black,thick,scale=1,domain=0: 1.5708,
smooth,variable=\t,shift={(-5,3)},rotate=-90]plot({1*sin(\t r)},
{1*cos(\t r)});
\path[font=\large]
(-6,3.1) node[left] {$P^4$}
(-4,3.1) node[right] {$P^3$}
(-6,.9) node[left] {$P^1$}
(-4,.9) node[right] {$P^2$}
(-5.3,2.47) node[below] {$O$};
\draw[color=black!40!white,scale=0.5,shift={(-10,-2)}]
(-0.05,2.65)to[out= -90,in=150, looseness=1] (0.17,2.3)
(0.17,2.3)to[out= -30,in=100, looseness=1] (-0.12,2)
(-0.12,2)to[out= -80,in=40, looseness=1] (0.15,1.7)
(0.15,1.7)to[out= -140,in=90, looseness=1.3](0,1.1)
(0,1.1)--(-.2,1.35)
(0,1.1)--(+.2,1.35);
\end{tikzpicture}\qquad\qquad\quad
\begin{tikzpicture}
\filldraw[color=yellow]
(-6,3)to[out= 40,in=180, looseness=1](-5,3.4)--
(-5,3.4)to[out= 0,in=140, looseness=1](-4,3);
\filldraw[color=yellow]
(-6,3)to[out= 0,in=180, looseness=1](-4,3)--
(-4,3)to[out= -135,in=45, looseness=1](-5,2)--
(-5,2)to[out= 130,in=-40, looseness=1](-6,3);
\filldraw[color=blue!50!white,thick,scale=1,domain=0: 1.5708,
smooth,variable=\t,shift={(-6,2)},rotate=0]plot({1*sin(\t r)},
{1*cos(\t r)});
\filldraw[color=yellow,thick,scale=1,domain=0: 1.5708,
smooth,variable=\t,shift={(-5,3)},rotate=-90]plot({1*sin(\t r)},
{1*cos(\t r)});
\filldraw[color=blue!50!white]
(-6,3)to[out= -135,in=90, looseness=1](-6.4,2)--
(-6.4,2)to[out= -90,in=135, looseness=1](-6,1);
\filldraw[color=blue!50!white]
(-6,3)to[out= -90,in=90, looseness=1](-6,1)--
(-6,1)to[out= 45,in=-135, looseness=1](-5,2)--
(-5,2)to[out= 135,in=-45, looseness=1](-6,3);
\filldraw[color=red!50!white]
(-6,1)to[out= -40,in=180, looseness=1](-5,0.6)--
(-5,0.6)to[out= 0,in=-140, looseness=1](-4,1);
\filldraw[color=red!50!white]
(-6,1)to[out= 0,in=180, looseness=1](-4,1)--
(-4,1)to[out= 150,in=-50, looseness=1](-5,2)--
(-5,2)to[out= -135,in=45, looseness=1](-6,1);
\filldraw[color=red!50!white,thick,scale=1,domain=0: 1.5708,
smooth,variable=\t,shift={(-5,1)},rotate=90]plot({1*sin(\t r)},
{1*cos(\t r)});
\filldraw[color=white,thick,scale=1,domain=0: 1.5708,
smooth,variable=\t,shift={(-4,2)},rotate=180]plot({1*sin(\t r)},
{1*cos(\t r)});
\draw[color=black,scale=1,domain=-3.141: 3.141,
smooth,variable=\t,shift={(-5,2)},rotate=0]plot({1.42*sin(\t r)},
{1.42*cos(\t r)});
\draw[color=black,thick,scale=1,domain=0: 1.5708,
smooth,variable=\t,shift={(-6,2)},rotate=0]plot({1*sin(\t r)},
{1*cos(\t r)});
\draw[color=black,thick,scale=1,domain=0: 1.5708,
smooth,variable=\t,shift={(-4,2)},rotate=180]plot({1*sin(\t r)},
{1*cos(\t r)});
\draw[color=black,thick,scale=1,domain=0: 1.5708,
smooth,variable=\t,shift={(-5,1)},rotate=90]plot({1*sin(\t r)},
{1*cos(\t r)});
\draw[color=black,thick,scale=1,domain=0: 1.5708,
smooth,variable=\t,shift={(-5,3)},rotate=-90]plot({1*sin(\t r)},
{1*cos(\t r)});
\path[font=\large]
(-6,3.1) node[left] {$P^4$}
(-4,3.1) node[right] {$P^3$}
(-6,.9) node[left] {$P^1$}
(-4,.9) node[right] {$P^2$}
(-5.3,2.47) node[below] {$O$};
\draw[color=black!40!white,scale=0.5,shift={(-13,-2)}]
(-0.05,2.65)to[out= -90,in=150, looseness=1] (0.17,2.3)
(0.17,2.3)to[out= -30,in=100, looseness=1] (-0.12,2)
(-0.12,2)to[out= -80,in=40, looseness=1] (0.15,1.7)
(0.15,1.7)to[out= -140,in=90, looseness=1.3](0,1.1)
(0,1.1)--(-.2,1.35)
(0,1.1)--(+.2,1.35);
\draw[color=black!40!white,scale=0.5,shift={(-7,-2)}]
(-0.05,2.65)to[out= -90,in=150, looseness=1] (0.17,2.3)
(0.17,2.3)to[out= -30,in=100, looseness=1] (-0.12,2)
(-0.12,2)to[out= -80,in=40, looseness=1] (0.15,1.7)
(0.15,1.7)to[out= -140,in=90, looseness=1.3](0,1.1)
(0,1.1)--(-.2,1.35)
(0,1.1)--(+.2,1.35);
\fill[white](-1.45,1) circle (1.4pt);
\end{tikzpicture}

\smallskip

\begin{tikzpicture}
\filldraw[color=blue!50!white,scale=1,domain=-3.141: 3.141,
smooth,variable=\t,shift={(1,0)},rotate=0]plot({1.42*sin(\t r)},
{1.42*cos(\t r)});
\filldraw[color=blue!50!white,scale=0.2,domain=-3.141: 3.141,
smooth,variable=\t,shift={(5,0)},rotate=0]plot({1.42*sin(\t r)},
{1.42*cos(\t r)});
\filldraw[color=white,shift={(6,-2)}]
(-6,3)to[out= 40,in=180, looseness=1](-5,3.4)--
(-5,3.4)to[out= 0,in=140, looseness=1](-4,3);
\filldraw[color=white,shift={(6,-2)}]
(-6,3)to[out= 0,in=180, looseness=1](-4,3)--
(-4,3)to[out= -135,in=45, looseness=1](-5,2)--
(-5,2)to[out= 125,in=-35, looseness=1](-6,3);
\filldraw[color=blue!50!white,thick,scale=1,domain=0: 1.5708,
smooth,variable=\t,shift={(0,0)},rotate=0]plot({1*sin(\t r)},
{1*cos(\t r)});
\filldraw[color=white, shift={(6,-2)}]
(-6,1)to[out= -40,in=180, looseness=1](-5,0.6)--
(-5,0.6)to[out= 0,in=-140, looseness=1](-4,1);
\filldraw[color=white,shift={(6,-2)}]
(-6,1)to[out= 0,in=180, looseness=1](-4,1)--
(-4,1)to[out= 150,in=-50, looseness=1](-5,2)--
(-5,2)to[out= -135,in=45, looseness=1](-6,1);
\filldraw[color=white,thick,scale=1,domain=0: 1.5708,
smooth,variable=\t,shift={(1,-1)},rotate=90]plot({1*sin(\t r)},
{1*cos(\t r)});
\filldraw[color=blue!50!white,thick,scale=1,domain=0: 1.5708,
smooth,variable=\t,shift={(2,0)},rotate=180]plot({1*sin(\t r)},
{1*cos(\t r)});

\filldraw[color=white,thick,scale=1,domain=0: 1.5708,
smooth,variable=\t,shift={(1,1)},rotate=-90]plot({1*sin(\t r)},
{1*cos(\t r)});
\draw[color=black,scale=1,domain=-3.141: 3.141,
smooth,variable=\t,shift={(1,0)},rotate=0]plot({1.42*sin(\t r)},
{1.42*cos(\t r)});
\draw[color=black,thick,scale=1,domain=0: 1.5708,
smooth,variable=\t,shift={(0,0)},rotate=0]plot({1*sin(\t r)},
{1*cos(\t r)});
\draw[color=black,thick,scale=1,domain=0: 1.5708,
smooth,variable=\t,shift={(2,0)},rotate=180]plot({1*sin(\t r)},
{1*cos(\t r)});
\draw[color=black,thick,scale=1,domain=0: 1.5708,
smooth,variable=\t,shift={(1,-1)},rotate=90]plot({1*sin(\t r)},
{1*cos(\t r)});
\draw[color=black,thick,scale=1,domain=0: 1.5708,
smooth,variable=\t,shift={(1,1)},rotate=-90]plot({1*sin(\t r)},
{1*cos(\t r)});
\filldraw[color=blue!50!white, thick, shift={(1,0)}]
(-0.29,-0.045)to[out=20,in=110, looseness=1](0.045,-0.29)--
(0.045,-0.29)--(0.045,-0.29)-- (0,0)--(-0.29,0)
(-0.29,-0.045);
\filldraw[color=blue!50!white, thick, shift={(1,0)},rotate=180]
(-0.29,-0.045)to[out=20,in=110, looseness=1](0.045,-0.29)--
(0.045,-0.29)--(0.045,-0.29)-- (0,0)--(-0.29,0)
(-0.29,-0.045);
\draw[color=black, thick, shift={(1,0)}]
(-0.29,-0.045)to[out=20,in=110, looseness=1](0.045,-0.29);
\draw[color=black, thick, shift={(1,0)},rotate=180]
(-0.29,-0.045)to[out=20,in=110, looseness=1](0.045,-0.29);
\path[font=\large]
(0,1.1) node[left] {$P^4$}
(2,1.1) node[right] {$P^3$}
(0,-1.1) node[left] {$P^1$}
(2,-1.1) node[right] {$P^2$};
\end{tikzpicture}
\begin{tikzpicture}[rotate=90]
\filldraw[color=yellow,scale=1,domain=-3.141: 3.141,
smooth,variable=\t,shift={(1,0)},rotate=0]plot({1.42*sin(\t r)},
{1.42*cos(\t r)});
\filldraw[color=yellow,scale=0.2,domain=-3.141: 3.141,
smooth,variable=\t,shift={(5,0)},rotate=0]plot({1.42*sin(\t r)},
{1.42*cos(\t r)});
\filldraw[color=white,shift={(6,-2)}]
(-6,3)to[out= 40,in=180, looseness=1](-5,3.4)--
(-5,3.4)to[out= 0,in=140, looseness=1](-4,3);
\filldraw[color=white,shift={(6,-2)}]
(-6,3)to[out= 0,in=180, looseness=1](-4,3)--
(-4,3)to[out= -135,in=45, looseness=1](-5,2)--
(-5,2)to[out= 125,in=-35, looseness=1](-6,3);
\filldraw[color=yellow,thick,scale=1,domain=0: 1.5708,
smooth,variable=\t,shift={(0,0)},rotate=0]plot({1*sin(\t r)},
{1*cos(\t r)});
\filldraw[color=white, shift={(6,-2)}]
(-6,1)to[out= -40,in=180, looseness=1](-5,0.6)--
(-5,0.6)to[out= 0,in=-140, looseness=1](-4,1);
\filldraw[color=white,shift={(6,-2)}]
(-6,1)to[out= 0,in=180, looseness=1](-4,1)--
(-4,1)to[out= 150,in=-50, looseness=1](-5,2)--
(-5,2)to[out= -135,in=45, looseness=1](-6,1);
\filldraw[color=white,thick,scale=1,domain=0: 1.5708,
smooth,variable=\t,shift={(1,-1)},rotate=90]plot({1*sin(\t r)},
{1*cos(\t r)});
\filldraw[color=yellow,thick,scale=1,domain=0: 1.5708,
smooth,variable=\t,shift={(2,0)},rotate=180]plot({1*sin(\t r)},
{1*cos(\t r)});

\filldraw[color=white,thick,scale=1,domain=0: 1.5708,
smooth,variable=\t,shift={(1,1)},rotate=-90]plot({1*sin(\t r)},
{1*cos(\t r)});
\draw[color=black,scale=1,domain=-3.141: 3.141,
smooth,variable=\t,shift={(1,0)},rotate=0]plot({1.42*sin(\t r)},
{1.42*cos(\t r)});
\draw[color=black,thick,scale=1,domain=0: 1.5708,
smooth,variable=\t,shift={(0,0)},rotate=0]plot({1*sin(\t r)},
{1*cos(\t r)});
\draw[color=black,thick,scale=1,domain=0: 1.5708,
smooth,variable=\t,shift={(2,0)},rotate=180]plot({1*sin(\t r)},
{1*cos(\t r)});
\draw[color=black,thick,scale=1,domain=0: 1.5708,
smooth,variable=\t,shift={(1,-1)},rotate=90]plot({1*sin(\t r)},
{1*cos(\t r)});
\draw[color=black,thick,scale=1,domain=0: 1.5708,
smooth,variable=\t,shift={(1,1)},rotate=-90]plot({1*sin(\t r)},
{1*cos(\t r)});
\filldraw[color=yellow, thick, shift={(1,0)}]
(-0.29,-0.045)to[out=20,in=110, looseness=1](0.045,-0.29)--
(0.045,-0.29)--(0.045,-0.29)-- (0,0)--(-0.29,0)
(-0.29,-0.045);
\filldraw[color=yellow, thick, shift={(1,0)},rotate=180]
(-0.29,-0.045)to[out=20,in=110, looseness=1](0.045,-0.29)--
(0.045,-0.29)--(0.045,-0.29)-- (0,0)--(-0.29,0)
(-0.29,-0.045);
\draw[color=black, thick, shift={(1,0)}]
(-0.29,-0.045)to[out=20,in=110, looseness=1](0.045,-0.29);
\draw[color=black, thick, shift={(1,0)},rotate=180]
(-0.29,-0.045)to[out=20,in=110, looseness=1](0.045,-0.29);
\path[font=\large,rotate=-90, shift={(-1,1)}]
(0,1.1) node[left] {$P^4$}
(2,1.1) node[right] {$P^3$}
(0,-1.1) node[left] {$P^1$}
(2,-1.1) node[right] {$P^2$};
\end{tikzpicture}
\begin{tikzpicture}
\filldraw[color=blue!50!white,scale=1,domain=-3.141: 3.141,
smooth,variable=\t,shift={(1,0)},rotate=0]plot({1.42*sin(\t r)},
{1.42*cos(\t r)});
\filldraw[color=yellow,shift={(6,-2)}]
(-6,3)to[out= 40,in=180, looseness=1](-5,3.4)--
(-5,3.4)to[out= 0,in=140, looseness=1](-4,3);
\filldraw[color=yellow,shift={(6,-2)}]
(-6,3)to[out= 0,in=180, looseness=1](-4,3)--
(-4,3)to[out= -140,in=30, looseness=1](-5,2)--
(-5,2)to[out= 125,in=-35, looseness=1](-6,3);
\filldraw[color=blue!50!white,thick,scale=1,domain=0: 1.5708,
smooth,variable=\t,shift={(0,0)},rotate=0]plot({1*sin(\t r)},
{1*cos(\t r)});
\filldraw[color=red!50!white, shift={(6,-2)}]
(-6,1)to[out= -40,in=180, looseness=1](-5,0.6)--
(-5,0.6)to[out= 0,in=-140, looseness=1](-4,1);
\filldraw[color=red!50!white,shift={(6,-2)}]
(-6,1)to[out= 0,in=180, looseness=1](-4,1)--
(-4,1)to[out= 150,in=-50, looseness=1](-5,2)--
(-5,2)to[out= -135,in=45, looseness=1](-6,1);
\filldraw[color=red!50!white,thick,scale=1,domain=0: 1.5708,
smooth,variable=\t,shift={(1,-1)},rotate=90]plot({1*sin(\t r)},
{1*cos(\t r)});
\filldraw[color=white,thick,scale=1,domain=0: 1.5708,
smooth,variable=\t,shift={(2,0)},rotate=180]plot({1*sin(\t r)},
{1*cos(\t r)});
\filldraw[color=white,shift={(6,-2)}]
(-4,3)to[out= -90,in=90, looseness=1](-4,1)--
(-4,1)to[out= 135,in=-45, looseness=1](-5,2)--
(-5,2)to[out= 45,in=-135, looseness=1](-4,3);
\filldraw[color=white,shift={(6,-2)}]
(-4,3)to[out= -90,in=90, looseness=1](-4,1)--
(-4,1)to[out= 45,in=-90, looseness=1](-3.6,2)--
(-3.6,2)to[out= 90,in=-40, looseness=1](-4,3);
\filldraw[color=yellow,thick,scale=1,domain=0: 1.5708,
smooth,variable=\t,shift={(1,1)},rotate=-90]plot({1*sin(\t r)},
{1*cos(\t r)});
\draw[color=black,scale=1,domain=-3.141: 3.141,
smooth,variable=\t,shift={(1,0)},rotate=0]plot({1.42*sin(\t r)},
{1.42*cos(\t r)});
\draw[color=black,thick,scale=1,domain=0: 1.5708,
smooth,variable=\t,shift={(0,0)},rotate=0]plot({1*sin(\t r)},
{1*cos(\t r)});
\draw[color=black,thick,scale=1,domain=0: 1.5708,
smooth,variable=\t,shift={(2,0)},rotate=180]plot({1*sin(\t r)},
{1*cos(\t r)});
\draw[color=black,thick,scale=1,domain=0: 1.5708,
smooth,variable=\t,shift={(1,-1)},rotate=90]plot({1*sin(\t r)},
{1*cos(\t r)});
\draw[color=black,thick,scale=1,domain=0: 1.5708,
smooth,variable=\t,shift={(1,1)},rotate=-90]plot({1*sin(\t r)},
{1*cos(\t r)});
\filldraw[color=white,scale=0.2,domain=-3.141: 3.141,
smooth,variable=\t,shift={(5,0)},rotate=0]plot({1.42*sin(\t r)},
{1.42*cos(\t r)});
\filldraw[color=red!50!white, shift={(1,0)}]
(-0.05,-0.05)to[out= 165,in=10, looseness=1](-0.29,-0.045)--
(-0.29,-0.045)to[out= -90,in=180, looseness=1](0.045,-0.29)--
(0.045,-0.29)to[out= 100,in=-75, looseness=1](-0.05,-0.05);
\filldraw[color=yellow, shift={(1,0)},rotate=180]
(-0.05,-0.05)to[out= 165,in=10, looseness=1](-0.29,-0.045)--
(-0.29,-0.045)to[out= -90,in=180, looseness=1](0.045,-0.29)--
(0.045,-0.29)to[out= 100,in=-75, looseness=1](-0.05,-0.05);
\filldraw[color=blue!50!white, shift={(1,0)}]
(-0.05,0.29)to[out= -80,in=110, looseness=1](0.05,0.05)--
(0.05,0.05)--
(-0.05,-0.05)to[out= 165,in=10, looseness=1](-0.29,-0.045)--
(-0.29,-0.045)to[out= 95,in=180, looseness=1](-0.05,0.29);
\draw[color=black, thick, shift={(1,0)}]
(-0.05,-0.05)to[out=45,in=-135, looseness=1](0.05,0.05);
\draw[color=black, thick, shift={(1,0)}]
(-0.05,-0.05)to[out=165,in=10, looseness=1](-0.29,-0.045);
\draw[color=black, thick, shift={(1,0)}]
(-0.05,-0.05)to[out=-75,in=110, looseness=1](0.045,-0.29);
\draw[color=black, thick, shift={(1,0)},rotate=180]
(-0.05,-0.05)to[out=165,in=10, looseness=1](-0.29,-0.045);
\draw[color=black, thick, shift={(1,0)}, rotate=180]
(-0.05,-0.05)to[out=-75,in=110, looseness=1](0.045,-0.29);
\path[font=\large]
(0,1.1) node[left] {$P^4$}
(2,1.1) node[right] {$P^3$}
(0,-1.1) node[left] {$P^1$}
(2,-1.1) node[right] {$P^2$}
(1.4,0.67) node[below] {$O^1$}
(0.7,-0.13) node[below] {$O^2$};
\end{tikzpicture}
\begin{tikzpicture}[rotate=90]
\filldraw[color=red!50!white,scale=1,domain=-3.141: 3.141,
smooth,variable=\t,shift={(1,0)},rotate=0]plot({1.42*sin(\t r)},
{1.42*cos(\t r)});
\filldraw[color=yellow,shift={(6,-2)}]
(-4,3)to[out= -90,in=90, looseness=1](-4,1)--
(-4,1)to[out= 135,in=-45, looseness=1](-5,2)--
(-5,2)to[out= 45,in=-135, looseness=1](-4,3);
\filldraw[color=yellow,shift={(6,-2)}]
(-4,3)to[out= -90,in=90, looseness=1](-4,1)--
(-4,1)to[out= 45,in=-90, looseness=1](-3.6,2)--
(-3.6,2)to[out= 90,in=-40, looseness=1](-4,3);
\filldraw[color=blue!50!white,shift={(6,-2)}]
(-6,3)to[out= 40,in=180, looseness=1](-5,3.4)--
(-5,3.4)to[out= 0,in=140, looseness=1](-4,3);
\filldraw[color=blue!50!white,shift={(6,-2)}]
(-6,3)to[out= 0,in=180, looseness=1](-4,3)--
(-4,3)to[out= -140,in=30, looseness=1](-5,2)--
(-5,2)to[out= 135,in=-45, looseness=1](-6,3);
\filldraw[color=red!50!white,thick,scale=1,domain=0: 1.5708,
smooth,variable=\t,shift={(0,0)},rotate=0]plot({1*sin(\t r)},
{1*cos(\t r)});
\filldraw[color=white, shift={(6,-2)}]
(-6,1)to[out= -40,in=180, looseness=1](-5,0.6)--
(-5,0.6)to[out= 0,in=-140, looseness=1](-4,1);
\filldraw[color=white,shift={(6,-2)}]
(-6,1)to[out= 0,in=180, looseness=1](-4,1)--
(-4,1)to[out= 150,in=-50, looseness=1](-5,2)--
(-5,2)to[out= -135,in=45, looseness=1](-6,1);
\filldraw[color=white,thick,scale=1,domain=0: 1.5708,
smooth,variable=\t,shift={(1,-1)},rotate=90]plot({1*sin(\t r)},
{1*cos(\t r)});
\filldraw[color=yellow,thick,scale=1,domain=0: 1.5708,
smooth,variable=\t,shift={(2,0)},rotate=180]plot({1*sin(\t r)},
{1*cos(\t r)});
\filldraw[color=blue!50!white,thick,scale=1,domain=0: 1.5708,
smooth,variable=\t,shift={(1,1)},rotate=-90]plot({1*sin(\t r)},
{1*cos(\t r)});
\draw[color=black,scale=1,domain=-3.141: 3.141,
smooth,variable=\t,shift={(1,0)},rotate=0]plot({1.42*sin(\t r)},
{1.42*cos(\t r)});
\draw[color=black,thick,scale=1,domain=0: 1.5708,
smooth,variable=\t,shift={(0,0)},rotate=0]plot({1*sin(\t r)},
{1*cos(\t r)});
\draw[color=black,thick,scale=1,domain=0: 1.5708,
smooth,variable=\t,shift={(2,0)},rotate=180]plot({1*sin(\t r)},
{1*cos(\t r)});
\draw[color=black,thick,scale=1,domain=0: 1.5708,
smooth,variable=\t,shift={(1,-1)},rotate=90]plot({1*sin(\t r)},
{1*cos(\t r)});
\draw[color=black,thick,scale=1,domain=0: 1.5708,
smooth,variable=\t,shift={(1,1)},rotate=-90]plot({1*sin(\t r)},
{1*cos(\t r)});
\filldraw[color=white,scale=0.2,domain=-3.141: 3.141,
smooth,variable=\t,shift={(5,0)},rotate=0]plot({1.42*sin(\t r)},
{1.42*cos(\t r)});
\filldraw[color=blue!50!white, shift={(1,0)},rotate=180]
(-0.05,-0.05)to[out= 165,in=10, looseness=1](-0.29,-0.045)--
(-0.29,-0.045)to[out= -90,in=180, looseness=1](0.045,-0.29)--
(0.045,-0.29)to[out= 100,in=-75, looseness=1](-0.05,-0.05);
\filldraw[color=red!50!white, shift={(1,0)}]
(-0.05,0.29)to[out= -80,in=110, looseness=1](0.05,0.05)--
(0.05,0.05)--
(-0.05,-0.05)to[out= 165,in=10, looseness=1](-0.29,-0.045)--
(-0.29,-0.045)to[out= 95,in=180, looseness=1](-0.05,0.29);
\filldraw[color=yellow, shift={(1,0)},rotate=180]
(-0.05,0.29)to[out= -80,in=110, looseness=1](0.05,0.05)--
(0.05,0.05)--
(-0.05,-0.05)to[out= 165,in=10, looseness=1](-0.29,-0.045)--
(-0.29,-0.045)to[out= 95,in=180, looseness=1](-0.05,0.29);
\draw[color=black, thick, shift={(1,0)}]
(-0.05,-0.05)to[out=45,in=-135, looseness=1](0.05,0.05);
\draw[color=black, thick, shift={(1,0)}]
(-0.05,-0.05)to[out=165,in=10, looseness=1](-0.29,-0.045);
\draw[color=black, thick, shift={(1,0)}]
(-0.05,-0.05)to[out=-75,in=110, looseness=1](0.045,-0.29);
\draw[color=black, thick, shift={(1,0)},rotate=180]
(-0.05,-0.05)to[out=165,in=10, looseness=1](-0.29,-0.045);
\draw[color=black, thick, shift={(1,0)}, rotate=180]
(-0.05,-0.05)to[out=-75,in=110, looseness=1](0.045,-0.29);
\path[font=\large,rotate=-90, shift={(-1,1)}]
(0,1.1) node[left] {$P^4$}
(2,1.1) node[right] {$P^3$}
(0,-1.1) node[left] {$P^1$}
(2,-1.1) node[right] {$P^2$};
\path[font=\large]
(1.6,0.5) node[below] {$O^1$}
(1,-0.43) node[below] {$O^2$};
\end{tikzpicture}
\end{center}
\caption{Up: initial datum. Down: expected evolution.}
\label{nonconnected}
\end{figure}

Now let us turn to the second problem, where a network evolves in such a way that some
of its vertices remain irregular but stable.  To explain this, we return to the description 
of self-similar expanders as geodesic Steiner trees with given asymptotic boundary on the
circle at infinity.  For many (and probably all) arrangements of $n$ points on this circle,
there exist Steiner networks where some or all of the interior vertices are not trivalent,
but do satisfy the criticality condition that the sum of the unit tangents of all curves
meeting at each vertex sum to $0$. Steiner networks with vertices satisfying this weaker
condition are no longer minimizing, or even stable, critical points for the length functional;
instead, they are critical points for length.  

As an example, suppose that the arrangement of four points on the circle are invariant with respect 
to reflection across the origin (thus $q_3 = -q_1$ and $q_4 = -q_2$. In this case, the
two diameters of the circle are geodesics for $g$ and the `ray-balancing' condition
that the unit tangents sum to $0$ holds at the origin, where these lines intersect. 

We assert that it is possible to carry out the entirety of our proof of short-time existence
in this paper using such irregular self-similar expanders.  The steps are all clear: we may
construct the approximate solution in precisely the same way, and then derive estimates
for the linear equation precisely as in \S 6 to solve away the rapidly vanishing error term,
just as in \S 7.  The only point that needs to be made is that even at an evolving irregular 
vertex satisfying this balancing condition the Lopatinski-Shapiro condition still holds.






\subsection{Theorem~\ref{main} as a restarting theorem}

Given any initial network $\calN_0$, Theorem~\ref{main} guarantees the existence of a flowout $\calN(t)$, i.e., a solution of
the network flow, for some possibly short time-interval $(0,T)$  which is a regular network for all such positive times.  It is quite 
conceivable, however, that at some later time the network $\calN(t)$ degenerates and becomes irregular at some first
positive time.  There are various ways this might happen. For example, some curve in $\calN(t)$ might shrink to zero length
at time $t_0$ so that the limit of $\calN(t)$ as $r \nearrow t_0$ is an irregular network in which the vertex which is the limit
of that edge has valence four.  Another type of degeneration occurs if $\calN(t)$ has a closed region bounded by fewer than 
six curves.  One can prove that the area of the enclosed region decreases monotonically and shrinks to zero in finite time.
At that time the lengths of some or all of the boundary curves of that region have decreased  to zero, so the limiting
network will have a vertex of valence four or five.  In such limits, the curvature of the curves around this region is
unbounded above, cf.\ \cite{Man}.    We refer to the well-known numerical simulations of network flow that appear
on Ken Brakke's webpage;  these phenomena all can be seen there. 

Suppose that $\calN(t)$ does degenerate as $t \nearrow t_0$, but in such a way that none of the curvatures of
any constituent curve blow up.  This happens if two triple-junctions coalesce into a vertex of valence four,
where the incoming edges meet at angles of either $\pi/3$ or $2\pi/3$.    At least in this case, it is natural
to ask whether the network flow can be continued past this singular time.  We refer to this type of singularity as
a `standard transition' and we emphasize that if we assume that if $\calN(t)$ is a tree, and if none of the curves 
in the limit $\calN(t_0)$ has multiplicity greater than one, then this is the only possible singularity type. 

We now interpret Theorem~\ref{main} as a `restarting' theorem for a suitably defined network flow.

\begin{definition}\label{withsingularities}
Let  $\Omega$ be a smoothly bounded open subset of $\mathbb{R}^2$ and $\mathcal{N}_0$ an admissible initial network.
A time dependent family of networks $\calN(t)$, $0 \leq t < T_{\max}$ is said to be a geometric solution of network flow if it
has fixed outer vertices (this is a Dirichlet boundary condition), converges to the initial network $\calN_0$ as $t \searrow 0$,  and if
$[0,T_{\max})$ decomposes as a finite union of subintervals $[0,a_1) \cup [a_1, a_2) \cup \ldots, [a_\ell, T_{\max})$ so that
for each interval $(a_j, a_{j+1})$ there is a regular network $\calN^{(j)}(t)$ which evolves by network flow, and so that
\[
\lim_{t \nearrow a_j} \calN^{(j-1)}(t) = \lim_{t \searrow a_j} \calN^{(j)}(t) = \calN_{j}
\]
In each of these limits, $\calN^{(j-1)}(t)$ has some arcs whose length decreases to $0$, while $\calN^{(j)}$ has a collection
of new arcs emanating from all vertices in  $\calN_j$ which have valence greater than $3$. 
\end{definition}

\begin{proposition}
Suppose $\mathcal{N}(t)$ is a solution to the network flow in $[0,T)$ and that as $t\to T$ the lengths of one or more curves in
$\calN(t)$ converge to $0$, but that the curvatures of all curves in $\calN(t)$ remain bounded in this limit. 
Then the solution can be extended too a longer time-interval $[0,T')$ for some $T' > T$. 
\end{proposition}
\begin{proof}
By the hypotheses, the limit as  $t\to T$ is a network with bounded curvature and irregular vertices.  It can
thus be used as an initial netwoork for a subsequent flow.  The juxtaposition of this `later' flow with
the original one produces a solution to this flow, in the sense described in Definition~\ref{withsingularities}
on the union of the two time intervals. 
\end{proof}

It is worth stating explicitly that one may control the topological complexity of the network as it passes
through each singular time in terms of the choices of expanders at each irregular vertex.   Considering only
connected expanders without loops, then the total number of enclosed regions is nonincreasing for the
duration of this flow. 

Notice that when a `standard transition' occurs, differently from~\cite{INS}, our Theorem~\ref{main} produces a unique network flow.

\subsection{Other approaches}
In this final section we discuss several approaches that have been used to analyze mean curvature flow in settings 
which allow for nonsmooth initial data.  We do not include here method based on the maximum 
principle, e.g.\ viscosity solutions. 

\medskip

Let $\Sigma_t$ denote a time-dependent smooth submanifold in a Riemannian manifold which evolves with normal 
velocity equal to its mean curvature at any point and time.   This can be formulated as a parabolic PDE for as  long
as $\Sigma_t$ remains smooth and embedded.  Smoothness, in particular, is guaranteed by parabolic regularity
theory so long as certain curvatures do not blow up, or if certain components (or regions) collapse.  However,
at some time $T$ the flow ceases to exist in a classical sense.   It is then natural to ask whether there is any
natural way (or perhaps multiple ways) to continue the flow past this singular time.   

In a different direction, it has long been of interest to define the evolution by mean curvature of more
complicated objects, for example clusters of surfaces, as necessitated by the evolution of grain boundaries.

One of the first successful approaches to this was contained in the thesis of Brakke~\cite{Br}. In it, he
introduced a weak formulation, phrased in the geometric measure theoretic notion of varifolds. This is
a broad enough class to contain highly irregular objects, and mean curvature is defined in a distributional sense.
More precisely, a family of $n$-dimensional varifolds $\{V_t\}_{t\in [0,T)}$ is a solution of the Brakke flow if,  for 
almost every time, $V_t$ is an integral $n$--varifold with locally bounded first variation and generalized curvature
$h$ in $L^2$ such that for any compact set $K$ in $\mathbb{R}^n$ and $t\in (0,T)$ we have $\sup_{\tau\in[0,t) }\Vert V_{\tau}\Vert(K)<\infty$. 
The flow is dictated by the requirement that for every function $\phi \in \calC^0(\RR^{N} \times [0,T);\mathbb{R}^+)$, 
we have
\begin{equation}
\frac{d\,}{dt} \int \phi(x, t)\, \mathrm{d}\Vert V_t\Vert \leq  \int \left( -\phi(x, t) h(V_t,x)^2 + 
\langle \nabla \phi, h(V_t,x)\rangle + (\del_t \phi)(x, t)\right)\, \mathrm{d}\Vert V_t\Vert \,.
\label{Brakke}
\end{equation}

Brakke was able to prove global existence of solutions for any initial data which is an integral varifold.  This flow has
many advantages, but also some deficiencies. For example, this definition allows the instantaneous loss of mass,
so even the empty set is a solution.  Solutions are not unique (see also~\cite{StTo}), and in addition, it is difficult 
to control the multiplicity of the evolving varifolds.   

These difficulties motivate the continuing search for formulations which allow for singular configurations but 
which provide better geometric control of the evolution. 

An important bit of progress was made almost forty years after Brakke's original work by Kim and Tonegawa \cite{KT},
at least for codimension one varifolds. They enhanced the definition of Brakke flow by coupling the time dependent family of 
varifolds with a finite number of time-dependent mutually disjoint open sets. Roughly speaking the varifold is 
the boundary of the union of these sets, and represents the evolving surfaces of a cluster (the open sets are the grains
and the varifold is the grain boundary). The sets evolve continously in time with respect to the Lebesgue measure
and the varifold satisfies~\eqref{Brakke}.  Kim and Tonegawa proved the existence and could also exclude trivial solutions.
Furthermore, when the initial datum is a closed $1$--rectifiable set in $\mathbb{R}^2$ with (locally) finite measure (thus, 
a generalization of the networks we consider here), they obtain regularity~\cite{KT2}: for almost every time
the support of the evolving varifolds consists locally of embedded $W^{2,2}$ curves whose endpoints meet at junctions 
with either $0^\circ$, $60^\circ$ or $120^\circ$ at almost every time. It is expected that these angles are (almost always) $120^\circ$ 
and thus these solutions resemble the network flow. 

On the other hand, as explained in~\cite{Man}, solutions to the network flow are solutions of the Brakke flow,
and there is \emph{equality} in \eqref{Brakke} when the network is \emph{regular}. 

An important point is that none of these formulations describe what happens on sets of time of measure
zero, so cannot give precise information on the types of instantaneous jumps we study in this paper. 

\medskip

We next recall the theory of minimizing movements; this is a time discrete approximation to gradient flow,
which was considered in the setting of geometric flows by a number of authors, including De Giorgi, 
Almgren, Taylor and Wang, Luckhaus and Sturzenhecker.   This method involves solving a sequence of 
variational problems at every one of the discrete time steps.  To pass from one time-step to the next, one
minimizes the sum of the area functional and an additional term which is added to make sure that we do 
not move too far in a single time step. 
There is a recent result by Fischer, Hensel, Laux and Simone~\cite{FiHeLaSi} concerning BV solutions
to planar multiphase mean curvature flow. They prove that in the absence of topological changes, 
solutions are unique. 

\medskip

As for methods which rely specifically on PDE, the theory started with the paper of Bronsard and Reitich \cite{BrRe},
who proved local existence and uniqueness provided the initial network is regular and of class $\calC^{2+\alpha}$,
with the sum of the curvatures at each triple junctions equal to zero.  This last condition is tantamount to the 
first compatibility condition. Their proof is based on a standard linearization procedure and a fixed point argument.
If the initial datum is a regular network of class $\calC^2$, but without any restriction on the curvatures at the junctions,
then \cite{Man} proves existence but not uniqueness.  Uniqueness was proved in \cite{GMP} for initial
networks in $W^{2-\nicefrac{2}{p}}_p$ for any $p\in (3,\infty)$. A fortiori, this gives uniqueness for regular networks 
of class $\calC^2$. 
Noting that $W^{2-\nicefrac{2}{p}}_p \hookrightarrow \calC^{1+\alpha}$, we see that the space 
$W_p^1\left((0,T);L_p((0,1);\mathbb{R}^2)\right)\cap L_p\left((0,T); W_p^2\left((0,1);\mathbb{R}^2\right)\right)$
where solutions lie embeds into the space of bounded, uniformly continuous functions from $[0,T]$ to $C^{1+\alpha}$.
This degree of regularity is needed in order to get the contraction estimates in the fixed point argument
and to be able to invoke classical parabolic theory for the associated linear system.   

None of these approaches allows one to treat initial networks which contain triple junctions where the curves
do not meet at $120^\circ$, or which contain vertices with valence greater than three.  Clearly, in either of
these cases one needs to work in a slightly broader class which does not impose continuity as $t \searrow 0$.

\medskip

As already mention in the introduction, there is already a proof of short-time existence due to lmanen, Neves and Schulze~\cite{INS} 
if the initial network is an irregular tree with bounded curvature.  We describe this in a bit more detail.
As we do here, the main issue is to desingularize the higher valence vertices using expanding solitons.  They glue
a truncation of a connected regular expanding solition in place of a higher valence vertex, in a ball of radius comparable
to $\sqrt{t}$ to obtain an approximation of the flow.  They then argue using an expander monotonicity formula which 
implies that self-similar expanding flows are ``dynamically stable''.  The existence of the approximating flow for a positive 
time  is guaranteed by this monotonicity, which shows that this
approximating flows stays close to the self-similar expanding flows in a suitable neighborhood. The authors establish
that the flow converges to the initial datum only in the sense of varifolds.  This makes it difficult to address questions
of uniqueness or multiplicity, as is possible in our approach.

This brief tour of various methods hopefully makes clear that the methods espoused in this paper do provide some useful
additional information regarding multiplicity of solutions that may not be so easy to obtain by other ways.

\bibliographystyle{abbrv}
\bibliography{stenf}

\end{document}